\documentclass[11pt]{amsart}
\usepackage{amssymb}
\usepackage{amsmath}
\usepackage{amsrefs}
\usepackage{mathrsfs}
\usepackage{faktor}
\usepackage{tikz}
\usepackage{thmtools}
\usepackage{mathrsfs,amsmath, tikz,verbatim, wasysym,xcolor, amsfonts, amssymb, amsthm, graphicx, hyperref, enumerate, xfrac, xspace, mathtools}
\usepackage{thm-restate}
\usetikzlibrary{arrows}
\usetikzlibrary{shapes.geometric}
\usetikzlibrary{matrix}
\usetikzlibrary{decorations.pathreplacing, decorations.markings}
\usepackage[utf8]{inputenc}
\usepackage[top=1in, bottom=1in, right=1in, left=1in]{geometry}
\usepackage{hyperref}
\usepackage{graphicx}
\usepackage{xcolor}
\usepackage{tikz-cd}

\usepackage[capitalise]{cleveref}
\title{Counting automorphic orbits in finitely generated groups}
\author{Luna Elliott, Alex Evetts and Alex Levine}

\newtheorem{theorem}{Theorem}[section]

\newtheorem{lem}[theorem]{Lemma}
\newtheorem{lemma}[theorem]{Lemma}
\newtheorem{proposition}[theorem]{Proposition}

\newtheorem{cor}[theorem]{Corollary}

\newtheorem{question}[theorem]{Question}

\theoremstyle{definition}
\newtheorem{dfn}[theorem]{Definition}
\newtheorem{defn}[theorem]{Definition}
\newtheorem{ex}[theorem]{Example}
\newtheorem{rmk}[theorem]{Remark}

\newtheorem{ntn}[theorem]{Notation}

\numberwithin{theorem}{section}
\allowdisplaybreaks

\DeclareMathOperator{\sgn}{sgn}
\DeclareMathOperator{\id}{id}

\DeclareMathOperator{\im}{im}

\DeclareMathOperator{\Aut}{Aut}
\DeclareMathOperator{\Inn}{Inn}
\DeclareMathOperator{\Out}{Out}
\DeclareMathOperator{\End}{End}

\DeclareMathOperator{\GL}{GL}
\DeclareMathOperator{\SL}{SL}

\DeclareMathOperator{\Sym}{Sym}

\newcommand{\Z}{\mathbb{Z}}
\newcommand{\N}{\mathbb{N}}
\newcommand{\Q}{\mathbb{Q}}

\newcommand{\R}{\mathbb{R}}

\newcommand{\cO}{\mathcal{O}}

\usepackage{ifthen}
\newboolean{rightactions}
\setboolean{rightactions}{true}

\newcommand{\row}{%
\ifthenelse{\boolean{rightactions}}{row}{column}%
}
\newcommand{\column}{%
\ifthenelse{\boolean{rightactions}}{column}{row}%
}

\makeatletter
\newcommand{\comp}[1]{\ifthenelse{\boolean{rightactions}}{#1\@ifnextchar\bgroup{\compright@i}{}}{\compleft@i{#1}}}
\newcommand{\compright@i}[1]{#1\@ifnextchar\bgroup{\compright@i}{}}
\newcommand{\compleft@i}[2][]{%
  \@ifnextchar\bgroup
    {\compleft@i[{#2#1}]}
    {#2#1}%
}
\makeatother

\newcommand{\conjugate}[2]{\comp{#2^{-1}}{#1}{#2}}

\makeatletter
\newcommand{\compweird}[1]{\compweird@i{#1}}
\newcommand{\compweird@i}[2][]{%
  \@ifnextchar\bgroup
    {\compweird@i[{#2#1}]}
    {#2#1}%
}
\makeatother
\newcommand{\act}[2]{\comp{(#1)}{#2}}
\newcommand{\actweird}[2]{\compweird{(#1)}{#2}}

\newcommand{\makeset}[2]{\left\lbrace #1  : 
 \begin{tabular}{@{}l@{}}
   #2
  \end{tabular}
  \right\rbrace}

\newcommand{\makepres}[2]{\left\langle #1  \;\middle|\; 
 \begin{tabular}{@{}l@{}}
   #2
  \end{tabular}
  \right\rangle}

\subjclass[2020]{20F69, 20F65, 20F28}
\keywords{automorphic orbits, conjugacy growth}
\begin{document}
\begin{abstract}
    We study an analogue of the conjugacy growth function in finitely generated
    groups: the automorphic growth function. This counts the number of
    automorphic orbits that intersect the ball of radius \(n\) in the
    group. We show that this is not a commensurability invariant, by giving
    virtually abelian counterexamples. We classify the automorphic growth rate of all
    virtually abelian groups of rank at most \(2\), the Heisenberg group, finite
    rank free groups and Thompson's groups \(T\) and \(V\). This last computation
    allows to conclude that \(T\) and \(V\) have exponential conjugacy
    growth.
\end{abstract}
\maketitle
\section{Introduction}
Given a finitely generated group \(G\), one can study the \emph{conjugacy growth
function}; that is the function that maps \(n\) to the number of conjugacy
classes in \(G\) which intersect with the ball of radius \(n\). This is
asymptotically independent of the choice of finite generating set, and has been
computed for various classes of groups \cites{HeisenbergConjugacy, GubaSapir,
Osin, HullOsin, Babenko, GekhtmanYang, Greenfeld, BreuillarddeCornulier,
ConjLinear, CiobanuEvettsHo, Fink}. A natural analogue of conjugacy growth is to
instead count the number of automorphic orbits in \(G\) that intersect with the
\(n\)-ball - the \emph{automorphic growth function} \(\alpha_{G}\). As
with conjugacy growth, the growth rate is independent of the choice of finite
generating set. Unlike conjugacy, automorphism groups are defined for any
algebraic object, thus allowing automorphic growth to be generalised to other
areas within algebra, such as monoids. As conjugacy growth provides a lower
bound for standard growth, automorphic growth bounds conjugacy growth below.
Moreover, if \(G\) is characteristic in any overgroup \(H\), then the
automorphic growth of \(G\) is still a lower bound for the conjugacy growth of
\(H\), since \(H\) acts on \(G\) by conjugation.

When studying abelian groups, conjugacy growth and standard growth trivially
coincide, as the conjugacy classes are singletons. Full automorphism groups
of abelian groups can be fairly involved, and thus even \(\Z^m\) provides a
counterexample to the idea that conjugacy growth and automorphic growth may
coincide; the automorphic growth of \(\Z^m\) is linear for all \(m \geq 1\)
(\cref{thm:freeabelian}). Despite their conjugacy growth rates coinciding with their
standard growth rates \cite{HeisenbergConjugacy}, virtually abelian groups provide
counterexamples to commensurability invariance. With conjugacy growth, Hull and
Osin \cite{HullOsin} provide an example of a group with exponential conjugacy growth
that has a finite-index subgroup with constant conjugacy growth (and hence
automorphic growth), although the groups in question are not finitely presented,
and it is not immediately clear what the automorphic growth of the larger group
is. In the class of virtually abelian groups of fixed rank, there are groups
with every non-constant polynomial growth rate up to the upper bound of standard
growth:

\begin{theorem}[\cref{allpossibleVA}]
    For every \(m , k \in \mathbb{N}\setminus\{0\}\) with \(m \geq k\), there is a
    virtually abelian group with rank \(m\) and automorphic growth
    rate that is polynomial of degree \(k\).
\end{theorem}

In addition, we provide a full classification of the automorphic growth of virtually abelian
groups of rank up to \(2\).

\begin{theorem}[\cref{VAmain1}]
    Suppose that $G$ is virtually abelian with a finite index normal subgroup $A\cong \Z^m$. 
    \begin{enumerate}
        \item If \(m=0\) then $\alpha_G$ is constant;
        \item If \(m=1\) then $\alpha_G$ is linear;
        \item If \(m=2\) then $\alpha_G$ is linear or quadratic. In this case, the growth of \(G\) is linear if and only if there is a homomorphism $\operatorname{sgn}
      \colon G\to \{1,-1\}$ such that for all $g\in G$ and $x \in
      A$ we have $\conjugate{x}{g}= x^{\actweird{g}{\operatorname{sgn}}}$.
    \end{enumerate}
\end{theorem}
The next example we consider is the Heisenberg group \(H(\Z)\). In the case of
conjugacy growth, this is an interesting example: the growth rate is \(n^2 \log
n\) \cite{Babenko} (see also \cite{HeisenbergConjugacy}); strictly smaller than
the standard growth of \(H(\Z)\), and not even a (truly) polynomial growth rate
- something that never occurs with standard growth due to Gromov's Theorem
\cite{bass, Gromov}. It also differs from its cousins the higher Heisenberg groups
\cite{HeisenbergConjugacy}, in this case, as these all have truly polynomial
growth rates. Unlike conjugacy growth however, the automorphic growth is again
exactly a polynomial rate, providing the first example where all three growth
rates are distinct.
    \begin{theorem}[\cref{thm:Heis}]
        The automorphic growth of the Heisenberg group is quadratic.
    \end{theorem}

We next move to groups with exponential standard growth. It is conjectured that
for `ordinary' groups (such as amenable groups), exponential standard growth
implies exponential conjugacy growth \cite{GubaSapir}. Osin proved that this is
not true for all groups with exponential growth rate, as one exists that has two
conjugacy classes \cite{Osin}. Naturally, this will also have bounded
automorphic growth. This group is not finitely presented, and no finitely
presented such example has been found. In fact, for the cases of finitely
presented groups with exponential standard growth we consider, the automorphic
growth rate is exponential. The first class of groups we consider with
exponential standard growth is the class of finite rank free groups. We use
Whitehead's algorithm \cite{whitehead_algm} to show that this is also
exponential.

  \begin{theorem}[\cref{main_free}]
     Non-abelian free groups of finite rank have exponential automorphic growth.
  \end{theorem}

We conclude with groups where the conjugacy growth has yet to be studied -
Thompson's groups \(T\) and \(V\). We chose these groups, in part, due to their
very well-studied automorphism groups. Due to its finite outer automorphism
group \cite{autft}, computing the automorphic growth rate for \(T\) is not
difficult. The case for \(V\) is much more involved. 

\begin{theorem}[\cref{Tautgrowth} and \cref{thm:V}]
    The automorphic growth rates of Thompson's groups \(T\) and \(V\) are exponential.
\end{theorem}

It is worth mentioning that growth relating to automorphisms has been studied before. Some studies fix an automorphism \(\phi\) of a group \(G\), and then study the growth of
the sequence obtained by iteratively applying \(\phi\) to a conjugacy class of \(G\) \cites{Fioravanti26, Levitt09, Coulon22}. Lee has worked on counting the number
of minimal length cyclic words in an automorphic orbit of a free group
\cites{Lee1, Lee2}. 

We cover the preliminaries and basic definitions of automorphic growth in
\cref{sec:basic}. \cref{sec:found} includes various lemmas that involve studying
how automorphic growth behaves when moving to subgroups, which are of particular
importance when studying virtually abelian groups in \cref{sec:VA}. We compute
the growth of the Heisenberg group in \cref{sec:Heis} and of free groups in
\cref{sec:free}. We conclude with Thompson's groups \(T\) and \(V\) in \cref{sec:V}.

\section{Basic definitions}
\label{sec:basic}

We begin with notational conventions used throughout.

\begin{ntn}
    Functions (with the exception of operators) will be written to the right of their
    arguments, to allow composition to be read from left to right; that is, we write
    \((x)f\), rather than \(f(x)\). Growth functions will be considered operators here,
    and thus written to the left of their arguments, as with the majority of the literature
    on growth of groups.

    For a set \(\Sigma\), we denote the free monoid generated by \(\Sigma\) with 
    \(\Sigma^\ast\), and we write \(\Sigma^n\) to denote all words in
    \(\Sigma^\ast\) of length \(n\). The length of a word \(w \in \Sigma^\ast\) is
    denoted \(|w|\). We use \(\varepsilon\) to denote the empty word.
    We consider \(0 \in \N\).
\end{ntn}

    We now define the automorphic growth of a group with respect to a given finite
    generating set.
  \begin{dfn}
    Let \(G\) be a group, \(\Sigma\) be a finite generating set
    for \(G\), and $\Omega\leq\Aut(G)$. For \(g\in G\), we write \(|g|_\Sigma
    \coloneqq \min \makeset{n\in \N}{\(g\in (\Sigma\cup \Sigma^{-1})^{n}\)}\) (when \(\Sigma\) is
    implicit from the context we will simply write \(|g|\)). Let \(\act{g}{\Omega}\)
    be an orbit of \(G\) under the action of \(\Omega\), and let \(G/\Omega\) be the set of all such orbits. The \textit{length} of
    \(\act{g}{\Omega
    }\) with respect to \(\Sigma\), denoted \(|\act{g}{\Omega}|_\Sigma\)
    (or simply \(|\act{g}{\Omega}|\)), is defined by
    \[
      |\act{g}{\Omega}| = \min \{|h| :  h \in\act{g}{\Omega}\}.
    \]
    The \textit{(cumulative) $\Omega$-automorphic growth function} of \((G,
    \Sigma)\) is defined to be the function from \(\N\) to itself that maps
    \(n\) to the number of $\Omega$-orbits of \(G\) with length  at most \(n\)
    with respect to \(\Sigma\). We use \(\alpha_{G^\Omega, \Sigma}\) to denote
    the $\Omega$-automorphic growth function. We will be most frequently
    interested in the case where $\Omega=\actweird{G}{\Aut}$, which we simply
    call \emph{automorphic growth}, and denote by $\alpha_{G,\Sigma}$.

    Write \(\pi \colon (\Sigma\cup \Sigma^{-1})^\ast \to G\) for the natural
    map. A word \(w \in \Sigma^\ast\) is called
    \textit{\(\Omega\)-automorphically minimal} in \(G\), if $w$ is a
    minimal-length representative of the orbit $
    \act{\act{w}{\pi}}{\Omega}$. If \(\Omega = \actweird{G}{\Aut}\) we just
    say \emph{automorphically
    minimal}.
  \end{dfn}
  When talking about the growth rate of a group, as is standard with other forms
  of growth (including standard growth), we only consider this up to
  the following equivalence. 

  \begin{dfn}\label{growtheq}
      We define a preorder \(\preccurlyeq\) on the set \(\End^+(\N, \leq)\) of
      non-decreasing non-zero functions from \(\N\) to itself by \(f \preccurlyeq g\)
      whenever there exists $\lambda\in \N  \setminus  \{0\}$ such that for
      every $n\in \N$, we have $\act{n}{f} \leq \lambda g (\lambda n + \lambda) +
      \lambda$. 
      We write $\sim $ for the equivalence relation given by $f\sim g$
      if both $f\preccurlyeq g$ and $g\preccurlyeq f$. We then view
      \(\preccurlyeq\) as a partial order on the \(\sim\) classes of
      \(\End^+(\N,\leq)\) in the usual fashion and we use \(\sim\) and \(=\)
      interchangeably for a pair of these classes.
  \end{dfn}
  In light of \cref{well-defined}, the $\sim$-equivalence class of the
  $\Omega$-automorphic growth of a finitely generated group $G$ does not depend
  on the choice of finite generating set. We thus define $\alpha_{G^\Omega}$ to
  be the $\sim$-equivalence class (and write $\alpha_{G}$ to mean
  $\alpha_{G^{\Aut(G)}}$). Note that the usual conjugacy growth is just
  $\alpha_{G^\Omega}$ where \(\Omega\) is the inner automorphism group of \(G\).
  We will often add or multiply \(\sim\)-equivalence classes of functions.
  This is all well-defined, as \(\preccurlyeq\) preserves addition and multiplication.
  We refer the reader to \cite{Meier}*{Section 11} for more details.

\begin{rmk}\label{growth-comparison-1}
    Note that if \(G\) is a group and \(\Omega_1\leq \Omega_2\leq \actweird{G}{\Aut}\), then every orbit of \(\Omega_2\) is a union of orbits of \(\Omega_1\). 
In particular, smaller groups of automorphisms have more (smaller) orbits than larger groups of automorphisms. 
This implies that \(\alpha_{G^{\Omega_2}}\preccurlyeq \alpha_{G^{\Omega_1}}\) and in particular \(\alpha_{G}\preccurlyeq\alpha_{G^{\Omega}}\).
This also implies that the automorphic growth is a lower bound for the conjugacy growth. 
\end{rmk}

In order to study the automorphic growth of groups, we will occasionally require a
generalisation of the notion of automorphic growth, where we instead we act on a locally
finite metric space rather than a group.
\begin{dfn}
    \label{action-growth-obj}
     We say that a metric space is \emph{locally finite} if all subsets of
     finite diameter are finite. We define an \emph{action growth object} to be
     a pair $(M,\Omega)$ where $M$ is a locally finite metric space and $\Omega$
     is a group acting on the underlying set of $M$. If $(M,\Omega)$ is an
     action growth object then we define the growth $\alpha_{(M,\Omega)}$ of the
     object to be the $\sim$-class (see \cref{growtheq}) of the map sending $n$
     to the cardinality of the set
    \[\makeset{J \in M/\Omega}{\(J\) contains a element of distance at most $n$ from $x_0$},\]
    where $x_0\in M$ is fixed (it is routine to verify that the class of this function is independent of the choice of $x_0$).
\end{dfn}
In particular, if $G$ is a finitely generated group equipped with the word
metric then $(G,\actweird{G}{\Aut})$ is an action growth object and its growth is the
automorphic growth of $G$. Note that this action does not necessarily preserve
the structure of the metric space \(M\) in any reasonable sense.
\begin{rmk}\label{actobjiso}
    If $(M_1,\Omega_1), (M_2,\Omega_2)$ are action growth objects and $\phi \colon M_1\to M_2$ is a bijective quasi-isometry which preserves whether or not points are in the same orbit, then $\alpha_{(M_1,\Omega_1)}=\alpha_{(M_2,\Omega_2)}$.
\end{rmk}
 Applying \cref{actobjiso} using the identity map from a group with one
    generating set to the same group with another, we obtain the following:
\begin{rmk}\label{well-defined}
    The automorphic growth of a finitely generated group is well-defined independent of generating set.
\end{rmk}

When studying any kind of growth of subgroups (or even subsets) of groups, one
must deal with the issue that the word metrics of the subgroup and the larger group
may not be quasi-isometric, and thus there will be multiple potential growth rates to consider.
To deal with this issue, we define the relative automorphic growth, which is the automorphic
growth of a subset of a group, but using the inherited metric from the group itself.

\begin{dfn}
    Let \(G\) be a finitely generated group. Let \(U \subseteq G\) and \(\Omega\leq \actweird{U}{\Sym}\).
    The \emph{relative \(\Omega\)-automorphic growth}  \(\alpha_{U^\Omega\subseteq G}\) of
    \(U\) in \(G\) with respect to $\Omega$ is defined to be the growth of the
    action growth object \((U,\Omega)\) using the inherited metric from a word metric for \(G\). If $\Omega$ is the group of permutations of $U$ which extend to automorphisms of $G$ then we omit $\Omega$ and simply write \(\alpha_{U\subseteq G}\).
\end{dfn}

Naturally, one case where one does not need the notion of relative growth is for undistorted
subgroups, as in these cases the two relevant metrics will be quasi-isometric, and thus
growth rates will not depend on which metric is used.

\begin{dfn}
    We say that a finitely generated subgroup of a finitely generated group is \emph{undistorted} if the inclusion map is a quasi-isometry to its image with the inherited metric.
\end{dfn}

\section{Foundational lemmas}\label{sec:found}
We might consider this section to be along the lines of a `closure properties'
section, although in the case of automorphic growth, nothing so strong can be
said. As we see in \cref{sec:VA}, automorphic growth rate is not a
commensurability invariant, and it is well-known that even characteristic
subgroup's automorphism group need not share many properties with the original
group's automorphism group. Nonetheless, in various restricted cases, some
things can be said, if one is willing to move from generic automorphic growth to
\(\Omega\)-automorphic growth for a relevant group \(\Omega\) of automorphisms.
Naturally, when studying virtually abelian groups in \cref{sec:VA}, these
results are particularly important, since in this class of groups, much of
the structure comes from a subgroup.

We begin with two observations that follow immediately from \cref{actobjiso}; in one case
we vary the group of automorphisms used, and in the other we pass to an undistorted subgroup.
\begin{lem}\label{conjugate-aut-groups}
    Let \(M\) be a locally finite metric space and \(\Omega\leq \Sym(M)\).
    If \(\theta \colon M\to M\) is a bijective quasi-isometry and \(\Lambda=\conjugate{ \Omega }{\theta}\)
    then
    \(
    \alpha_{(M,\Omega)} =  \alpha_{(M,\Lambda)}.
    \)
\end{lem}
\begin{proof}
See \cref{actobjiso}.
\end{proof}

\begin{lem}\label{relativeundistorted}
    Let $H$ be an undistorted subgroup of $G$, and $\Omega\leq \actweird{G}{\Aut}$ be a group fixing \(H\) setwise, $P$ the subgroup of $\actweird{H}{\Aut}$ consisting of the restriction of $\Omega$ to $H$. Then the relative $P$-automorphic growth of $H$ in $G$ is equivalent to the $P$-automorphic growth of $H$, that is $\alpha_{H^P\subseteq G}\sim\alpha_{H^P}$
\end{lem}
\begin{proof}
    See \cref{actobjiso}.
\end{proof}
  Unlike standard ang conjugacy growth, there is no direct relationship between the
  automorphic growth of a group and its quotients, as the automorphism
  group can change drastically when quotienting. However, if the
  normal subgroup in question is characteristic (that is, fixed setwise
  under automorphisms), then the expected inequality does hold.

  \begin{lem}\label{quotient-growth}
  	If $Q$ is a quotient of $G$ by a characteristic subgroup $N$, then
  	$\alpha_Q\preccurlyeq\alpha_G$.
  \end{lem}

  \begin{proof}
    For each \(\theta \in \actweird{G}{\Aut}\) define \(\overline{\theta} \in \actweird{Q}{\Aut}\) by
    \(\act{g / N}{\overline{\theta}} = N \cdot \act{g}{\theta}\) for all \(g\in G\). We
    show that \(\overline{\theta}\) is indeed an automorphism. For each
    $\gamma\in G/N$, choose a lift $\tilde{\gamma}\in G$. Then define
    $\act{\gamma}{\overline{\theta}} = N\cdot \act{\tilde{\gamma}}{\theta}$. Since $N$ is
    characteristic, \(\overline{\theta}\) is a well-defined homomorphism; that
    is, $\act{Ng}{\overline{\theta}} = N \cdot \act{g}{\theta}$. Now suppose that $\act{\gamma}{
    \comp{\overline{\theta}}} = \act{\delta}{\overline{\theta}}$ for some
    $\gamma,\delta\in G/N$. Then $N \cdot \act{\tilde{\gamma}}{ \theta} = N \cdot
    \act{\tilde{\delta}}{\theta}$, which implies $\act{\tilde{\gamma}}{\theta} \cdot
    (\act{\tilde{\delta}}{\theta})^{-1} \in N$ and so \(\tilde{\gamma}
    \tilde{\delta}^{-1} \in \act{N}{\theta} = N\) and $\gamma=\delta$. Thus
    $\overline{\theta}$ is injective.
    For surjectivity, let $\delta\in G/N$. Since $\theta$ is surjective, there
    exists $g\in G$ with $\act{g}{\theta} = \tilde{\delta}$. We have
    $\act{Ng}{\overline{\theta}} = N \cdot \act{g}{\theta} = N {\tilde{\delta}}=\delta$, and
    hence $\overline{\theta}$ is indeed surjective.
  
  	Choose a finite, generating set $\Sigma$ for $G$, and let
  	$\overline{\Sigma}$ denote the image of $\Sigma$ in $Q$ which is a finite
  	generating set for $Q$. Suppose $q\in Q$ has length $n$ with respect to
  	$\overline{\Sigma}$. Then there exist $\sigma_1,\sigma_2,\ldots,\sigma_n\in
  	\Sigma \cup \Sigma^{-1}$ such that
  	$q=\overline{\sigma}_1\overline{\sigma}_2\cdots\overline{\sigma}_n$. Then
  	$g \coloneqq \sigma_1\sigma_2\cdots\sigma_n\in G$ is a lift of $q$, and so
  	$|g|_{\Sigma}\leq n$.

  	Now note that if $q,q'\in Q$ are in distinct automorphic orbits of \(Q\), we
  	have that they cannot be in the same orbit using automorphisms of the form
  	\(\overline{\theta}\) for some \(\theta \in \actweird{G}{\Aut}\). Thus any pair of lifts
  	of $q,q'$ are in distinct automorphic orbits in $G$. So any set of geodesic 
  	representatives for automorphic orbits in $Q$, each of length at most $n$
  	lifts to a set of representatives for distinct automorphic orbits of $G$,
  	each of length at most $n$. Therefore \(\alpha_{Q}\preccurlyeq\alpha_{G}\),
  	proving the claim.
  \end{proof}

 We next relate the automorphic growth of a subgroup to the automorphic growth
 of a characteristic subgroup (or at least a subgroup preserved by the relevant
 group of automorphisms). 

  \begin{lem}\label{lem:charsbgrp}
    Let $H$ be a subgroup of $G$, \(\Omega\leq \actweird{G}{\Aut}\) be such that \(\act{H}{\phi}
    = H\) for all \(\phi \in \Omega\), and let $P=\{\phi|_H :
    \phi\in\Omega\}\leq\Aut(H)$ be the group of automorphisms of $H$ arising
    from restricting automorphisms in \(\Omega\). Then \(\alpha_{H^P\subseteq G}
    \preccurlyeq \alpha_{G^\Omega}\).
  \end{lem}
  \begin{proof}
    Let $h_1,h_2\in H$. Then there exists $\phi_H\in P$ such that $\act{h_1}{\phi_H} =
    h_2$ if and only if there exists \(\phi_G\in \Omega\) such that $\act{h_1}{ \phi_G} = h_2$.
    Thus $\alpha_{H^P\subseteq G}(n)$ is precisely the
    number of $\Omega$-orbits of $G$ of length at most $n$ which are contained
    in $H$, and thus is at most the total number of $\Omega$-orbits of length at
    most $n$, which is the claimed result.
  \end{proof}

  An unusual construction that doesn't require characteristic subgroups to compute a lower
  bound for automorphic growth is direct products. This follows easily since the
  direct product of the two relevant automorphism groups embeds into the automorphism
  group of the direct product.

  \begin{lem}
  \label{dir-prod}
    Let \(G = H \times K\) where
    \(H\) and \(K\) are finitely generated groups. Then \(\alpha_G
    \preccurlyeq \alpha_H \cdot \alpha_K\).
    If, in addition, \(H\) and \(K\) are characteristic in \(G\), then
    \(\alpha_G \sim \alpha_H \cdot \alpha_K\).
  \end{lem}
  \begin{proof}
    We first consider the bound for generic \(H\) and \(K\).
    For every \(\phi \in \actweird{H}{\Aut}\) and
    \(\psi \in \actweird{K}{\Aut}\), define
    \begin{align*}
      \theta_{\phi, \psi} \colon G & \to G \\
      hk & \mapsto \act{h}{\phi}\cdot \act{k}{\psi} .
    \end{align*}
    We will show that \(\theta_{\phi ,\psi} \in \actweird{G}{\Aut}\). Let \(h_1 k_1, \ h_2
    k_2 \in G\). Then
    \begin{align*}
      \act{h_1 k_1 h_2 k_2}{\theta_{\phi, \psi}}  & = \act{h_1 h_2 k_1 k_2}{\theta_{\phi, \psi}}\\
      & = \act{h_1}{\phi} \cdot \act{h_2}{\phi} \cdot \act{k_1}{\psi} \cdot  \act{k_2}{\psi} \\
      & = \act{h_1}{\phi} \cdot  \act{k_1}{\psi} \cdot \act{h_2}{\phi} \cdot  \act{k_2}{\psi} \\
      & =\act{h_1 k_1}{\theta_{\phi, \psi}}\cdot \act{h_2 k_2} {\theta_{\phi, \psi}}.
    \end{align*}
    So \(\theta_{\phi, \psi}\) is a homomorphism. As \(\phi\) and \(\psi\) are
    both surjective, it follows that \(\theta_{\phi, \psi}\) is also surjective.
    Let \(hk \in \ker (\theta_{\phi, \psi})\). Then \(\act{h}{\phi} \cdot \act{k}{\psi}  = 1\), and
    so \(H \ni (\act{h}{\phi} ) = ( \act{k}{\psi})^{-1}\). It follows that \( \act{k}{\psi}  = 1\) and
    so \(k = 1\). Thus \(\act{h}{\phi} = 1\) and \(h = 1\). We can conclude that
    \(\theta_{\phi, \psi}\) is injective and thus an automorphism.
    
    It follows that if \(h_1 k_1\) and \(h_2 k_2\) are not in the same automorphic orbit of \(G\),
    then \(h_1\) is not in the same automorphic orbit of \(H\) to \(h_2\) or \(k_1\) is not in the same automorphic orbit of \(K\) to \(k_2\).
    
     Let \(\Sigma_H, \Sigma_K\) be finite generating sets for \(H\) and \(K\) respectively. 
    Thus
    \begin{align*}
        \actweird{n}{\alpha_{H, \Sigma_H}}
        \cdot \actweird{n}{\alpha_{K, \Sigma_K}}
        & = \left| \makeset{\cO_H \times \cO_K}{\(\cO_H \in H/\Aut(H)\), \(|\cO_H|_{\Sigma_H} \leq n\),\\ \(\cO_K \in K/\Aut(K)\), \(|\cO_K|_{\Sigma_K} \leq n\)}\right| \\
        & \geq \left|\{\cO \in G/\Aut(G)  : |\cO|_{\Sigma_H \cup \Sigma_K} \leq 2n\}\right|\\
        & = \actweird{2n}{\alpha_{G, \Sigma_H \cup \Sigma_K}}. 
    \end{align*}
    It now suffices to consider the case where \(H\) and \(K\) are
    characteristic in \(G\). In this case, note that if \(hk \in G\) and \(\phi
    \in \actweird{G}{\Aut}\), then \(\act{hk}{\phi} = \act{h}{\phi} \cdot \act{k}{\phi} = \act{h}{\phi} |_H \cdot
    \act{k}{\phi}|_K\). Thus if \(\mathcal O\) is an automorphic orbit in \(G\), then
    when expressing \(G\) using the Cartesian product of \(H\) and \(K\), we
    have that \(\mathcal O = \mathcal O_H \times \mathcal O_K\), for some orbits
    \(\mathcal O_H\) of \(H\) and \(\mathcal O_K\) of \(K\). Thus the \(\geq\)
    in the above system of inequalities is \(=\) in this case, and so \(\alpha_G
    \sim \alpha_H \cdot \alpha_K\).
  \end{proof}

We return to looking at varying the group of automorphisms used. Unlike the base
group they act on, passing to finite index here does not affect the growth rate.

\begin{lem}\label{finite-index-aut-actions}
      Let \(G\) be a finitely generated group, and \(\Omega \leq \Lambda \leq 
      \actweird{G}{\Aut}\) be such that \(\Omega\) has finite index in \(\Lambda\). Then
      \(\alpha_{G^\Omega} \sim \alpha_{G^\Lambda}\).
\end{lem}
  \begin{proof}
      We need only show that \(\alpha_{G^{\Omega}}  \preccurlyeq
      \alpha_{G^\Lambda}\). We begin by showing that every \(\Lambda\)-orbit is
      the union of at most $[\Lambda : \Omega]$ \(\Omega\)-orbits. Choose a
      left transversal $T\subseteq \actweird{G}{\Aut}$ for the cosets of $\Omega$ in
      \(\Lambda\). For any $g\in G$ we have
      \begin{align*}
          \{\act{g}{\phi} : \phi\in \Lambda\} 
          &= \bigcup_{\psi\in T}\{\act{g}{\comp{\psi}{\tau}}: \tau \in \Omega\}.
      \end{align*}
      and so the \(\Lambda\)-orbit of $g$ is a union of \([\Lambda:\Omega]\) \(\Omega\)-orbits, namely of \(\{\act{g}{\psi} : \psi \in T\}\). Thus the number of \(\Omega\)-orbits of length at most \(n\) is at most $[\Lambda : \Omega]$ times the number of \(\Lambda\)-orbits of length at most \(n\).
      In particular \(\alpha_{G^\Lambda} \succcurlyeq
      \alpha_{G^{\Omega}}\).
  \end{proof}

  If \(G\) is a group such that \(\Out(G)\) is finite, then \(\Inn(G)\) has finite index
  in \(\Aut(G)\), and thus we have the following:

   \begin{cor}\label{finiteOut}
      If $\actweird{G}{\Out}$ is finite, then the automorphic growth function of
      $G$ is equivalent to its conjugacy growth function.
  \end{cor}

  Applying \cref{finiteOut} to special linear groups gives the following. A similar
  argument is used when computing the automorphic growth of Thompson's group \(T\).
  \begin{cor}
           For \(m \geq 3\), \(\SL_m(\Z)\) has exponential automorphic growth.
  \end{cor}

  \begin{proof}
     By \cite{ConjLinear}*{Theorem 1.1}, \(\SL_m(\Z)\) has
     exponential conjugacy growth for \(m \geq 2\), and by \cite{HuaReiner}*{Theorem 3},
     \(\Out(\SL_m(\Z))\) is finite for \(m \geq 3\). The result now follows
     from \cref{finiteOut}.
  \end{proof}

\section{Virtually abelian groups}
\label{sec:VA}
    In this section we study the automorphic growth of virtually abelian groups.
    Automorphism groups of free abelian groups are very well understood; they are
    the general linear groups, and as such computing automorphic orbits is
    straightforward. Despite often being considered to be a `tame' class of
    groups, virtually abelian groups have somewhat complicated automorphism
    groups, related to subgroups of general linear groups. Eick has studied
    these automorphism groups, showing that their generating sets are
    computable \cite{Eick}. The case of virtually abelian groups with
    rank at most \(2\) is less involved due to the fewer possibilities
    of what the automorphic growth can be, with the quadratic upper bound
    provided by standard growth. In the general case, there are virtually
    abelian groups of every possible (exactly) polynomial rank.

    We begin with a classification of the growth of finitely generated free
    abelian groups.
  \begin{theorem}\label{thm:freeabelian}
    If \(A\) is a non-trivial finitely generated free abelian group then
    \(\actweird{n}{\alpha_A} = n + 1\), with respect to the standard generating set.
  \end{theorem}
  \begin{proof}
    Consider the finitely generated free abelian group \(\mathbb{Z}^r\), where
    \(r > 0\). We show that two elements of $\Z^r$ are in the same orbit if and
    only if they have the same greatest common divisor. Note that replacing one
    generator of \(\mathbb{Z}^r\) with its sum with another generator (and
    fixing all other generators) is an automorphism of $G$. Thus the operation
    of adding one entry of a vector to another entry preserves which orbit the
    element belongs to. Let $x=(x_1,\ldots,x_r)\in \Z^r$. By applying the
    Euclidean algorithm (and possibly negating entries at times) we can replace
    $x$ with another element in the same automorphic orbit with all entries the
    same except that the rightmost non-zero entry is now zero, and the entry to
    the left of which is the greatest common divisor of the these two entries.
    We can repeat this until all the entries of $x$ except possibly the first
    are zero. The first entry will be the greatest common divisor of the
    original entries.

    
    We now show these are unique representatives
    for all automorphic orbits. So suppose
    \(M = [m_{ij}] \in \GL_r(\Z)\) is such that
    \(\comp{[x\ 0\ \ldots\ 0]^{\ifthenelse{\boolean{rightactions}}{}{T}}}{\cdot}{M} = [y\ 0\ \ldots\ 0]^{\ifthenelse{\boolean{rightactions}}{}{T}}\),
    where \(x, y \in \N \setminus \{0\}\). Then
    \(m_{\comp{1}{j}} =0\) for all \(j \geq 2\). Using the
    standard formula, \(\actweird{M}{\det} = m_{11} \cdot d\),
    where \(d\) is the determinant of \([m_{ij}]_{i, j = 2}^n\). Since \(m_{11}, d \in \Z\) and \(|\actweird{M}{\det}| = 1\), it follows that
    \(m_{11} = \pm 1\). Thus \(x = \pm y\). Since
    \(x, y \geq 0\), we can conclude that \(x = y\). So these are indeed unique representatives.
    Since 
    \(\actweird{x_1, \ldots, x_r}{\gcd} \leq x_1 + \cdots + x_r\), these are minimal representatives for their automorphic orbits.

    So the number of
    automorphic orbits of \(\Z^r\) of length
    at most \(n\) is just the number of vectors
    of the form \((x, 0, \ldots, 0)\), where
    \(x \in \{0, \ldots, n\}\), and so there are
    exactly \(n + 1\) of these.
  \end{proof}
  Before diving into the characterisation of the automorphic growth of virtually
  abelian groups, we first explicitly compute the growth in a well-known example - 
  the Klein bottle group.
  \begin{ex}
      We show that the Klein bottle group \(\langle a, b \mid \conjugate{a}{b} = a^{-1} \rangle\) is a virtually abelian group with quadratic automorphic growth.
      
      The Klein bottle group has quadratic standard growth, and so the
      automorphic growth is at most quadratic. It therefore suffices to show
      that the growth is at least quadratic. Let \(G\) denote the Klein bottle
      group. Firstly, note that \(\langle a, b^2 \rangle\) is an abelian normal
      subgroup of \(G\). The quotient has size 2, and \(\{1, b\}\) is a
      transversal. So every element of \(G\) can be expressed uniquely as \(a^i
      b^j \), where \(i, j \in \mathbb{Z}\). 
      We begin by showing \(\actweird{G}{Z} = \langle b^2 \rangle\), and so
      \(\langle b^2 \rangle\) is characteristic. From before, we have that \(a\)
      is an element of the abelian group \(\langle a, b^2\rangle\). Thus
      \(C_G(a)\supseteq \langle a, b^2\rangle\). As \(\langle a, b^2\rangle\)
      has index 2, this is a maximal subgroup of \(G\). As \(a\) does not
      commute with \(b\), it follows that \(C_G(a)= \langle a, b^2\rangle\). In
      addition, if \(k \in \mathbb{Z}\), \(\conjugate{a^k}{b} =
      (\conjugate{a}{b})^k = a^{-k}\), and so \(\actweird{b}{C_G} = \langle b
      \rangle\). Thus \(\actweird{G}{Z} = C_G(a) \cap C_G(b) =  \langle b^2
      \rangle\), and \(\langle b^2 \rangle\) is characteristic.

      We next show that \([G, G] = \langle a^2 \rangle\), and so \(\langle a^2
      \rangle\) is also characteristic. The
      abelianisation has presentation \(\makepres{a,b}{\(a^2\)}\) as an abelian
      group. Thus \(a^ib^j\) belongs to the commutator subgroup if and only if
      \(i\) is even and \(j=0\). In particular, \([G, G] = \langle a^2 \rangle\),
      and this is a characteristic subgroup of \(G\). As \(\langle a^2 \rangle\)
      and \(\langle b^2 \rangle\) are infinite cyclic characteristic subgroups
      of \(G\) and \(\actweird{\mathbb{Z}}{\Aut} \cong C_2\), if \(\phi \in
      \actweird{G}{\Aut}\), then \(\act{a^2}{\phi} \in \{a^2, a^{-2}\}\) and
      \(\act{b^2}{\phi} \in \{b^2, b^{-2}\}\). Thus if \(i, j, k, l \in
      \mathbb{N}\setminus \{0\}\) and \(\phi \in \actweird{G}{\Aut}\) are such that
      \(\act{a^{2i} b^{2j}}{\phi} = a^{2k} b^{2l}\), then \(a^{\pm 2i} b^{\pm
      2j} = a^{2k} b^{2l}\). Since \(i, j, k, l > 0\), it follows that \(i = k\)
      and \(j = l\). We have thus shown that \(R = \makeset{a^{2i} b^{2j}}{ \(i,
      j \in \mathbb{Z}_{> 0} \)}\) is a set of pairwise non-automorphic elements
      of \(G\). As \(\langle a^2, b^2 \rangle\) is a finite-index subgroup of
      \(G\), it is undistorted (see, e.g. \cite{Mann}), and so \(| a^{2i}
      b^{2j} | = 4 \mu ij + \mu\) for some (fixed) constant \(\mu\). Since
      this set has quadratic growth, we have thus shown that the automorphic
      growth of \(G\) is at least quadratic, as required.
  \end{ex}

The following lemma allows us to assume that our finite-index free abelian
subgroup is characteristic. This has already been used in the study of
automorphism groups of virtually abelian groups
in \cite{Eick}.
\begin{lem}[{\cite{Eick}*{Lemma 2.1}}]
    \label{virt-ab-char}
    Let \(G\) be a finitely generated virtually abelian group. Then \(G\) admits
    a free abelian characteristic subgroup of finite index.
\end{lem}

When dealing with the quadratic case, we have to study normalisers. Since centralisers
are usually easier to understand, we employ the following easy fact:

\begin{lem}
    \label{norm-cent}
    Let \(G\) be a group and \(F \leq G\) be finite. Then \(\actweird{F}{C_G}\)
    is a finite-index normal subgroup of \(\actweird{F}{N_G}\).
\end{lem}
\begin{proof}
    Define \(\phi\colon \actweird{F}{N_G} \to \actweird{F}{\Sym}\) by \(\act{a}{(\act{f}{\phi})} = \conjugate{a}{f}\). This is clearly a homomorphism, the kernel is \(\actweird{F}{C_G}\) and the image is contained in a finite group.
\end{proof}

The following gives a condition on which automorphisms of our finite-index free
abelian characteristic subgroup \(A\) extend to automorphisms of our virtually
abelian group \(G\). Our proof of the quadratic case is done by showing that
`few' automorphisms satisfy this condition, and thus the relative automorphic
growth of \(A\) must be `close' to its standard growth, which is
quadratic. Since the relative
automorphic growth of \(A\) is a lower bound for the automorphic growth of
\(G\), showing that the relative automorphic growth of \(A\) is quadratic
is sufficient to show that growth of \(G\) is at least quadratic.
\begin{lem}\label{barF_Aintro}
    Suppose that \(G\) is a finitely generated group and \(A\leq G\) is a characteristic abelian subgroup of finite index.
    Let \(\kappa_{A, G}\colon G\to \actweird{G}{\Aut}\) be the usual action by conjugation and let \(\rho_{A, G}\colon \actweird{G}{\Aut}\to \actweird{A}{\Aut}\) be defined by restriction to \(A\).
        Then \(\bar{F}_{A, G} \coloneqq \actweird{\comp{\kappa_{A, G}}{\rho_{A, G}}}{\im}\leq \actweird{A}{\Aut}\) is finite and normal in \(P_{A, G} \coloneqq \actweird{\kappa_{A, G}}{\im}\).
        In particular, \(P_{A, G}\) is contained in the normaliser of \(\bar{F}_{A, G}\) in \(\actweird{A}{\Aut}\).
\end{lem}
\begin{proof}
To show that \(\bar{F}_{A, G}\) is finite, we need only show that \(A\leq \actweird{\comp{\kappa_{A, G}}{\rho_{A, G}}}{\ker}\). This occurs precisely because \(A\) is abelian.
By the Correspondence Theorem,  \(\bar{F}_{A, G}=\act{\actweird{\kappa_{A, G}}{\im}}{\rho_{A, G}}\) is normal in \(P_{A, G}=(\Aut(G))\rho_{A, G}\) if and only if \(\actweird{\kappa_{A, G}}{\im}\) is normal in \(\actweird{G}{\Aut}\). 
But \(\actweird{\kappa_{A, G}}{\im}\) is the inner automorphism group of \(G\), so this is immediate. 
\end{proof}

The next four lemmas are various `standard' facts about matrix groups, which we
require in the proof that rank 2 virtually abelian groups have quadratic growth
if the action on the finite-index characteristic free abelian subgroup is not
entirely by identity or inversion maps. The following is a well-known corollary
to a well-known result. We give the brief proof of the corollary below.

\begin{lem}
    \label{fin-ord-diag}
    If \(r\in \N\) and \(M \in \actweird{\mathbb{C}}{\GL_r}\) is finite order, then it is diagonalisable.
\end{lem}
\begin{proof}
    \cite{HoffmanKunze}*{Chapter 6, Theorem 6} states that the minimal polynomial of $M$ has distinct roots if and only if it is diagonalisable. Since $M$ is finite order, we have $M^k=I$ for some $k$, and hence $x^k-1$ is an annihilating polynomial for $M$, and has distinct roots. The minimal polynomial divides $x^k-1$ and therefore also has distinct roots.
\end{proof}

The following, which we expect is known, describes centralisers in
\(\actweird{\mathbb{C}}{\GL_r}\).
\begin{lem}\label{glnZ-centraliser}
    Suppose that \(M\in \actweird{\mathbb{C}}{\GL_r}\) is diagonalisable with eigenvalues
    \(\lambda_1,\ldots,\lambda_k\in \mathbb{C}\) and multiplicities \(m_1,\ldots,
    m_k\), respectively. Then \(\actweird{M}{C_{\actweird{\mathbb{C}}{\GL_r}}}\cong \oplus_{i=1}^k
    \actweird{\mathbb{C}}{\GL_{m_i}}\) (with equality when \(k=1\)).
\end{lem}
\begin{proof}
    Let \(E_{i}\leq \mathbb{C}^r\) be the eigenspaces of \(M\) with values \(\lambda_i\) (here we view the elements $v\in E_i$ as \row\ matrices such that $\comp{v}{\cdot}{M}=\lambda_i v$).
    If \(B\in \actweird{M}{C_{\actweird{\mathbb{C}}{\GL_r}}}\), then for all \(v_i\in E_i\) we have 
    \[\comp{v_i}{B}{M}=\comp{v_i}{M}{B}=\comp{(\lambda_i v_i)}{B}=\lambda_i \comp{v_i}{B}.\]
    So \(M\) scales all elements of \(\comp{E_i}{B}\) by \(\lambda_i\).
    It follows from the definition of \(E_i\) that \(\comp{E_i}{B}=E_i\).
    So for all \(B\in \actweird{M}{C_{\actweird{\mathbb{C}}{\GL_n}}}\) and all \(i\in \{1, \ldots, k\}\) we can define \(B_i \colon E_i\to E_i\) be the automorphism induced by \(B\).

    Conversely, if \(B_i\colon E_i\to E_i\) for \(i\in \{1, \ldots, k\}\) are any
    automorphisms of \(E_i\), then there exists \(B\in \actweird{\mathbb{C}}{\GL_r}\) which
    extends each \(B_i\) and such that \(BM = MB\) (as \(\conjugate{M}{B}\)
    scales any eigenvector of \(M\) as $M$ does). We now have
    \(\actweird{M}{C_{\actweird{\mathbb{C}}{\GL_r}}} \cong \oplus_{i=1}^k \actweird{E_i}{\Aut}\cong
    \oplus_{i=1}^k \actweird{\mathbb{C}}{\GL_{m_i}}\). Note that in the case \(k=1\) the
    first isomorphism is actually equality.
\end{proof}

Using the well-known result that \(\actweird{\Z}{\GL_2}\) is virtually free, we describe
its centralisers.
\begin{lem}
    \label{finite-ord-cent-Z2}
    If $M\in \actweird{\Z}{\GL_2}$ has finite order, then we have one of 
    \begin{enumerate}
        \item \(M\) is a scalar matrix and \(\actweird{M}{C_{\actweird{\Z}{\GL_2}}} = \actweird{\Z}{\GL_2}\);
        \item  $\actweird{M}{C_{\actweird{\Z}{\GL_2}}}$ is abelian and virtually cyclic.
    \end{enumerate}
\end{lem}
\begin{proof}
By \cref{fin-ord-diag}, \(M\) is diagonalisable. 
Thus we can apply \cref{glnZ-centraliser}. If \(M\) has a unique eigenvalue then $M$ is scalar and we are in case (1).

Suppose that \(M\) has two distinct eigenvalues. Then \cref{glnZ-centraliser} implies
that $\actweird{M}{C_{\actweird{\Z}{\GL_2}}}$ embeds in $\actweird{\mathbb{C}}{\GL_1}\oplus \actweird{\mathbb{C}}{\GL_1}$.
This group is abelian, so $\actweird{M}{C_{\actweird{\Z}{\GL_2}}}$ is abelian. Recall the standard result that $\actweird{\Z}{\GL_2}$ is virtually free and let \(H\) be a finite-index free subgroup. 
It follows that \(H \cap \actweird{M}{C_{\actweird{\Z}{\GL_2}}}\)
is both free and abelian, and so \(H \cap \actweird{M}{C_{\actweird{\Z}{\GL_2}}}\) is either infinite
cyclic or trivial. 
Moreover \(H\cap \actweird{M}{C_{\actweird{\Z}{\GL_2}}}\) is finite index in
\(\actweird{M}{C_{\actweird{\Z}{\GL_2}}}\) so the result follows.
\end{proof}
The following is well-known, but we include a short proof for completeness.
\begin{lem}\label{removing complexity}
   If \(M\in \actweird{\Z}{\GL_2}\) has a non-real eigenvalue, then either \(M^{6}\) or \(M^{4}\) is the identity matrix.
\end{lem}
\begin{proof}

    It is well known and easy to check that the characteristic polynomial of \(M\) is
  \(x^2-\operatorname{trace}(M)x+\operatorname{det}(M)\).
  Thus the quadratic formula implies that the eigenvalues of \(M\) are
  \[\frac{\actweird{M}{\operatorname{trace}}}{2} \pm \sqrt{\frac{\actweird{M}{\operatorname{trace}}^2}{4}-\actweird{M}{\det}}.\]

 It follows that \(\actweird{M}{\det}=1\) and \(\actweird{M}{\operatorname{trace}}\in \{-1,0, 1\}\). 
  Thus the only possible eigenvalues are 
  \[e^{\frac{i\pi}{3}},e^{\frac{2i\pi}{3}},e^{-\frac{i\pi}{3}},e^{-\frac{2i\pi}{3}}, i,\text{ and }-i.\]
  As these eigenvalues must come in a conjugate pair, it follows that \(M\) is diagonalisable with diagonalisation
  \(\begin{bmatrix}
      e^{i\theta } & 0\\
     0& e^{-i\theta }\\
  \end{bmatrix}\) where \(\theta\) is an integer multiple of \(\frac{\pi}{3}\) or \(\frac{\pi}{2}\). The result follows.
 \end{proof}

We have now completed enough setup to embark on the proof of the case when the
conjugation action of the virtually abelian group \(G\) on the finite-index
characteristic free abelian subgroup is not entirely by the identity map or
inversion. Despite having seemingly `weaker' assumptions, this case ends up being
far simpler than the remaining case, due to the fact that to show that the
growth is quadratic, a lower bound of quadratic growth is all that is needed.
\begin{proposition}\label{VirtuallyZ2Quad}
    Let \(G\) be a group with a finite-index normal subgroup $A\cong \Z^2$. 
    Suppose there is $g\in G$ such that the automorphism $a\mapsto \conjugate{a}{g}$ of \(A\) is neither the identity map nor the inversion map.
    Then \(G\) has quadratic automorphic growth.
\end{proposition}

\begin{proof}
    By \cref{virt-ab-char} there exists a finite index characteristic subgroup \(A'\) of
    \(G\), with \(A' \leq A\). The conjugation action of $g$ on $A'$ is still neither inversion nor the identity as an automorphism of $A'$ can have at most one extension to $A$.
    Thus without loss of generality, we can use \(A'\) as \(A\); that is, we can
    assume \(A\) is characteristic.

    Fix a (finite) right transversal \(T\) for \(A\) in \(G\).
    Let \(P_{A, G} \leq \actweird{A}{\Aut}\) be the group of automorphisms of \(G\) restricted to
  \(A\). Let \(\bar{F}_{A,G} \leq \actweird{A}{\Aut}\) be the (finite) subgroup of automorphisms defined by
  conjugation of \(A\) by elements of \(T\).

  By \cref{lem:charsbgrp}, it suffices to show that \(\alpha_{A^{P_{A, G}}\subseteq G}\) is
  quadratic. Note that since $A$ has finite index, it is undistorted, so $\alpha_{A^{P_{A,G}}\subseteq G}\sim\alpha_{A^{P_{A,G}}}$ by \cref{relativeundistorted}. From \cref{barF_Aintro}, we have \(P_{A, G} \leq N_{\Aut(A)}(\bar{F}_{A,
  G})\). Since \(\bar{F}_{A, G}\) is finite, \cref{norm-cent} tells us that
  \(\actweird{\bar{F}_{A, G}}{N_{\actweird{A}{\Aut}}}\) contains \(\actweird{\bar{F}_{A, G}}{C_{\actweird{A}{\Aut}}}\) as a
  finite-index normal subgroup. Finally, since \(\bar{F}_{A, G}\) contains non-scalar
  matrices, \cref{finite-ord-cent-Z2} tells us that \(\actweird{\bar{F}_{A,G}}{C_{\actweird{\Z}{\GL_2}}}\)
  is virtually cyclic. Thus \(\actweird{\bar{F}_{A, G}}{N_{\actweird{A}{\Aut}}}\) is virtually cyclic and so
  \(P_{A, G}\) is virtually cyclic; that is, virtually \(\Z\) or finite. If
  \(P_{A, G}\) is finite, it is virtually trivial so \cref{finite-index-aut-actions} implies that
  \(\alpha_{A^{P_{A, G}}}\) is equivalent to the standard growth of $A$, which is
  quadratic.

  So suppose that \(P_{A, G}\) is virtually \(\langle M \rangle\), where \(M\) has
  infinite order. By \cref{removing complexity}, the eigenvalues of \(M\) are
  real. By squaring \(M\) if needed, we may also assume that \(M\) has no
  negative eigenvalues and \(\actweird{M}{\det}=1\). Then \cref{finite-index-aut-actions}
  implies that it suffices to show that \(\alpha_{A^{\langle M \rangle}}\) is at least
  quadratic, noting the \(\alpha_{A^{\langle M \rangle}}\) is bounded above by the
  standard growth of \(A\), and so can be at most quadratic.
  
  Let \(\lambda\) be the largest (row) eigenvalue for \(M\) and let \(v\) be a corresponding real eigenvector.
  If it is possible, choose \(w\) to be a real eigenvector (not necessarily corresponding to
  \(\lambda\)) linearly independent with
  \(v\); if there are no such vectors, then let \(w\) be any real vector linearly
  independent with \(v\). 
  Let \(l, a \in \R\) be the scalars such that \(\act{w}{M}=lw+av\).
  From the equations
  \[\act{v}{M}=\lambda v, \qquad \act{w}{M}=av+lw\]
  we see that the characteristic polynomial of $M$ is $(x-l)(x-\lambda)$.
  Thus \(l\) is an eigenvalue of \(M\) and \(l\lambda=\actweird{M}{\det}=1\). 
  Thus \(\act{w}{M}=\frac{1}{\lambda}w+av\), and \(\lambda\geq 1\) (since \(\lambda\) was chosen to be the largest eigenvalue). 
  As \(M\) is not the identity matrix, either \(a\neq 0\) or \(\lambda> 1\). 
  The choice of \(w\) implies that we cannot have both of these, for otherwise \(M\) would have two distinct eigenvalues and \(w\) would not be an eigenvector corresponding to either of them.
  In summary, we may assume that exactly one of
  the following holds:
  \begin{enumerate}[(a)]
      \item \(\lambda>1\) and \(\act{w}{M} = \frac{1}{\lambda} w\);
      \item \(\lambda=1\), \(\act{v}{M}=v\), and \(\act{w}{M} =w+av\) for some \(a\in \R \setminus \{0\}\).
       \end{enumerate}
  In case (a), consider the set
 \[B \coloneqq \makeset{p\in \Z^2}{\(p=xv+yw\) where \(\frac{1}{\lambda}<|\frac{x}{y}|<\lambda\)}.\]
 Note that for all \(k\in \Z\), we have that if \(p \in \Z^2\), writing \(\act{p}{M^k} = x_k v +y_k w\),
 then 
 \begin{align*}
    p \in \act{B}{M^k} & \iff \act{p}{M^{-k}} \in B \\
    & \iff \lambda^{-1} <  \left| \frac{x_{-k}}{y_{-k}} \right| < \lambda \\
    & \iff \lambda^{-1} < \left| \frac{x_0 \lambda^{-k}}{y_0 \lambda^k} \right|<\lambda \\
    & \iff \lambda^{-1} < \lambda^{-2k} \left| \frac{x_0}{y_0} \right| < \lambda.
 \end{align*}
 Thus
 \begin{align*}
     \act{B}{M^k} &=\makeset{p\in \Z^2}{\(p=xv+yw\) where
     \(\lambda^{-1}<|\frac{x}{y}|\lambda^{-2k}<\lambda\)}\\
     &=\makeset{p\in \Z^2}{\(p=xv+yw\) where \(\lambda^{2k-1}<|\frac{x}{y}|<\lambda^{2k+1}\)}\\
 \end{align*}
 which are pairwise disjoint as $k$ ranges over $\Z$. Since each of these sets grow quadratically (as subsets
 of \(\Z^2\)), it follows that the growth \(\alpha_{A^{\langle M \rangle}}\) must
 be at least quadratic.
 
    In case (b), consider the set
 \[B \coloneqq \makeset{p\in \Z^2}{\(p=xv+yw\) where \(|x|<\frac{1}{2}|ay|\)}.\] 
 As \(\act{xv+yw}{M}=xv+yw+yav=(x+ay)v+yw\), we have for all \(k \in \Z\) that
 \[\act{B}{M^k}=\makeset{p\in \Z^2}{\(p=xv+yw\) where \(\frac{2k-1}{2}|ay|<x<\frac{2k+1}{2}|ay|\)}\]
 which are pairwise disjoint as \(k\) ranges over \(\Z\). Since each of these
 sets grow quadratically (as subsets of \(\Z^2\)), it follows that the growth
 \(\alpha_{A^{\langle M \rangle}}\) must be at least quadratic.
\end{proof}

The following lemma allows us to deal with virtually abelian groups where the
conjugation action of the group on the finite-index characteristic free abelian
subgroup is by only scalar matrices (that is, by the identity matrix or
inversion). In the case when the action is just by the identity matrix; that is,
our extension is central, we have a finite-index centre. Schur's Theorem
(see for example \cite{IsaacsFinGpTheory}*{Theorem 5.7}) states that a group
with a finite-index centre has a finite commutator subgroup. To deal with the
non-central case, we need an analogue of the commutator subgroup that deals
with this `inversion' action.  This turns out to be the kernel of the map
from our virtually abelian group \(G\) to \(\Z^n \rtimes C_2\), with the
\(C_2\) correpsonding to the inversion action.

\begin{lem}\label{schur-cor}
    Suppose that $G$ is a group with a finite-index characteristic subgroup $A\cong \Z^r$. Suppose further that there is a homomorphism $\operatorname{sgn}
      \colon G\to \{1,-1\}$ such that for all $g\in G$ and $x \in
      A$ we have $\conjugate{x}{g}= x^{\actweird{g}{\operatorname{sgn}}}$. 
      Then there is a homomorphism $\sigma\colon G\to \Z^r \rtimes \{1,-1\}$ with $\{1,-1\}$ acting by inversion on $\Z^r$ such that $\sigma|_A$ is injective and the $\{1,-1\}$ coordinate of $\act{g}{\sigma}$ agrees with $\actweird{g}{\sgn}$.
\end{lem}
\begin{proof}
    Let $H$ be the kernel of $\sgn$ and note that $A$ is central in $H$. By
    Schur's Theorem, (see for example \cite{IsaacsFinGpTheory}*{Theorem 5.7}),
    we have \([H,H]\) is finite, and so $[H,H]\cap A=\{1_G\}$. 
    Let $C_H$ be the
    subgroup of $H$ which is the preimage of the torsion subgroup of
    $H/[H,H]$ and note that $C_H$ is characteristic in $H$. 
    Note
    that \(C_H\) is the extension of a subgroup of \([H, H]\) (which is
    therefore finite) by the torsion subgroup of \(H / [H, H]\), which is
    also finite, and thus \(C_H\) itself is finite. The fact that
    \(C_H\) is characteristic in \(H\) together with
    the fact that $H$ is normal in $G$ implies that $C_H$ is normal
    in $G$. 
    Thus $H/C_H \cong \Z^r$. Thus the quotient $G/C_H$ is an index \([G:
    H] \leq 2\) extension of $\Z^r$. The index is two if and only if there
    exists $g\in G$ with $\actweird{g}{\sgn}=-1$, and note that \(\{1, g\}\) is a
    transversal for \(\Z^r\) in \(G / C_H\). If no such element exists
    then we are done. Otherwise, there exists $g\in G$ such that $\sgn(g)=-1$.
    We have that $C_H g$ acts by conjugation on \(G / C_H\) as
    required in the statement of the lemma. We need only show that the extension
    splits. This holds as $C_H g^2$ is an element of the $\Z^r$ part of the
    quotient which commutes with $g$ so must be the identity.
\end{proof}

Before we can complete the proof of this remaining case when the action of
\(G\) on \(A\) is by scalar matrices, we must first understand how `often'
automorphisms of \(A\) extend to automorphisms of \(G\). This culminates
in \cref{autA-is-finiteindex} which shows that the subgroup of such automorphisms
of \(A\) has finite index in \(\actweird{A}{\Aut}\). We first require two technical lemmas
before we can prove this. Ultimately, this lemma and the following lemma,
\cref{matrix-lem} study a set \(N_G\) which considers which automorphisms
extend. 

A monoid presentation is called \emph{special} if every relation is of the form
\(w = 1\), for some word \(w\) over the generators. We frequently abuse notation
and consider the relations of a special monoid presentation to be the set of
words \(w\), rather than the set of pairs \((w, 1)\), as one commonly does for
group presentations.
  \begin{lem}\label{va-extend-aut}
      Suppose that we have the following:
      \begin{enumerate}
          \item $G$ is a group with a finite index characteristic 
      subgroup $A\cong \Z^m$. We denote the isomorphism between \(A\)
      and \(\Z^m\) by \(a\mapsto a^Z\) and the element of \(A\) mapped to the \(i^{th}\) standard generator of \(\Z^m\) by \(a_i\) (for an endomorphism \(\gamma\) of \(A\) we also denote the corresponding endomorphism of \(\Z^m\) by \(\gamma^Z\));
      \item fthere is a homomorphism $\operatorname{sgn}
      \colon G\to \{1,-1\}$ such that for all $g\in G$ and $x \in
      A$ we have $\conjugate{x}{g}= x^{\actweird{g}{\operatorname{sgn}}}$;
      \item \(\langle \Sigma \mid R \rangle\) is a finite monoid presentation, where
      all relations in \(R\) are of the form \(r = 1\) (and thus simply considered
      to be positive words \(r\)),
      \(\Sigma\subseteq G\) and the map \(s\mapsto sA\) for \(s\in \Sigma\)
      defines an isomorphism from the group presented by \(\langle \Sigma \mid R
      \rangle\) to \(G/A\);
      \item for all \(r\in R\),  $r = r_1\ldots r_{|r|}$, with \(r_i \in \Sigma\), \(x_r\in (\{a_1,\ldots, a_n\}\cup \{a_1^{-1},\ldots, a_n^{-1}\})^\ast\) is a string such that the relation \(r=x_r\) holds in \(G\);
      \item \[ 
        N_G=\makeset{(\phi,\gamma)}{\(\phi\in \actweird{A}{\Aut}\), \(\gamma \colon
        \Sigma\to A\) and for all $r\in R$\\ $\sum_{i=1}^{|r|} \actweird{r_1\ldots
        r_{i - 1}}{\sgn} (\act{r_{i}}{\gamma})^Z=\act{x_r^Z}{\phi^Z}-x_r^Z$};
      \]
      \item 
       \[P=\makepres{\{a_1, \ldots, a_n\} \cup \Sigma}{for all \(i,j\in \{1,\ldots,n\},s\in \Sigma,r\in R\)\\ 
       \(a_i a_j = a_j a_i, \conjugate{a_i}{s} = a_i^{\actweird{s}{\sgn}}, r = x_r\)}.
      \]
      \end{enumerate} 
      Then \(P\) is a presentation for \(G\) with the generators in the presentation representing themselves and
      \[ 
          N_G=\makeset{(\phi,\gamma)}{\(\gamma \colon \Sigma\to A\) and there is
          $\tilde{\phi}\in \actweird{G}{\Aut}$ with $\tilde{\phi}|_A=\phi$ and\\ for $s\in
          \Sigma$ we have $\act{s}{\tilde{\phi}}=\act{s}{\gamma} \cdot s$}.
      \]
  \end{lem}
    \begin{proof}
        We first show that \(P\) presents \(G\). Let \(\xi \colon P\to G\) be
        the canonical quotient map. We show \(\xi\) is injective. 
        Since \(\langle a_1,\ldots, a_n\rangle \cong \Z^m\) in \(G\), and these elements commute in \(P\), it follows that \(\xi|_{\langle a_1,\ldots, a_n\rangle} \colon \langle a_1,\ldots, a_n\rangle\to A\) is an isomorphism.
        Thus \(\actweird{\xi}{\ker}\cap \langle a_1,\ldots, a_n\rangle=\{1_P\}\).        
        From the definition of \(R\), the presentation obtained from \(P\) by setting \(a_1=\ldots=a_n=1\) is a presentation for \(G/A\). 
        Thus the kernel of the composition of \(\xi\) with the quotient map from \(G\) to \(G/A\) is \(\langle\langle a_1,\ldots, a_n\rangle\rangle=\langle a_1,\ldots, a_n\rangle\). Thus any element of the kernel of \(\xi\) belongs to \(\langle a_1,\ldots, a_n\rangle\) implying that it is trivial. It follows that \(\xi\) is injective, as required.

       Let $\phi\in \actweird{A}{\Aut}$ and \(\gamma \colon \Sigma\to A\) be a function.
        From the presentation \(P\), it follows that there is $\tilde{\phi}\in \End(G)$ with
        $\tilde{\phi}|_A=\phi$ and $\act{s}{\tilde{\phi}}=\act{s}{\gamma} \cdot s$ for all \(s\in \Sigma\) if and only if the following conditions hold:
        \begin{enumerate}
            \item $\act{a_i}{\phi} \cdot \act{a_j}{\phi} = \act{a_j}{\phi} \cdot \act{a_i}{\phi}$
            for all \(i, j \in  \{1, \ldots, n\}\);
            \item $\comp{(\act{s}{\gamma} \cdot s)^{-1}} {\cdot} {\act{a_i}{\phi}} {\cdot} {(\act{s}{\gamma} \cdot s)} = (\act{a_i}{\phi})^{\actweird{s}{\sgn}}$ for $i \in \{1, \ldots, n\}$ and $s\in \Sigma$,
            \item $\prod_{i=1}^{|r|}(\act{r_i}{\gamma} \cdot r_i) =\act{x_r}{\phi}$ for all \(r\in R\).
        \end{enumerate}
        Condition (1) always holds as $\phi$ is a homomorphism and condition (2) always holds as
        \begin{align*}
            \conjugate{\cdot \act{a_i}{\phi} \cdot}{(\act{s}{\gamma} \cdot s)} = \conjugate{\cdot (\act{a_i}{\phi})\cdot}{s}=((a_i)\phi)^{\actweird{s}{\sgn}}.
        \end{align*}
        Thus, there is a homomorphism $\tilde{\phi} \colon G\to G$ with
        $\tilde{\phi}|_A=\phi$ and $\act{s}{\tilde{\phi}}=\act{s}{\gamma} \cdot s$ for all \(s\in \Sigma\) if and only if 
        \[\prod_{i=1}^{|r|}(\act{r_i}{\gamma} \cdot r_i) =\act{x_r}{\phi}\]
        for all \(r\in R\). 
        Applying the relations of \(P\) to the left hand side of this equality yields
        the following equivalent equality:
        \[\left(\prod_{i=1}^{|r|}(\act{r_i}{\gamma})^{\actweird{r_1r_2\cdots r_{i-1}}{\sgn}} \right)x_r=\act{x_r}{\phi}.\]
        Using the structure of \(A\), this is in turn equivalent to
       \[\sum_{i=1}^{|r|} \actweird{r_1\ldots
        r_{i - 1}}{\sgn} (\act{r_{i}}{\gamma})^Z=\act{x_r^Z}{\phi^Z}-x_r^Z.\]
       So we now have
        \[
            N_G=\makeset{(\phi,\gamma)}{there is a homomorphism $\tilde{\phi} \colon G\to G$ such that\\
            $\tilde{\phi}|_A=\phi$ is an automorphism of \(A\) and for all\\ $s\in \Sigma$ we have
            $\act{s}{\tilde{\phi}}=\act{s}{\gamma}\cdot s$}.
        \]
        All of these homomorphisms \(\tilde{\phi}\) above permute \(A\) and
        preserve each coset setwise. Thus the homomorphisms \(\tilde{\phi}\)
        above are surjective. They are also injective since if \(s\in \Sigma^\ast\)
        and \(a,b\in A\), then
        \[\act{as}{\tilde{\phi}}=\act{bs}{\tilde{\phi}}\Rightarrow \act{a}{\tilde{\phi}}\act{s}{\tilde{\phi}}=\act{b}{\tilde{\phi}}\act{s}{\tilde{\phi}}\Rightarrow \act{a}{\tilde{\phi}}=\act{b}{\tilde{\phi}}\Rightarrow a=b\Rightarrow as=bs.\]
        Thus
        \[
            N_G=\makeset{(\phi,\gamma)}{there is $\tilde{\phi}\in \actweird{G}{\Aut}$ such that \(\tilde{\phi}|_A=\phi\) and for $s\in \Sigma$ we have
            $\act{s}{\tilde{\phi}}=\act{s}{\gamma}\cdot s$}
        \]
        as required.
\end{proof}

The following is mostly a restatement of \cref{va-extend-aut} into matrix notation,
which allows its application in the proof of \cref{autA-is-finiteindex}.
\begin{lem}\label{matrix-lem}
Suppose that we have the following:
      \begin{enumerate}
          \item $G$ is a group with a finite index characteristic 
      subgroup $A\cong \Z^m$. We denote the isomorphism by \(a\mapsto a^Z\) and the element of \(A\) mapped to the \(i^{th}\) standard generator of \(\Z^m\) by \(a_i\) (for \(\gamma \colon A\to A\) we also denote the corresponding self-map of \(\Z^m\) by \(\gamma^Z\));
      \item there is a homomorphism $\operatorname{sgn}
      \colon G\to \{1,-1\}$ such that for all $g\in G$ and $x \in
      A$ we have $\conjugate{x}{g}= x^{\actweird{g}{\operatorname{sgn}}}$;
      \item \(\langle \Sigma \mid R \rangle\) is a finite monoid presentation, where
      all relations in \(R\) are of the form \(r = 1\) (and thus simply considered
      to be positive words \(r\)),
      \(\Sigma\subseteq G\) and the map \(s\mapsto sA\) for \(s\in \Sigma\)
      defines an isomorphism from the group presented by \(\langle \Sigma \mid R
      \rangle\) to \(G/A\);
      \item for all \(r\in R\),  $r = r_1\ldots r_{|r|}$, with \(r_i \in \Sigma\) and \(x_r\in (\{a_1,\ldots, a_m\}\cup \{a_1^{-1},\ldots, a_m^{-1}\})^\ast\) is a string such that the relation \(r=x_r\) holds in \(G\);
      \item for all $\phi\in \actweird{A}{\Aut}$, we write $\gamma_\phi \colon A \to A$ for
      the map $x\mapsto (x)\phi\cdot x^{-1}$ and \(\Gamma_\phi\) for the \(m
      \times m\) matrix which acts on \(\Z^m\) as \(\gamma_\phi^Z\) does (where we view elements of \(\Z^m\) as \row\ vectors);
      \item  $\Lambda_{\langle \Sigma |R\rangle}$ is the $\comp{R}{\times}{\Sigma}$ matrix with integer
    entries defined by
    \[
        \Lambda_{\langle \Sigma |R\rangle} = \left[\mathop{\sum_{i = 1}^{|r|}}_{r_i=t} \actweird{r_{1 } \cdots r_{i - 1}}{\sgn} \right]_{(\comp{r}{,}{t})} ;
    \]
    \item \(W_{\langle \Sigma |R\rangle}\) is the \(\comp{R}{\times}{m}\) matrix whose \(r^{th}\) \row\ is \(x_r^Z\);
    \item \(N_G\) is the set from \cref{va-extend-aut}.
      \end{enumerate}
Then there is a full rank $R\times R$ integer matrix $L_{\langle \Sigma |R\rangle}$ such that $\comp{L_{\langle \Sigma |R\rangle}}{\Lambda_{\langle \Sigma |R\rangle}}$ is an integer multiple of a reduced \row\ echelon matrix and
    \begin{align*}
        N_G&=\makeset{(\phi,\gamma)}{\(\phi\in \actweird{A}{\Aut}\), the $\comp{\Sigma}{\times}{n}$
        matrix \(X\) whose \(t^{th}\) \row\ is \((\act{t}{\gamma})^Z\) \\satisfies
        \(\comp{L_{\langle \Sigma |R\rangle}}{\Lambda_{\langle \Sigma |R\rangle}}{X} =  \comp{L_{\langle \Sigma |R\rangle}}{W_{\langle \Sigma |R\rangle}}{\Gamma_\phi}\)}
    \end{align*}    
\end{lem}
\begin{proof}
    If \(\tilde{\phi} \in \actweird{G}{\Aut}\) is such that $\tilde{\phi}$ fixes each
    coset in \(G / A\) setwise, then we define $\gamma_{\tilde{\phi}} \colon
    G\to A$ by $\act{g}{\gamma_{\tilde{\phi}}}=\act{g}{\tilde{\phi}} \cdot g^{-1}$.  \cref{va-extend-aut} says that the
    possible pairs $(\tilde{\phi}|_A, \gamma_{\tilde{\phi}}|_\Sigma)$ are
    described by
    \[
        N_G=\makeset{(\phi,\gamma)}{\(\phi\in \actweird{A}{\Aut}\), \(\gamma \colon \Sigma\to
        A\) and for all $r = r_1 \cdots r_{|r|}\in R$\\ $\sum_{i=1}^{|r|} \actweird{r_1\ldots r_{i-1}}{\sgn} \cdot
        (\act{r_i}{\gamma})^Z=\act{x_r^Z}{\phi^Z}-x_r^Z$}.
    \]
    Since \(A\) is abelian, for each \(\phi\in \actweird{A}{\Aut}\),
    $\gamma_\phi$ is always an endomorphism of $A$. The system of equations in the set \(N_G\) above written as a matrix equation says exactly that
    \[\comp{\Lambda_{\langle \Sigma |R\rangle}}{X} = \comp{W_{\langle \Sigma |R\rangle}}{\Gamma_\phi},\]
    where \(X\) is the \(\comp{\Sigma}{\times}{m}\) matrix whose \(t^{th}\) \row\ is \((\act{t}{\gamma})^Z\). It follows that
    \begin{align*}
        N_G&=\makeset{(\phi,\gamma)}{\(\phi\in \actweird{A}{\Aut}\), the $\comp{\Sigma}{\times}{m}$
        matrix \(X\) whose \(t^{th}\) \row\ is \((\act{t}{\gamma})^Z\) \\satisfies
        \(\comp{\Lambda_{\langle \Sigma |R\rangle}}{X} = \act{W_{\langle \Sigma |R\rangle}}{\Gamma_\phi}\)}.
    \end{align*}
    Let \(U \in \actweird{\Q}{\GL_m}\) be such that \(\comp{U}{\Lambda_{\langle \Sigma |R\rangle}}\) is in
     reduced \row\ echelon form. Let \(L_{\langle \Sigma |R\rangle}\) be a positive integer multiple of \(U\) with integer entries. As \(L_{\langle \Sigma |R\rangle}\) is invertible over \(\mathbb{Q}\), it follows that this matrix can be cancelled on the left. The result follows.
\end{proof}

Before we can show that the set of automorphisms of \(A\) that extend to automorphisms
of \(G\) has finite index in \(\actweird{A}{\Aut}\), we require the result which gives the 
sufficient condition for a subgroup to have finite index in \(\actweird{A}{\Aut}\) which we show is satisfied in
the proof of \cref{autA-is-finiteindex}.
  \begin{lem}\label{finite-index-stab}
      Let \(H\) be finite index subgroup of a finitely generated group \(G\). 
      Then the group 
      \[
          K \coloneqq \makeset{\phi\in \actweird{G}{\Aut}}{\(\act{H}{\phi} = H\)}
      \]
      has finite index in \(\actweird{G}{\Aut}\).
  \end{lem}
  \begin{proof}
     Let \(J \coloneqq \makeset{\act{H}{\phi}}{\(\phi\in
     \actweird{G}{\Aut}\)}\). It suffices to show that \(J\) is finite. The set
     \(J\) consists only of index \([G:H]\) subgroups of \(G\), of which there
     are finitely many. 
  \end{proof}

We continue with our another technical lemma in a similar vein to
\cref{matrix-lem}, except this time we are able to conclude something about the
group of automorphisms of \(A\) that extend to automorphisms of \(G\).

\begin{lem}\label{lem:finindextend}
    Suppose that 
    \begin{enumerate}
        \item  $G$ is a group with a finite-index characteristic subgroup $A\cong \Z^m$;
        \item there is a homomorphism $\operatorname{sgn}
      \colon G\to \{1,-1\}$ such that for all $g\in G$ and $x \in
      A$ we have $\conjugate{x}{g}= x^{\actweird{g}{\operatorname{sgn}}}$;
        \item  \(\langle \Sigma \mid R\rangle\), \(L_{\langle \Sigma
        |R\rangle}\), \(W_{\langle \Sigma
        |R\rangle}\), \(\Lambda_{\langle \Sigma |R\rangle}\), and \(\gamma_\phi\) are as in \cref{matrix-lem};
        \item when $r\in R$ corresponds to a zero \row\ of $\comp{L_{\langle \Sigma
        |R\rangle}}{\Lambda_{\langle \Sigma |R\rangle}}$, the corresponding
        \row\ of $\comp{L_{\langle \Sigma |R\rangle}}{W_{\langle \Sigma
        |R\rangle}}$ is also a zero \row;
        \item \(D\subseteq R\) is the set of non-zero rows of \(\comp{L_{\langle \Sigma |R\rangle}}{W_{\langle \Sigma
        |R\rangle}}\), \(E \subseteq \Sigma\) corresponds to the set of non-pivot columns of \(\comp{L_{\langle \Sigma |R\rangle}}{W_{\langle \Sigma
        |R\rangle}}\) and \(\Omega\leq \Aut(G)\) is the group of automorphisms which fix each element of \(E\).
    \end{enumerate}
      Then the group \(P_{A,G,\Omega}\) of automorphisms of \(A\) extendable to
      automorphisms in \(\Omega\) has finite index in \(\actweird{A}{\Aut}\) and whenever \(\phi\in \Aut(A)\) has an extension \(\tilde{\phi}\in \Omega\), we have 
    \[\beta_d \cdot
(\act{p_d}{\tilde{\phi}}\cdot p_d^{-1})^Z=\act{\actweird{\comp{L_{\langle \Sigma |R\rangle}}{W_{\langle \Sigma |R\rangle}}}{\operatorname{\row}_d}}{\gamma_\phi^Z}\]
        for all \(d\in D\), where \(\beta_d\) is the first non-zero entry of \row\ \(d\) in $\comp{L_{\langle \Sigma
        |R\rangle}}{\Lambda_{\langle \Sigma |R\rangle}}$  and \(p_d\in \Sigma\) corresponds to the column containing the entry \(\beta_d\).
\end{lem}
\begin{proof}
From point (4), it follows that for all \(\Sigma \times m\) integer matrices \(X\), the equation \(\comp{L_{\langle \Sigma |R\rangle}}{\Lambda_{\langle \Sigma |R\rangle}}{X} =  \comp{L_{\langle \Sigma |R\rangle}}{W_{\langle \Sigma |R\rangle}}{ \Gamma_\phi}\) holds if and only if it holds in the rows which are not zero \row s of $\comp{L_{\langle \Sigma |R\rangle}}{\Lambda_{\langle \Sigma |R\rangle}}$.
From \cref{matrix-lem}, we know that the sets
\[\makeset{(\phi,\gamma)}{\(\phi\in \actweird{A}{\Aut}\), \(\gamma \colon \Sigma\to A\), the $\comp{\Sigma}{\times}{m}$
        matrix \(X\) whose \(t^{th}\) row \\ is \((\act{t}{\gamma})^Z\) satisfies
        \(\comp{L_{\langle \Sigma |R\rangle}}{\Lambda_{\langle \Sigma |R\rangle}}{X} = \comp{L_{\langle \Sigma |R\rangle}}{W_{\langle \Sigma |R\rangle}}{\Gamma_\phi}\) in the\\ \row s from \(D\).}
        \]
         \[\makeset{(\phi,\gamma)}{\(\phi\in \Aut(A)\), \(\gamma \colon \Sigma \to A\) and there is
         $\tilde{\phi}\in \Aut(G)$ with $\tilde{\phi}|_A=\phi$ and\\ for $t\in
         \Sigma$ we have $\act{t}{\tilde{\phi}}=\act{t}{\gamma} \cdot t$}\]
         are equal. 
         In particular, the case that 
         \(\tilde{\phi}\in \Omega\) corresponds to when the rows of \(X\) from \(E\) are \(\bold{0}\).
For all $d\in D$, let
     $\beta_d\in \mathbb{Z}$ be the only non-zero entry in the \(d^{th}\) row of
     \(\comp{L_{\langle \Sigma |R\rangle}}{\Lambda_{\langle \Sigma |R\rangle}}\)
     whose column is not from \(E\), which exists due to the structure of the matrix \(\comp{L_{\langle \Sigma
     |R\rangle}}{\Lambda_{\langle \Sigma |R\rangle}}\) 
     described in \cref{matrix-lem}. For each \(e\in E\), let \(\beta_{d,e}\) be the entry of
     \(\comp{L_{\langle \Sigma |R\rangle}}{\Lambda_{\langle \Sigma |R\rangle}}\)
     in the row \(d\) and column \(e\). When $X$ is a \(\Sigma\times m\) matrix
     and \(\phi\in \Aut(A)\), it follows that \(\comp{L_{\langle \Sigma
     |R\rangle}}{\Lambda_{\langle \Sigma |R\rangle}}{X} = \comp{L_{\langle
     \Sigma |R\rangle}}{W_{\langle \Sigma |R\rangle}}{\Gamma_\phi}\) if and only
     if for all \(d\in D\) we have 
    \[\beta_d \cdot
        \actweird{X}{\operatorname{\row}_{p_d}}+\sum_{e\in E} \beta_{d,e}\cdot \actweird{X}{\operatorname{\row}_e}=\act{\actweird{\comp{L_{\langle \Sigma |R\rangle}}{W_{\langle \Sigma |R\rangle}}}{\operatorname{\row}_d}}{\gamma_\phi^Z}.
    \]
When the entries of \(X\) in the rows from \(E\) are zero, the above equation reduces to
    \[\beta_d \cdot
        \actweird{X}{\operatorname{\row}_{p_d}}=\act{\actweird{\comp{L_{\langle \Sigma |R\rangle}}{W_{\langle \Sigma |R\rangle}}}{\operatorname{\row}_d}}{\gamma_\phi^Z}.
    \]
    
   Thus $\phi \in \actweird{A}{\Aut}$ always extends to an element $\tilde{\phi}\in \Omega$ when there is a matrix \(X\) such that the above equality holds for all
   $d\in D$. Moreover, in this case \(\actweird{X}{\operatorname{\row}_{p_d}}=(p_d)\tilde{\phi}\cdot p_d^{-1}\).
   We have shown the required equality.
   We have also shown that \(\phi\in \Aut(A)\) has an extension in \(\Omega\) if and only if for all \(d\in D\) we have
    \[\act{\actweird{\comp{L_{\langle \Sigma |R\rangle}}{W_{\langle \Sigma
    |R\rangle}}}{\operatorname{\row}_d}}{\phi^Z}\in \beta_d\cdot \Z^m +
    \actweird{\comp{L_{\langle \Sigma |R\rangle}}{W_{\langle \Sigma
    |R\rangle}}}{\operatorname{\row}_d}.\]
    To show the remaining statement that \(P_{A,G,\Omega}\) has finite index in
    \(\Aut(A)\), it thus suffices to show that the subgroup of $\actweird{\Z^m}{\Aut}$ fixing
    all the cosets of $\beta_d \cdot \Z^m$ setwise for all \(d\in D\) has finite
    index in $\actweird{\Z^m}{\Aut} \cong \actweird{A}{\Aut}$. This follows if
    we show that the subgroup of $\actweird{\Z^m}{\Aut}$ fixing all the cosets
    of $H=\left(\prod_{d\in D}\beta_d\right) \cdot \Z^m$ has finite index in
    $\actweird{\Z^m}{\Aut}$. By \cref{finite-index-stab}, the subgroup of
    $\actweird{\Z^m}{\Aut}$ fixing $H$ setwise has finite index in
    $\actweird{\Z^m}{\Aut}$. The group of automorphisms fixing all cosets is
    then a finite-index subgroup of that group, as required.
\end{proof}  

The fourth hypothesis of the previous lemma is unfortunate as it references the details of a presentation as well as the technical definitions from \cref{matrix-lem}.
In the following proposition we have our first `non-technical' result. 
From there we will be close to the main results of the section.

\begin{lem}\label{autA-is-finiteindex}
    Suppose that $G$ is a group with a finite-index characteristic subgroup $A\cong \Z^m$. Suppose further that for all $g\in G$ and $x \in A$ we have $\conjugate{x}{g}= x^{\actweird{g}{\operatorname{sgn}}}$ where $\operatorname{sgn}
      \colon G\to \{1,-1\}$ is a homomorphism. 
      Then the group \(P_{A,G}\) of automorphisms of \(A\) extendable to
      automorphisms of \(G\) has finite index in \(\actweird{A}{\Aut}\).
\end{lem}
\begin{proof}
We use the notation from \cref{matrix-lem}. If $\operatorname{sgn}$ is surjective then specifically choose $\langle
        \Sigma \mid R\rangle$ such that exactly one element of $ s\in \Sigma$ has
        \(\actweird{s}{\sgn}=-1\).        By \cref{lem:finindextend}, it suffices to show condition (4) holds.
 
  Suppose that $r\in R$ corresponds to a zero \row\ of $\comp{L_{\langle \Sigma
        |R\rangle}}{\Lambda_{\langle \Sigma |R\rangle}}$.
        We show that the corresponding
        \row\ of $\comp{L_{\langle \Sigma |R\rangle}}{W_{\langle \Sigma
        |R\rangle}}$ is also a zero \row. 
         In the case that \(s\in \Sigma\) with \(\sgn(s)=-1\) does not exist,
        for point (2) we take \(\{s\} = \varnothing\), and point (3)
        is not considered in the following definition of \(Y\).
        We define $Y\colon \Sigma^\ast \to \Z^\Sigma$ inductively by 
        \begin{enumerate}
            \item $\act{t}{Y}$ is the generator of \(\Z^\Sigma\) corresponding to  $t$
            for $t\in \Sigma$;
            \item  $\act{tu}{Y}$ is $\act{t}{Y}+\act{u}{Y}$, for
            \(t\in \Sigma \setminus \{s\}\) and \(u \in \Sigma^\ast\);
            \item  $\act{su}{Y}$ is $\act{s}{Y}-\act{u}{Y}$, for \(u \in \Sigma^\ast\).
        \end{enumerate}
        In particular, \(\act{r}{Y}\) is the \(r^{th}\) row of \(\Lambda_{\langle \Sigma |R\rangle}\) for all \(r\in R\).
        Let \(L_{\langle \Sigma |R\rangle}=[\ell_{\comp{r}{,}{r'}}]_{(\comp{r}{,}{r'})}\). We have
        \[
            \actweird{L_{\langle \Sigma |R\rangle}\Lambda_{\langle \Sigma |R\rangle}}{\operatorname{\row}_r} = \sum_{r' \in R} \ell_{r,r'} \actweird{\Lambda_{\langle \Sigma |R\rangle}}{\operatorname{\row}_{r'}}.
        \]
        As it is a product of relations of \(G\), it follows that the relation
        \[
            \prod_{r'\in R}{(r')}^{\ell_{\comp{r}{,}{r'}}} =\prod_{r'\in R} x_{r'}^{\ell_{\comp{r}{,}{r'}}}
        \]
        holds in $G$ (for any ordering on \(R\) used to define the product). We also have by the definition of matrix multiplication that 
        \[\left(\prod_{r'\in R} x_{r'}^{\ell_{\comp{r}{,}{r'}}}\right)^Z=\sum_{r'\in R} \ell_{\comp{r}{,}{r'}}\cdot x_{r'}^Z =\actweird{\comp{L_{\langle \Sigma |R\rangle}}{W_{\langle \Sigma |R\rangle}}}{\operatorname{\row}_r}.\]
        Note also that for all $r'\in R$, we have an even
        number of letters in \(r\) equal to $s$ (none, if \(s\) is undefined) since their
        product is trivial in $G/A$. 
        Thus
        \[
            \comp{\left(\prod_{r'\in R}(r')^{\ell_{\comp{r}{,}{r'}}}\right)}{Y}=\sum_{r'\in R} \ell_{\comp{r}{,}{r'}}\cdot \act{r'}{Y}=\sum_{r' \in R} \ell_{\comp{r}{,}{r'}} \actweird{\Lambda_{\langle \Sigma |R\rangle}}{\operatorname{\row}_{r'}}=\actweird{\comp{L_{\langle \Sigma |R\rangle}}{\Lambda_{\langle \Sigma |R\rangle}}}{\operatorname{\row}_r}=\mathbf{0}.
        \]
        We show that if $u\in \Sigma^\ast$ has an even number of occurrences of
        letters \(s\) (which may not be defined, but if so, there are an even
        number of occurrences) and $\act{u}{Y} = \bold{0}$ then $u$ represents the
        identity in the group $\act{G}{\sigma}$ from \cref{schur-cor}. From this
        it will follow that $\prod_{r'\in R}(r')^{\ell_{\comp{r}{,}{r'}}}
        =\prod_{r'\in R} x_{r'}^{\ell_{\comp{r}{,}{r'}}}$ is the identity of $A\leq
        G$, as required.
        
        Suppose that $u=u_0 s u_1 s u_2 \cdots s u_{2k}$ where each \(u_i\) does not
        contain \(s\), and thus if \(s\) is not defined, we have \(u =  u_0\).
        Then
        \begin{equation}
            \label{eqn_Y_0}
          \bold{0}= \act{u}{Y} = \act{u_0}{Y} - \act{u_1}{Y}+\act{u_2}{Y}-\ldots + \act{u_{2k}}{Y}
        \end{equation}
        and $u =_{\act{G}{\sigma}} u_0 u_1^{-1}u_2 u_3^{-1}\ldots u_{2k}s^{2k}=u_0
        u_1^{-1}u_2 u_3^{-1}\ldots u_{2k}$. 
        As the product $u_0 u_1^{-1}u_2
        u_3^{-1}\ldots u_{2k}$ in $\act{G}{\sigma}$ is in an abelian group and from
        \eqref{eqn_Y_0} the corresponding product in a free abelian group
        is zero, it follows that $u$ is trivial in \(\act{G}{\sigma}\),  as required.
 
\end{proof}

\cref{finite-index-aut-actions} allows us to conclude that the relative
automorphic growth of our finite-index free abelian characteristic subgroup
\(A\) is linear. The remainder of the proof that the automorphic growth of
these virtually \(\Z^m\) groups \(G\) is linear must consider the other cosets of \(A\)
in \(G\) and compute their growth. Since these are not subgroups but sets, we cannot
define their `automorphic growth'. However, they do have a metric associated with them
(albeit one that the action of \(\actweird{G}{\Aut}\) does not necessarily behave well with), and
thus we can use the action growth objects defined in \cref{action-growth-obj} to
consider these. The majority of the following proof is occupied with these
non-subgroup cosets, and computing their growth.

\begin{proposition}\label{payoff}
   Suppose that $G$ is a group with a finite-index normal subgroup $A\cong \Z^m$. Suppose further that there is a homomorphism $\operatorname{sgn}
      \colon G\to \{1,-1\}$ such that for all $g\in G$ and $x \in
      A$ we have $\conjugate{x}{g}= x^{\actweird{g}{\operatorname{sgn}}}$. 
      Then $\alpha_G$ is linear.
\end{proposition}
\begin{proof}
By \cref{virt-ab-char} we can replace \(A\) with a finite-index subgroup
characteristic in \(G\). Using the fact that finite-index subgroups are
undistorted (see, for example \cite{Mann}), \cref{relativeundistorted}, \cref{autA-is-finiteindex},
\cref{finite-index-aut-actions} and \cref{thm:freeabelian}, imply that the
relative automorphic growth of $A$ in $G$ is linear. We use the notation of
\cref{matrix-lem} using a monoid presentation $\langle \Sigma \mid R\rangle$ where
all relations in \(R\) are of the form \(r = 1\),
such that \(\Sigma\) is a transversal for \(A\) in \(G\). As such we denote
\(\Sigma\) by \(T\).

    Let $D$ be the set of elements of $T$ corresponding to non-zero \row s of $\comp{L_{\langle T |R\rangle}}{\Lambda_{\langle T |R\rangle}}$ and  let \(E\subseteq \{1,\ldots,m\}\) correspond to the set of non-pivot columns of \(\comp{L_{\langle T |R\rangle}}{\Lambda_{\langle T |R\rangle}}\).
    Consider the
    group $\Omega \leq \actweird{G}{\Aut}$ of automorphisms acting trivially on $G/A$ and fixing the set $E$
    pointwise. 
    If the $r^\textrm{th}$ \row\ of $\comp{L_{\langle T| R\rangle}}{\Lambda_{\langle T| R\rangle}}$ is a zero \row\ then the
    corresponding \row\ of $\comp{L_{\langle T| R\rangle}}{W_{\langle T| R\rangle}}{\Gamma_\phi}$ must be zero when $\comp{L_{\langle T| R\rangle}}{\Lambda_{\langle T| R\rangle}}{X}
    =\comp{L_{\langle T| R\rangle}}{W_{\langle T| R\rangle}}{\Gamma_\phi}$ has a solution $X$. 
    It thus follows from \cref{matrix-lem} that 
    \[\comp{\actweird{\comp{L_{\langle T| R\rangle}}{W_{\langle T| R\rangle}}}{\operatorname{\row}_r}}{\Gamma_\phi}=\mathbf{0}\]
    whenever $\phi\in \actweird{A}{\Aut}$ has an extension in \(\actweird{G}{\Aut}\).
    Thus from the definition of \(\Gamma_\phi\), the element of \(A\) corresponding to \(\actweird{\comp{L_{\langle T| R\rangle}}{W_{\langle T| R\rangle}}}{\operatorname{\row}_r}\) is fixed by the group of elements of \(\actweird{A}{\Aut}\) which extend to \(\actweird{G}{\Aut}\).
      In other words, \(P_{A,G}\) is contained in the stabiliser in \(\Aut(A)\) of the point corresponding to \(\actweird{\comp{L_{\langle T| R\rangle}}{W_{\langle T| R\rangle}}}{\operatorname{\row}_r}\).
    From \cref{autA-is-finiteindex}, it thus follows that the stabiliser of the element \(\actweird{L_{\langle T| R\rangle}W_{\langle T| R\rangle}}{\operatorname{\row}_r}\) in \(\GL_{m}(\Z)\) has finite index whenever \(\actweird{\comp{L_{\langle T| R\rangle}}{\Lambda_{\langle T| R\rangle}}}{\operatorname{\row}_r}=\mathbf{0}\). 
    Thus if \(\actweird{\comp{L_{\langle T| R\rangle}}{\Lambda_{\langle T| R\rangle}}}{\operatorname{\row}_r}=\mathbf{0}\), we must have \(\actweird{L_{\langle T| R\rangle}W_{\langle T| R\rangle}}{\operatorname{\row}_r}=\bold{0}\). We can therefore apply \cref{lem:finindextend} and conclude that the group \(P_{A,G,\Omega}\leq \Aut(A)\) of automorphisms with extensions in \(\Omega\) has finite index and whenever \(\phi\in \Aut(A)\) has an extension \(\tilde{\phi}\in \Omega\), we have 
    \begin{equation}\label{eql}
        \beta_d \cdot
       ((p_d)\tilde{\phi}\cdot p_d^{-1})^Z=\act{\actweird{\comp{L_{\langle T |R\rangle}}{W_{\langle T |R\rangle}}}{\operatorname{\row}_d}}{\gamma_\phi^Z}
    \end{equation}
        for all \(d\in D\).

 We show that $\alpha_{G^\Omega}$ is
    linear (this is a stronger claim than $\alpha_{G}$ being linear).
   To do this, we show that $\alpha_{(At)^\Omega\subseteq G}$ is at most linear for all $t\in T$.        
    Let $t\in E$. 
    Then $\act{t}{\tilde{\phi}}\cdot t^{-1}=1_G$ for all
    $\tilde{\phi}\in \Omega$. Thus for all $x\in A$, the orbit of $xt$ under
    $\Omega$ is equal to the orbit of $x$ under $\Omega$ shifted by $t$.
    Thus $\alpha_{(At)^\Omega\subseteq G}=\alpha_{A^\Omega\subseteq G}$ is linear using \cref{finite-index-aut-actions} and \cref{relativeundistorted}.
    
    Let $t\in T\setminus E$ and \(d\in D\) be the row whose pivot is in column \(t\). In particular, \(t=p_d\).
    For all $\phi\in \actweird{A}{\Aut}$ with extensions in $\Omega$, let $\phi_t\in \actweird{A}{\Sym}$ be such that \(\phi\) and
    \(xt\mapsto \act{x}{\phi_t} \cdot t\) are restrictions of a common element
    \(\tilde{\phi}\in\Omega\). Let \(\gamma_{\tilde{\phi}} \colon G\to G\) be the
    map \(g\mapsto \act{g}{\tilde{\phi}} \cdot g^{-1}\). Then 
    \[
        \act{x}{\phi_t}\cdot t=\act{x}{\tilde{\phi}}\act{t}{\tilde{\phi}}=\act{x}{\phi}\cdot \act{t}{\gamma_{\tilde{\phi}}}\cdot
        t.
    \]
    so \(\act{x}{\phi_t}=\act{x}{\phi}\cdot \act{t}{\gamma_{\tilde{\phi}}}=(x)\phi\cdot (t)\tilde{\phi}\cdot t^{-1}\).
    From \eqref{eql}, it then follows that for all \(x\in A\) we have 
    \[\act{x^Z}{\phi_t^Z}=
    \act{x^Z}{\phi^Z} + \frac{1}{\beta_d}\cdot \actweird{\comp{L_{\langle T |R\rangle}}{W_{\langle T |R\rangle}}{\Gamma_{\phi}}}{\operatorname{\row}_d}.\] 
    The map \(\phi^Z\in \Aut(\Z^m)\) has a unique linear extension in \(\Aut(\frac{1}{\beta_d}\Z^m)\).
    Therefore, we can use the above equality to extend the map \(\phi_t^Z\) to a permutation of \(\frac{1}{\beta_d}\mathbb{Z}^m\).
    We have $\alpha_{(At)^\Omega\subseteq G}=\alpha_{(A,P_{A,G,t})}$  where
    \(P_{A,G,t} \coloneqq \makeset{\phi_{t}}{$\phi\in P_{A,G,\Omega}$}\). 
    Let \(f_t \coloneqq \actweird{\comp{L_{\langle T |R\rangle}}{W_{\langle T |R\rangle}}}{\operatorname{\row}_d}\).
    Let $\rho_t\colon \frac{1}{\beta_d}\mathbb{Z}^m\to \frac{1}{\beta_t}\mathbb{Z}^m$ be the map $\act{x}{\rho_t}=x+\frac{1}{\beta_d}\cdot f_{t}$. For all $\phi\in P_{A,G,\Omega}$ and $x\in \frac{1}{\beta_d}\mathbb{Z}^m$ we obtain
    \[
        \act{x}{\phi_{t}^Z}=\act{x}{\phi^Z} + \frac{1}{\beta_d}\cdot f_{t}\Gamma_{\phi}=\act{x}{\phi^Z}+ \frac{1}{\beta_d}\cdot (\act{f_{t}}{\phi^Z}-f_{t})=\act{x +\frac{1}{\beta_d}\cdot f_{t}}{\phi^Z} - \frac{1}{\beta_d}\cdot f_{t}=\act{x}{\comp{\rho_t}{\phi^Z}{\rho_t^{-1}}}.
    \]
    The displayed equation says 
    that $P_{A,G,t}=\comp{\rho_t}{P_{A,G,\Omega}}{\rho_t^{-1}}$ when both groups are viewed as acting on \(\frac{1}{\beta_d}\mathbb{Z}^m\).
    As \(P_{A,G,\Omega}\) was a finite-index subgroup of \(\actweird{\Z^m}{\Aut}\), \cref{finite-index-stab} implies that \(P_{A,G,\Omega}\) can be viewed as a finite-index subgroup of \(\actweird{\frac{1}{\beta_d}\Z^m}{\Aut}\).
    It follows from \cref{actobjiso}, \cref{finite-index-aut-actions}, and \cref{thm:freeabelian} that \(\alpha_{(\frac{1}{\beta_d}\Z^m,P_{A,G,\Omega})}\) is linear with respect to the standard metric on \((\frac{1}{\beta_d}\Z^m,+)\) which is quasi-isometric to the Cayley metric.  \cref{conjugate-aut-groups} then implies that \(\alpha_{(\frac{1}{\beta_d}\Z^m,P_{A,G,t})}\) is linear. As \(\alpha_{(\frac{1}{\beta_d}\Z^m,P_{A,G,t})}\succcurlyeq \alpha_{(\Z^m,P_{A,G,t})}\), we have thus shown that \(\alpha_{(At)^\Omega\subseteq G}\) is at most linear for all \(t \in T\).

    To see that \(\alpha_G\) is linear, we first note that since
    \(\Omega \leq \actweird{G}{\Aut}\), we have that \(\alpha_G \preccurlyeq \alpha_{G^\Omega}\).
    Thus it suffices to show that \(\alpha_{G^\Omega}\) is linear. Since \(\Omega\)
    fixes \(At\) setwise for each \(t \in T\), we have that \(\alpha_{G^\Omega}
    = \sum_{t \in T} \alpha_{(At)^\Omega \subseteq G}\), and thus
    \(\alpha_{G^\Omega}\) is linear, as required.
\end{proof}

We next fully classify the automorphic growth of virtually abelian groups of rank at
most \(2\).
\begin{theorem}\label{VAmain1}
    Suppose that $G$ is virtually abelian with a finite-index normal subgroup $A\cong \Z^m$.
    \begin{enumerate}
        \item If \(m=0\) then $\alpha_G$ is constant.
        \item If \(m=1\) then $\alpha_G$ is linear.
        \item If \(m=2\) then $\alpha_G$ is linear or quadratic. In this case, the growth of \(G\) is linear if and only if there is a homomorphism $\operatorname{sgn}
      \colon G\to \{1,-1\}$ such that for all $g\in G$ and $x \in
      A$ we have $\conjugate{x}{g}= x^{\actweird{g}{\operatorname{sgn}}}$.
    \end{enumerate}
\end{theorem}
\begin{proof}
Follows immediately from \cref{VirtuallyZ2Quad} and \cref{payoff}.
\end{proof}

The purpose of the following construction is to show that there is a virtually
abelian group rank \(r\) of automorphic growth rate of every non-constant
polynomial rate up to degree \(r\). We begin by constructing a virtually abelian
group of rank \(r\) and automorphic growth rate polynomial of degree \(r\). 
\begin{proposition}
\label{va-polynomial-r}
    Let \(r \in \mathbb{N}\setminus \{0\}\). Then the group 
    \[
        G = \makepres{ a_1, \ldots, a_r, t_1, \ldots, t_r}{\(a_i a_j = a_j a_i,
        t_i t_j = t_j t_i, t_i^2 = 1, \conjugate{a_j}{t_i} = \begin{cases}
            a_j & i \neq j \\
            a_i^{-1} & i = j
        \end{cases} \)}.
    \]
    is virtually abelian of rank \(r\), and has automorphic growth function
    \(\alpha_G \sim (n \mapsto n^r)\).
    Furthermore, \([G, G] = \langle a_1^2, \ldots, a_r^2 \rangle\).
\end{proposition}

\begin{proof}
   First note that \(G = \langle a_1, \ldots, a_r \rangle \rtimes \langle t_1,
   \ldots, t_r \rangle\), and that \(G / \langle a_1, \ldots, a_r \rangle \cong
   C_2^r\), and so \(G\) is indeed virtually abelian of rank \(r\). We next show
   that \( \langle a_1^2, \ldots, a_r^2 \rangle = [G, G]\). We have that
   \(a_i^2\) is trivial in the abelianisation, as the relation \(\conjugate{a_i}{t_i} = a_i^{-1}\) abelianises to the relation \(a_i = a_i^{-1}\).  We can write any commutator as \([as, bt]\) for some \(a,b \in \langle
   a_1, \ldots, a_r \rangle\) and \(s, t \in \langle t_1, \ldots, t_r \rangle\).
   Using \(\phi_s, \phi_t \in \actweird{\langle a_1, \ldots, a_r \rangle}{\Aut}\) to denote
   the conjugation actions of \(s\) and \(t\) on \(\langle a_1, \ldots, a_r
   \rangle\), we have that 
   \[[as, bt] =\comp{s^{-1}}{a^{-1}}{t^{-1}}{b^{-1}}{a}{s}{b}{t}= \comp{\act{a^{-1}}{\phi_s}}{\act{b^{-1}}{\phi_{st}}}{\act{a}{\phi_{st}}}{\act{b}{\phi_t}}=\comp{\act{a^{-1}}{\phi_s}}{\act{a}{\phi_{st}}}{\act{b^{-1}}{\phi_{st}}}{\act{b}{\phi_t}}.\] 
   Since \(\phi_s\) and \(\phi_t\) both
   just invert certain generators in \(\langle a_1, \ldots, a_r \rangle\) and
   fix the rest, any generators that are inverted will double up with those in
   \(a\) and \(b\), and any that are fixed will cancel with those in \(a\) and
   \(b\). Thus the resulting element will be generated by squares of elements,
   and so \([G, G] \subseteq \langle a_1^2, \ldots, a_r^2 \rangle\). We have
   thus shown that \([G, G] = \langle a_1^2, \ldots, a_r^2 \rangle \).
   
   Thus \(A \coloneqq \langle a_1^2, \ldots, a_r^2 \rangle\) is a free abelian
   characteristic subgroup of \(G\). Since \(A\) has finite index in \(\langle
   a_1, \ldots, a_r \rangle\), and this has finite index in \(G\), it follows
   that \(A\) has finite index in \(G\). Let \(P_{A, G} \leq \actweird{A}{\Aut}\) be the
   set of automorphisms of \(A\) that extend to automorphisms of \(G\). By
   \cref{lem:charsbgrp}, \(\alpha_{A^{P_{A, G}} \subseteq G} \preccurlyeq
   \alpha_G\). Since \(A\) has finite index in \(G\), it is undistorted (see for
   example \cite{Mann}), and thus \(\alpha_{A^{P_{A, G}}} \sim \alpha_{A^{P_{A,
   G}} \subseteq G} \preccurlyeq \alpha_G\). Since \(G\) has standard growth
   that is polynomial of degree \(r\), it follows that \(\alpha_G \preccurlyeq
   (n \mapsto n^r)\), and so it suffices to show that \((n \mapsto n^r)
   \preccurlyeq \alpha_{A^{P_{A, G}}}\).

   Let \(\bar{F}_{A, G} \leq \actweird{A}{\Aut}\) be the group of automorphisms of \(A\)
   induced by the conjugation action of \(G\) on \(A\). Then \cref{barF_Aintro} tells us
   that \(P_{A, G} \leq \actweird{\bar{F}_{A,G}}{N_{\actweird{A}{\Aut}}}=:N\), and so
   \(\alpha_{A^{N}} \preccurlyeq \alpha_{A^{P_{A, G}}}\). 
   Next,
   \cref{norm-cent} tells us that \(C:=\actweird{\bar{F}_{A, G}}{C_{\actweird{A}{\Aut}}}\) is a finite
   index subgroup of \(\actweird{\bar{F}_{A, G}}{N_{\actweird{A}{\Aut}}}\). 
    Then \cref{finite-index-aut-actions} tells us
   that \(\alpha_{A^C} \sim \alpha_{A^{N}}
   \preccurlyeq \alpha_{A^{P_{A, G}}}\). Hence it suffices to show that
   \((n \mapsto n^r) \preccurlyeq \alpha_{A^C}\).

   We next compute \(C\). When written as matrices, the generators of
   \(\bar{F}_{A, G}\) are diagonal matrices with one occurrence of \(-1\) and
   \(1\)s elsewhere on the diagonal. Write \(M_i\) for one of these matrices
   with \(-1\) in the \((i, i)\)th entry (corresponding to the generator
   \(t_i\)). If \(X = [x_{kl}] \in \actweird{\Z}{\GL_r} \cong \actweird{A}{\Aut}\), then if
   \(\delta_{ij}\) denotes the Kronecker delta function applied to \((i, j)\),
   \(M_i X = [(-1)^{\delta_{li}} x_{kl}]\)
   and
   \(X M_i = [(-1)^{\delta_{ki}} x_{kl}]\).
   . These can only be equal
   if for all \(k\) and \(l\) such that exactly one of \(k\) and \(l\)
   is \(i\), we have \(x_{kl} = 0\). Thus if \(X\) commutes with all
   \(M_i\)s, it follows that \(X\) is necessarily diagonal. Since
   \(\actweird{X}{\det} \in \{-1, 1\}\), and \(\actweird{X}{\det}\) is the product of its
   diagonal entries, all of its diagonal entries are either \(1\) or
   \(-1\). There are finitely many choices for these, and so \(C\) is
   finite.

   We can thus use \cref{finite-index-aut-actions} to conclude
   that \(\alpha_{A^C} \sim \alpha_{A^{\{\id\}}}\). Since the latter
   is just the standard growth of \(A\), we have \(\alpha_{A^C}
   \sim \alpha_{A^{\{\id\}}} \sim (n \mapsto n^r)\), as required.
\end{proof}

We conclude the section by showing that when considering (exact) polynomial
growth rates, all non-constant possibilities occur for polynomials of lower degree than
standard growth.

\begin{theorem}\label{allpossibleVA}
    Let \(G\) be the rank \(r\) virtually abelian group from
    \cref{va-polynomial-r} and let \(s \in \N\setminus \{0\}\). Then \(\alpha_{G \times
    \Z^s} \sim (n \mapsto n^{r + 1})\). 
    In particular, for every \(m , k \in \mathbb{N}\setminus \{0\}\) with \(m \geq k\),
    there is a virtually abelian group with rank \(m\) and automorphic growth
    rate \(n \mapsto n^k\).
\end{theorem}
\begin{proof}
    We show that \(G\) and \(\Z^s\) are both characteristic in \(G \times \Z^s\),
    which will allows us to apply \cref{dir-prod}.
    First note that \(\actweird{G \times \Z^s}{Z} = \actweird{G}{Z} \times \actweird{\Z^s}{Z} = \Z^s\), and so
    \(\Z^s\) is characteristic in \(G\). 
    We next compute the commutator
    subgroup. If \(g_1 h_1, g_2 h_2 \in G \times \Z^s\), then \([g_1h_1, g_2
    h_2] = [g_1, g_2]\), and so \([G \times \Z^s, G \times \Z^s] = [G, G]\).
    \cref{va-polynomial-r} tells us that \([G, G] = \langle a_1^2, \ldots, a_r^2
    \rangle\). It follows that \(A \coloneqq \langle a_1^2, \ldots, a_r^2
    \rangle\) is also characteristic in \(G \times \Z^s\). Since \(A\)
    has finite index in \(G\), there
    exists \(p > 0\) (the exponent of \(G / A\)) such that for all \(g \in G\),
    \(g^p \in A\). Since \(A\) is characteristic in \(G \times \Z^s\), it
    follows that if \(g \in G\) and \(\phi \in \actweird{G \times \Z^s}{\Aut}\), then \((\act{g}{\phi})^p \in A\). Since \(\Z^s\) is
    torsion-free, this implies that \(\act{g}{\phi} \in G\), and so \(G\) is
    characteristic in \(G \times \Z^s\).

    We have now shown that both \(G\) and \(\Z^s\) are characteristic in
    \(G \times \Z^s\), and so \cref{dir-prod} tells us that
    \(\alpha_{G \times \Z^s} \sim \alpha_G \cdot \alpha_{\Z^s}\). \cref{va-polynomial-r} tells us that \(\alpha_G \sim (n \mapsto n^r)\) and
    \cref{thm:freeabelian} tells us that \(\alpha_{\Z^s} \sim (n \mapsto n)\).
    Thus \(\alpha_{G \times \Z^s} \sim (n \mapsto n^{r + 1})\), as required.

    For the final claim, note that if \(k = 1\), then \(\Z^m\) is an example
    (by \cref{thm:freeabelian}), and if \(k > 1\), then \(G \times \Z^s\),
    with \(r + 1 = k\) and \(r + s = m\) is an example.
\end{proof}

\section{The Heisenberg group}
\label{sec:Heis}
    This section covers the Heisenberg group. Usually considered be the
    `easiest' example of a nilpotent group that is not virtually abelian, it
    nonetheless exhibits various `strange' behaviour. In the case of growth it
    is the first (and only) example we consider where the standard, conjugacy
    and automorphic growth rates are all distinct. The standard growth rate is
    quartic \cite{bass}, the conjugacy is \(n^2\log n\)
    \cite{Babenko} and the automorphic, as we see below, is quadratic.

    \begin{theorem}
        \label{thm:Heis}
        The automorphic growth of the Heisenberg group is quadratic.
    \end{theorem}

    \begin{proof}
        Recall that the Heisenberg  group \(\actweird{\Z}{H}\) is defined by the presentation
        \[\langle a, b, c \mid c = [a, b], c \text{ is central} \rangle.\] 
        All lengths of elements in this proof will be with respect to the
        generating set \(\{a, b\}\).
        We will show that the automorphic growth is quadratic. 
        To do this, we will show that:
        \begin{enumerate}
            \item for all \(g\in \actweird{\Z}{H}\setminus [\actweird{\Z}{H}, \actweird{\Z}{H}]\) with \(|g|\leq
            n\), there are \(i,k\leq n\) such that \(g\) is in the same
            automorphic orbit as \(a^ic^k\);
            \item  none of the elements of \(\makeset{c^k}{\(k\in \N\)}\) belong to
            the same automorphic orbit and the relative standard growth of \([\actweird{\Z}{H},
            \actweird{\Z}{H}]\) in \(\actweird{\Z}{H}\) is quadratic. 
        \end{enumerate}
        From (1), it will follow that every automorphic orbit outside the commutator subgroup of length at most \(n\) is in
        the orbit of one of the elements \(a^ic^k\) and hence there are at most \(n^2\)
        such orbits. From (2) it will follow that the relative automorphic growth in \([\actweird{\Z}{H},\actweird{\Z}{H}]\) is quadratic and hence the automorphic growth of \(\actweird{\Z}{H}\) is at least quadratic. Together these show that the automorphic growth is exactly quadratic as required.
        
        We start with (2). These elements are not in the same automorphic orbit
        as they belong to the commutator subgroup, which is characteristic and
        isomorphic to \(\Z\).
        The exact length of \(c^k\) is  \(2 \lceil
        2\sqrt{|k|} \rceil\) by \cite{Blachere}*{Theorem 2.2}. This implies
        (2).
                
        It remains to show (1).
        Consider the matrices
        \[A=\begin{bmatrix}
            -1 &0\\
            0 & 1
        \end{bmatrix} \quad \text{ and }\quad B=\begin{bmatrix}
            1 & \ifthenelse{\boolean{rightactions}}{0}{1}\\
            \ifthenelse{\boolean{rightactions}}{1}{0} & 1
        \end{bmatrix}.\]
        We show that for each \(M =[m_{i,j}]\in \{A,B\}\) the map \(\phi_M\)
        defined by \(\act{a}{\phi_M} = a^{m_{1,1}}b^{m_{\comp{1}{,}{2}}}\), \(\act{b}{\phi_M} =
        a^{m_{\comp{2}{,}{1}}}b^{m_{2,2}}\), and \(\act{c}{\phi} = c^{\actweird{M}{\operatorname{det}}}\) is
        a well-defined automorphism. We show that \(\phi_A\) and \(\phi_B\)
        preserve the relations of the group in each case. It is well known, (see,
        for example \cite{MR4982576}*{Lemma 2.3}) that for all \(g,h\in \actweird{\Z}{H}\)
        we have \([g^{-1},h]=[g,h^{-1}]=[g,h]^{-1}\). In particular, for
        \(\phi_A\) we see that \([a^{-1},b]=c^{-1}=c^{\actweird{A}{\det}}\) and for
        \(\phi_B\) we see that
        \([a,ab]=a^{-1}b^{-1}a^{-1}aab=a^{-1}b^{-1}ab=[a,b]=c\).
        We also need to show that \(\phi_A\) and \(\phi_B\) are bijections.
        Note that \(\act{a}{\phi_A\phi_A}=\act{a^{-1}}{\phi_A}=a\),  \(\act{b}{\phi_A\phi_A}=\act{b}{\phi_A}=b\), and \(\act{c}{\phi_A\phi_A}=\act{c^{-1}}{\phi_A}=c\).
        Thus \(\phi_A^2\) is an endomorphism fixing all the generators of \(\actweird{\Z}{H}\).
        In particular \(\phi_A^2\) is the identity and \(\phi_A\) is a self-inverse bijection.
        For \(\phi_B\), consider \(\phi_{B^{-1}}\), defined using the
        matrix \(B^{-1} = \begin{bmatrix}
            1 & \ifthenelse{\boolean{rightactions}}{0}{-1} \\ \ifthenelse{\boolean{rightactions}}{-1}{0} & 1
        \end{bmatrix}\) in the same manner as \(\phi_A\) and \(\phi_B\).
        Then
        \begin{align*}
            \act{\act{a}{\phi_B}}{\phi_{B^{-1}}} & = \act{a}{\phi_{B^{-1}}} = a, \quad
            \act{\act{a}{\phi_{B^{-1}}}}{\phi_B} = \act{a}{\phi_B} = a \\
            \act{\act{b}{\phi_B}}{\phi_{B^{-1}}} &= \act{ab}{\phi_{B^{-1}}} = aa^{-1} b = b\\
            \act{\act{b}{\phi_{B^{-1}}}}{\phi_B} & = \act{a^{-1}b}{\phi_B} = a^{-1} ab = b \\
            \act{\act{c}{\phi_B}}{\phi_{B^{-1}}} & = c = \act{\act{c}{\phi_{B^{-1}}}}{\phi_B}.
        \end{align*}
        Thus \(\phi_{B^{-1}}\) is the inverse of \(\phi_B\), which is thus a
        bijection.  
        Note that the
        induced actions of \(\phi_A\) and \(\phi_B\) on the group
        \(\actweird{\Z}{H}/[\actweird{\Z}{H},\actweird{\Z}{H}]\cong \Z^2\) are the usual actions of the matrices
        \(A\) and \(B\) acting on the right. 
        As \(A\) and \(B\) generate the group \(\actweird{\Z}{\GL_{2}}\),
        it follows that for all \(M\in \actweird{\Z}{\GL_2}\) there is an automorphism
        \(\phi_M\) of \(\actweird{\Z}{H}\) which acts on the abelianisation as \(M\) does.

        We know from the proof of \cref{thm:freeabelian}, that for all \((i, j)
        \in \Z^2 \setminus \{(0, 0)\}\), there is a matrix \(M\), such that
        \(\act{i, j}{M} = (\actweird{i, j}{\gcd}, 0)\). 
        We have thus shown that for all \(a^i b^j
        c^k \in \actweird{\Z}{H}\setminus [\actweird{\Z}{H},\actweird{\Z}{H}]\) there exists \(\phi \in
        \act{\actweird{\Z}{H}}{\Aut}\) with \(\act{a^i b^j c^k}{\phi} = a^{\actweird{i, j}{\gcd}} c^{k'}\), for
        some \(k' \in \Z\). We can then apply the inner automorphism defined by
        \(x \mapsto \conjugate{x}{b}\) to map \(a^{\actweird{i, j}{\gcd}} c^{k'}\) to \(a^{\actweird{i,
        j}{\gcd}} c^{k' + \actweird{i,
        j}{\gcd}}\). If we repeatedly apply this inner
        automorphism or its inverse, we can reduce \(k'\) modulo \(|\actweird{i,
        j}{\gcd}|\), which in particular shows that there is an automorphism taking
        \(a^i b^j c^k\) to \(a^{\actweird{i,
        j}{\gcd}} c^{k''}\), for some \(k'' \in \Z\)
        satisfying \(0 \leq k'' < \actweird{i,
        j}{\gcd}\). 
        
        Let \(g\in \actweird{\Z}{H}\setminus [\actweird{\Z}{H},\actweird{\Z}{H}]\) have
        length at most \(n\). Thus \(g=g_1g_2\ldots g_n\) where \(g_i\in
        \{a,b,a^{-1},b^{-1}\}\) for all \(i\leq n\). As \([a,b]=c\),
        \([a^{-1},b]=c^{-1}\), \([a,b^{-1}]=c^{-1}\), \([a^{-1},b^{-1}]=c\), it
        follows that there are \(i_g,j_g,k_g \in \Z\) with
        \(g=a^{i_g}b^{j_g}c^{k_g}\) where \(|i_g|,|j_g|\leq n\). As \(g\) does
        not belong to the commutator subgroup, either \(i_g\neq 0\) or \(j_g\neq
        0\). In particular, \(\actweird{i_g,j_g}{\gcd}\) is well-defined and at most
        \(\actweird{|i_g|,|j_g|}{\max}\). Thus using the earlier discussion, there exists
        \(g'\) in the automorphic orbit of \(g\) such that 
        \(g'=a^ic^k\) where \(0\leq i\leq n\) and \(k\leq i\), as required.
    \end{proof}

\section{Free groups}

In this section we see our first groups with exponential automorphic growth. Two
words in a free group are conjugate precisely if they have the same `cyclic
reduction', for example \(a^{-1}bcabcabcabcaa\sim abcabcabcabc\). As such, it is easy to
see the conjugacy growths of non-abelian free groups are exponential. If words
have the same cyclic reduction then they must be automorphic so we will often
think of words in this section only up to cyclic reduction. The automorphism
groups of finite rank free groups have been highly studied and are famously
complex \cite{autfreehard}. It is known, however that these automorphism groups
are finitely presented \cite{finitepresautfree}. Moreover, Whitehead provides us
with an algorithm for breaking down an element of a free group to a minimal
length word \cite{whitehead_algm}. The existence of this algorithm is central to
the arguments in this section.

\label{sec:free}
  \begin{dfn}
    A \textit{Whitehead automorphism} of a free group on a set \(\Sigma\), is an
    automorphism defined by a permutation of \(\Sigma \cup \Sigma^{-1}\), or an
    automorphism that for some fixed \(a \in \Sigma \cup \Sigma^{-1}\) takes
    each \(x \in \Sigma\) to (a possibly distinct) one of \(x\), \(xa\),
    \(a^{-1} x\) or \(a^{-1} x a\). In any case, the Whitehead automorphism
    fixes \(a\).
  \end{dfn}

These Whitehead automorphisms can be used to efficiently reduce any word (see \cite{CGT}*{Proposition
      4.17}).
As such,
understanding the effects of these automorphisms on large enough families of words will be sufficient for our purposes. 
We first find a large family of automorphically minimal words.

  \begin{lem}
    \label{free-gp-automorphically-minimal}
      Let \(F\) be a finitely generated non-abelian free group, where \(a\) and
      \(b\) are distinct letters and lie in a free basis \(\Sigma\). Then
      cyclically reduced elements of the form \(a^{i_1} b^{j_1} \cdots a^{i_k}
      b^{j_k}\), where \(i_1, j_1 , \ldots, i_k, j_k \in \Z\), \(|i_\ell| \geq
      2\) and \(|j_\ell| \geq 2\) for all \(\ell\), are automorphically minimal
      with respect to \(\Sigma\).
  \end{lem}

  \begin{proof}
      Let \(g \coloneqq a^{i_1} b^{j_1} \cdots a^{i_k} b^{j_k}\) be cyclically
      reduced. Suppose for a contradiction that \(g\) is not automorphically
      minimal. Then Whitehead's Algorithm (see for example \cite{CGT}*{Proposition
      4.17}), tells us that there is a
      Whitehead automorphism \(\phi\) such that the length of \(\act{g}{\phi}\) after
      being cyclically reduced is strictly less than that of \(g\). If \(\phi\)
      is defined by a permutation of the basis of \(F\), then \(|\act{g}{\phi}| =
      |g|\) and \(\act{g}{\phi}\) is cyclically reduced, a contradiction. Thus we can
      assume \(\phi\) is one of the other sort of Whitehead automorphism; that
      is defined using a fixed generator (or inverse) of \(F\). 
      
      Note that if the generator \(c\) of \(F\) used to define \(\phi\) 
      is not \(a\) or \(b\), or an inverse of either, then 
      the length of the cyclically reduced form of \(\act{g}{\phi}\)
      is greater than or equal to \(|g|\), a contradiction. Thus
      \(c \in \{a, b, a^{-1}, b^{-1}\}\), and without loss of
      generality, we assume \(c = a\).
      
      Note that \(\act{a}{\phi} = a\). Thus the four remaining cases can be
      expressed by 
      \[\act{b}{\phi} \in \{b, ba, a^{-1}b, a^{-1} ba\}.\] If \(\act{b}{\phi} =
      b\), then \(\phi\) fixes all cyclic conjugates of \(g\), a contradiction.
      If \(\act{b}{\phi}= a^{-1} ba\), then \(\act{g}{\phi} = a^{i_1 -1} b^{j_1} a^{i_2}
      b^{j_2} \cdots a^{i_k} b^{j_k} a\). After cyclically reducing \(\act{g}{\phi}\),
      the resulting word will have the same length as \(g\), a contradiction. Now
      suppose \(\act{b}{\phi} = ba\). Then \(\act{g}{\phi} = a^{i_1} (ba)^{j_1} \cdots
      a^{i_k} (ba)^{j_k}\). The only cyclic cancellation that can occur in this
      expression is one of the following:
      \begin{itemize}
          \item For any \(z \in \{1, \ldots, k\}\), if \(j_z > 0\) and \(i_{z +
          1} < 0\), then between the last \(a\) in \((ba)^{j_z}\) and first
          \(a^{-1}\) in \(a^{i_{z + 1}}\), where \(z + 1\) is taken to be \(1\)
          if \(z = k\);
          \item For any \(z \in \{1, \ldots, k\}\), if \(j_z < 0\) and \(i_{z} > 0\), then between the first \(a^{-1}\) in \((ba)^{j_z}\) and the
          last \(a\) in \(a^{i_{z}}\).
      \end{itemize}
      In both cases, at most one \(a\) can be cancelled within each \((ba)\).
      Since \(|i_z| \geq 2\) and \(|j_z| \geq 2\), that means that no
      \(a^{i_z}\) or \(b^{j_z}\) term can be fully cancelled, and so this
      cancellation can never lead to further cancellation. Thus there are at
      most \(k\) cyclic cancellations in this expression. Since this expression
      has length \(\sum_{z = 1}^k (|i_z| + 2|j_z|)\) and \(|j_z| \geq 2\), this
      means that
      \[
        |\act{g}{\phi}|_{cyc} \geq -2k + \sum_{z = 1}^k (|i_z| + 2|j_z|) \geq
        \sum_{z = 1}^k (|i_z| + |j_z|) = |g|,
      \]
      a contradiction. We thus have one remaining case: \(\act{b}{\phi} = a^{-1} b\),
      but this case is analogous to the \(\act{b}{\phi} = ba\) case. Since every
      case produces a contradiction, we can conclude that \(g\) is indeed
      automorphically minimal.
  \end{proof}

We now have our large family of minimal words. 
However, this is not necessarily sufficient as some of these words could be automorphically equivalent.
We use similar techniques again to show that this is not that case (at least for a large subfamily).

  \begin{lem}
        \label{free-gp-aut-norm-form}
      Let \(F\) be a finitely generated non-abelian free group, where \(a\) and
      \(b\) are distinct generators in some basis \(\Sigma\). 
      Let \(g \coloneqq
      a^{i_1} b^{j_1} \cdots a^{i_k} b^{j_k}, h \coloneqq  a^{p_1} b^{q_1}
      \cdots a^{p_l} b^{q_l}\), where \(k, l > 0\), \(i_1, j_1, p_1, q_1 \ldots,
      i_k\), \(j_k, p_l, q_l \in \Z\), \(|i_z| \geq 3\), \(|j_z| \geq 3\), \(|p_z|
      \geq 3\) and \(|q_z| \geq 3\) for all \(z\). Then \(g\) and \(h\) are
      automorphic if and only if one can be obtained from the other using only a
      cyclic permutation together with an automorphism permuting \(\{a, b,
      a^{-1}, b^{-1}\}\).
  \end{lem}

  \begin{proof}
      \((\Leftarrow)\): If \(h\) can be obtained from \(g\) by applying
      cyclic permutations and a permutation automorphism, 
      then \(g\) and \(h\) are automorphic.
  
      \((\Rightarrow)\): Suppose \(g\) and \(h\) are automorphic.
      \cref{free-gp-automorphically-minimal} tells us that \(g\) and \(h\) are
      automorphically minimal with respect to \(\Sigma\). Thus by Whitehead's
      Algorithm, (see for example \cite{CGT}*{Proposition 4.17}), there is a
      finite sequence of Whitehead automorphisms \(\phi_1, \ldots, \phi_r\) and
      inner automorphisms 
      \(\psi_0, \psi_1, \ldots, \psi_r\), such that \(\act{g}{\comp{\psi_0}{\phi_1}{\psi_1}{\cdots}{\phi_r}{\psi_r}} = h\) and \(\act{g}{\comp{\psi_0}{\phi_1}{\cdots}{\phi_s}{\psi_s}}\) is automorphically minimal for all \(s \leq r\). We now
      proceed in a similar fashion to \cref{free-gp-automorphically-minimal} by
      checking which Whitehead automorphisms preserve the length of the cyclic
      conjugates of \(g\). That is, we take as input a cyclic permutation of
      \(g\) and apply a non-permutation Whitehead automorphism \(\phi\) that
      preserves the length up to cyclic permutation, and then show that the
      obtained element is (up to cyclic permutation) still in the same form as
      \(g\) and \(h\), and in fact we have just performed a cyclic permutation.
      We start by excluding the case of \(\phi\) a permutation of the generators and
      their inverses, and consider this case later. 


      So suppose \(\phi\) is a Whitehead automorphism that is not a
      permutation that preserves the length of some cyclic conjugate of \(g\),
      up to cyclic permutation.
      Further suppose for a contradiction that \(\phi\) does not fix (setwise) 
      the set of cyclic conjugates of \(g\).
      Then \(\phi\) is defined using some fixed generator \(c\) of \(F\).
      If \(c \notin \{a, b, a^{-1}, b^{-1}\}\), then \(\phi\) can only
      map \(g\) (or any cyclic conjugate  of it) to one of \(g\), \(c^{-1} gc\),
      or some element that can be obtained by mapping some \(x \in \{a, b\}\) to
      \(x c\) or \(c^{-1} x\). In such a case, the only way the obtained
      element can have the same length as \(g\) is if every \(c\) is cancelled,
      in which case it is precisely the element we started with. Thus
      \(c \in \{a, b, a^{-1}, b^{-1}\}\), and without loss of generality,
      suppose \(c = a\).

      Note that \(\act{a}{\phi} = a\). Thus we have that \(\act{b}{\phi} \in \{b, ba, a^{-1}b,
      a^{-1} ba\}\). 
      If \(\act{b}{\phi} \in \{b, a^{-1}ba\}\), as we saw
      in the proof of \cref{free-gp-automorphically-minimal}, every cyclic conjugate of
      \(g\) is either mapped to something strictly longer than \(g\) or a cyclic
      conjugate of \(g\), a contradiction. Thus \(\act{b}{\phi} \in \{ba, a^{-1} b\}\).
      We first consider \(\act{b}{\phi} = ba\). As we saw in the proof of
      \cref{free-gp-automorphically-minimal}, there are at most \(k\) cyclic
      cancellations after applying \(\phi\) to \(g\). Thus, using the fact that
      \(|i_z| \geq 3\) and \(|j_z| \geq 3\), we have
      \[
        |\act{g}{\phi}|_{cyc} \geq -2k + \sum_{z = 1}^k |i_z| + 2|j_z| \geq k +
        \sum_{z = 1}^k
        |i_z| + |j_z| = k + |g| > |g|,
      \]
      a contradiction. The case \(\act{b}{\phi} = a^{-1}b\) analogously produces a
      contradiction. We have thus shown that every Whitehead automorphism that
      maps a cyclic conjugate of \(g\) to something with length \(|g|\) is
      either a permutation of generators and inverses, or preserves the fact that the
      obtained element is of the same form as \(g\) and \(h\) and fixes the
      cyclic conjugates of \(g\) setwise.

      It now remains to show that we can assume all of the permutations involved
      fix \(\langle a, b \rangle\) setwise as currently it is possible that such
      a permutation exchanges \(a\) or \(b\) with a letter not in \(\{a, b,
      a^{-1}, b^{-1}\}\). Noting that \(\actweird{F}{\Inn}\) is normal in \(\actweird{F}{\Aut}\),
      we can rewrite an expression \(\comp{\psi_0}{\phi_1}{\cdots}{\phi_r}{\psi_r}\), where each \(\psi_s\) is inner and each \(\phi_s\) is a permutation of
      the generators of \(F\) and their inverses as \(
      \comp{\alpha_0}{\cdots}{\alpha_r}{\phi_1}{\cdots}{\phi_r} \), for some \(\alpha_0, \ldots, \alpha_r \in
      \Inn(F)\). As a product of permutations of generators \(\comp{\phi_1}{\cdots}{\phi_r}\) is also a permutation of the generators of \(F\) and their
      inverses, 
      we have that \(h\) can be obtained from \(g\) by applying an inner
      automorphism, and then a permutation of the generators. But since \(g\)
      and \(h\) have the same length, the inner automorphism acts by cyclic
      conjugation, and hence fixes \(\langle a, b \rangle\). This forces the
      permutation to also fix \(\langle a, b \rangle\), as required.
  \end{proof}

We now have a large-looking family of distinct minimal representatives.
In particular,  we can conclude the proof of exponential automorphic growth by verifying that our family grows exponentially.

  \begin{theorem}\label{main_free}
     Non-abelian free groups of finite rank have exponential automorphic growth.
  \end{theorem}

  \begin{proof}
    Fix \(m \in \N\). In order to compute growth, it is sufficient to consider a subsequence, and so we can
    assume \(m \in 12\N\). In light of
    \cref{free-gp-automorphically-minimal} and \cref{free-gp-aut-norm-form}, it
    suffices to show that there are exponentially many elements of \(\langle a, b \rangle
    \leq F\) of this form of length \(m\), up to cyclic conjugacy, and permuting
    \(\{a, b, a^{-1}, b^{-1}\}\). 

    Let $L=\{a^{i_1}b^{j_1}\cdots a^{i_k} b^{j_k}:k>1,i_n,j_n\geq 3\}$. Note
    that any automorphisms of \(F\) that are permutations of the generating set
    do not map \(L\) to itself, as words in \(L\) always begin with a positive
    power of \(a\). The number of elements of $L$ of length $m\geq6k$, for a
    fixed $k$, is the number of ways to put $m$ balls in $2k$ buckets, so that
    each bucket has at least 3 balls. 
    This is the same as the number of ways to
    put $m-3\cdot 2k$ balls in $2k$ buckets (after putting $3$ balls in each bucket).
    The standard `stars and bars' argument, gives that \(L\) has 
    $\binom{m-6k+(2k-1)}{2k-1}=\binom{m-4k-1}{2k-1}$ elements, i.e. from a total of
    $m-4k-1$ positions, we choose $2k-1$ as `bars' separating $m-6k$ `stars'
    into $2k-1$ boxes.
    
    To reduce by cyclic permutation, we should divide by $k$, since to keep the
    same form we need the subwords $a^{i_n}b^{j_n}$ to stick together. We should also consider when the word is a power, so there would be fewer than $k$
    distinct cyclic permutations. However, we still have a lower bound for the
    number of automorphism orbits of length $m$
     \[\sum_{k=1}^{\frac{m}{6}}\frac{1}{k}\binom{m-4k-1}{2k-1}\geq
     \frac{1}{n}\binom{12n-4n-1}{2n-1},\]
   where \(m=12n\).
   It thus suffices to show that
   \[\binom{12n-4n-1}{2n-1}\]
   grows exponentially in \(n\).
\[\binom{12n-4n-1}{2n-1}=\frac{(8n-1)!}{(6n)!(2n-1)!}\geq \frac{6n+1}{2n-1}\cdot\frac{6n+1}{2n-1}\ldots \frac{6n+1}{2n-1}> 3^{n}, \]
which is exponential.
\end{proof}

\section{Thompson's groups \(T\) and \(V\)}

Thompson's groups \(T\leq V\) were the first known examples of finitely
presented simple groups. These groups were introduced in the 1960s and have received
considerable attention since then. We show that both of these groups have exponential
automorphic growth (\cref{Tautgrowth} and \cref{thm:V}). The group \(V\) has many equivalent definitions
including the full group of the full one-sided shift on a two letter alphabet
\cite{Visafullgroup}, the automorphism group of the free J\'onsson-Tarski
algebra \cite{higman1974finitely}, a certain group of almost automorphisms the
infinite binary tree  \cite{skipper2021almost}, a group of strand diagrams
\cite{belk2014conjugacy}, a group of tree pairs \cite{ConnonFloydParry}, and
the group of bisynchronising minimal transducers with trivial cores
\cite{collinaut}. In this document we view \(V\) as the group of homeomorphisms
of the Cantor space (\cref{cantorspace}) which can be defined using prefix
replacements (\cref{cantorspace}).
We then view \(T\) as the subgroup of \(V\) consisting of the homeomorphisms which preserve the lexicographic cyclic order on the Cantor space.
As \(T\) has a finite outer automorphism group \cite{autft}, analysing the automorphic growth of \(V\) will be far more involved than \(T\).

Among the many investigations into \(V\), there have been multiple solutions to the conjugacy problem \cites{SalazarDaz2010ThompsonsGV,belk2014conjugacy,barker2016power} and the automorphism group has received attention \cites{collinaut,elliott2023description}.
In analysing the automorphic growth of \(V\) we make use of the description of \(\actweird{V}{\Aut}\) given in \cite{collinaut} (see \cref{autV}).
The intuition for our approach is also heavily inspired by the solution to the conjugacy problem presented in \cite{belk2014conjugacy} (see \cref{closedstranddiagram}), although this solution is not used explicitly in proofs.
The dynamics of elements of \(V\) acting on the Cantor space are often described using so-called `revealing pairs'.
These pairs were introduced in \cite{brin2004higher} and have appeared again multiple times \cites{salazar2010thompson, bleak2013centralizers}. Revealing pairs are also important to our analysis of automorphic growth.
As such, before moving to the proof of exponential automorphic growth,
we start this section with a subsection which reintroduces the background concepts from the literature which we require.

During the writing of this paper a new result about \(V\) was proven in \cite{bishop2026periodgrowthcocontextfreegroups}. The proof of Theorem 6.2 of \cite{bishop2026periodgrowthcocontextfreegroups} implies that the set of orders of torsion elements of \(V\) grows exponentially.
This implies that the automorphic growth of the torsion elements of \(V\) is exponential.
Our proof however, focuses on the non-torsion elements of \(V\).

\label{sec:V} 
\subsection{Background}
For our purposes, the group \(V\) is a subgroup of the group \(\actweird{\mathfrak{C}}{\operatorname{Homeo}}\) of self-homeomorphisms of \(\mathfrak{C}\).
We think of elements of \(\mathfrak{C}\) as \emph{infinite words} extending the finite words in the free monoid \(\{0,1\}^\ast\).
\begin{defn}\label{cantorspace}
    The Cantor space, which we denote by 
    \(\mathfrak{C}\), is the set 
    \(\{0,1\}^\omega\) of infinite sequences in 0s and 1s equipped with the product topology.
        If \(w\in \{0,1\}^\ast\), then we denote the set of all elements of \(\mathfrak{C}\) with \(w\) as a prefix by \(w\mathfrak{C}\). The sets of the form \(w\mathfrak{C}\) are called \emph{cones}. 
\end{defn}
If \(v,w\in \{0,1\}^\ast\cup \{0,1\}^\omega\), then we write \(v
\leq w\) to mean that \(v\) is a prefix of \(w\).
Note that, in particular, for \(v,w\in \{0,1\}^\ast\), we have
\[v\leq w \iff v\mathfrak{C} \supseteq w\mathfrak{C}.\]
In particular, any two cones are either comparable or have empty intersection.
The set of all cones form a basis for the topology on \(\mathfrak{C}\).
While the elements of \(V\) are homeomorphisms of \(\mathfrak{C}\), each element will also give us a way of mapping almost any finite word to another in the following sense.
\begin{defn}
For all \(w_1, w_2\in \{0,1\}^\ast\), there is a canonical homeomorphism from \(w_1\mathfrak{C}\) to \(w_2\mathfrak{C}\) given by \(w_1y\mapsto w_2y\).
We say that a self-homeomorphism of \(\mathfrak{C}\) \emph{rigidly} maps \(w_1\mathfrak{C}\) to \(w_2\mathfrak{C}\) (or sometimes simply \(w_1\) to \(w_2\)) if it restricts to this canonical bijection from \(w_1\mathfrak{C}\) to \(w_2\mathfrak{C}\).
\end{defn}
The group \(V\) consists precisely of the homeomorphisms which map all elements of \(\mathfrak{C}\) in the above fashion.
\begin{defn}\label{def:V}
Thompson's group \(V\) is the group of \(v\in \operatorname{Homeo}(\mathfrak{C})\) such that for all \(x\in \mathfrak{C}\) there is a cone containing \(x\) which \(v\) rigidly maps to another cone.
If \(v\in V\) rigidly maps a word \(w_1\in \{0,1\}^\ast\) to \(w_2\in \{0,1\}^\ast\), then we write \(\act{w_1}{v}\) to denote this word \(w_2\).
\end{defn}

By a result of Rubin \cites{rubin,rubinshort}, every automorphism of a `nice enough' group of homeomorphisms of a topological space \(X\) is induced by the conjugation action of some homeomorphism of \(X\).
Thompson's group \(V\) is such a group and in \cite{collinaut} a complete description of these homeomorphisms of \(\mathfrak{C}\) is provided.
Grigorchuk, Nekrashevich and Sushchanskii introduce a group of homeomorphisms of the Cantor space known as the \emph{rational} homeomorphisms \cite{GNS}. 
These are, in some sense, the homeomorphisms definable with a `simple-enough' machine.
The description of automorphisms of \(V\) given in \cite{collinaut} reduces the automorphism group of \(V\) to a certain subclass of these machines.
We first introduces how these machines function in general.

\begin{defn}
    An \emph{initial transducer} \(T\) is a \(4\)-tuple \((Q,\pi,\lambda,q_0)\) where 
    \begin{itemize}
        \item \(Q\) is a finite set (we refer to elements of \(Q\) as \emph{states});
        \item \(\pi \colon Q\times \{0,1\}\to Q\) is a function (we call \(\pi\) the \emph{transition function});
        \item \(\lambda \colon Q\times \{0,1\}\to \{0,1\}^\ast\) is a function (we call \(\lambda\) the \emph{output function});
        \item \(q_0\in Q\) (we call \(q_0\) the \emph{initial state} of \(T\)).
    \end{itemize}
    We extend the domains of \(\pi\) and \(\lambda\) to \(Q\times \{0,1\}^\ast\) inductively as follows:
    \begin{itemize}
        \item for all \(q\in Q\), \(\actweird{q,\varepsilon}{\pi}=q
    \) and \(\actweird{q,\varepsilon}{\lambda}=\varepsilon
    \);
    \item for \(w\in \{0,1\}^\ast\) and \(l\in \{0,1\}\), we define:
    \[\actweird{q,wl}{\pi}=\actweird{\actweird{q,w}{\pi},l}{\pi}\text{ and }\actweird{q, wl}{\lambda}=\actweird{q,w}{\lambda}\actweird{\actweird{q, w}{\pi},l}{\lambda}.\]
    \end{itemize}
\end{defn}

Each of these machines naturally acts on the set \(\{0,1\}^\ast\) of all finite words using \(\actweird{q_0,\cdot}{\lambda}\).
However this action doesn't always extend to the Cantor space and it is possible for two machines to define the same function.
Grigorchuk, Nekrashevich and Sushchanskii establish some conditions to avoid these problems \cite{GNS}. 

\begin{defn}
Suppose that \(T=(Q,\pi,\lambda,q_0)\) is an initial transducer.
\begin{itemize}
    \item We say that \(T\) is \emph{non-degenerate} if for all \(q\in Q\) and \(n\in \N\) there is \(N_n\in \N\) such that for all \(w\in \{0,1\}^{N_n}\) we have \(|\actweird{q,w}{\lambda}|\geq n\);
    \item We say that \(T\) has \emph{complete response} if for all \(q\in Q\) and \(l\in \{0,1\}\), the word \(\actweird{q,l}{\lambda}\) is the longest common prefix of all words in the set \(\makeset{\actweird{q,lw}{\lambda}}{\(w\in \{0,1\}^{N_1}\)}\);
    \item We say that \(T\) has \emph{no duplicate states} if whenever \(q,q'\in Q\) satisfy \(\actweird{q,w}{\lambda}=\actweird{q',w}{\lambda}\) for all \(w\in \{0,1\}^\ast\), we have \(q=q'\);
    \item We say that \(T\) is \emph{accessible} if for all \(q\in Q\), there is \(w\in \{0,1\}^\ast\) with \(\actweird{q_0,w}{\pi}=q\);
    \item We say that \(T\) is \emph{minimal} if \(T\) satisfies all of the above conditions.
\end{itemize}
\end{defn}
Given these conditions, we have a one-to-one correspondence between the (isomorphism types of) minimal transducers and the so-called rational continuous maps on the Cantor space.
\begin{defn}
     Let \(T=(Q,\pi,\lambda,q_0)\) be a non-degenerate initial transducer.
     For all \(q\in Q\) and \(x=x_0x_1\ldots\in \mathfrak{C}\), define 
    \[\actweird{q,x}{\lambda} \coloneqq \actweird{q,x_0}{\lambda}\actweird{\actweird{q,x_0}{\pi},x_1}{\lambda}\actweird{\actweird{q,x_0x_1}{\pi},x_2}{\lambda}\ldots\]
    We will see in \cref{rationalmin} that \(\actweird{q, x}{\lambda}\) is a well-defined
    element of \(\mathfrak{C}\). The map \(x\mapsto \actweird{q, x}{\lambda}\) is called the
    \emph{rational map of }\(q\). The \emph{rational map of }\(T\) is defined to
    be the \emph{rational map of }\(q_0\).
\end{defn}
The following proposition is proved in \cite{GNS}.
\begin{proposition}\label{rationalmin}
    Rational maps are always well-defined continuous maps and moreover each rational map is the rational map of a unique minimal transducer (up to relabelling the states).
\end{proposition}

We can now think of rational maps as machines. 
We wish to restrict our attention to the machines which describe the rational maps which induce automorphisms of \(V\). 
These are the so-called bisynchronising homeomorphisms of \cite{collinaut}.
\begin{defn}
    We say that \(T\) is \emph{strongly synchronising of length \(k\in \N\)} if
    for all \(p,q\in Q\) and \(w\in \{0,1\}^k\) we have \(\actweird{p,w}{\pi}=\actweird{q,w}{\pi}\).
    In this case we say that \(w\) synchronises \(T\) to the state \(\actweird{p,w}{\pi}\).
    We say that \(T\) is \emph{strongly synchronising} if it has a finite
    synchronising length. We say that a continuous map \(c \colon \mathfrak{C}\to
    \mathfrak{C}\) is \emph{synchronising} if it is rational and its minimal
    transducer is strongly synchronising. We say that a homeomorphism  \(c
    \colon \mathfrak{C}\to \mathfrak{C}\) is \emph{bisynchronising} if \(c\) and
    \(c^{-1}\) are both synchronising. 
    We denote the set of bisynchronising
    homeomorphisms of \(\mathfrak{C}\) by \(B_{2,1}\). 
\end{defn}

These strongly synchronising machines have appeared many times since their introduction.
For more background on strongly synchronising transducers see for example \cites{collinaut,arXiv:2004.00516, elliott2023description, zbMATH07720878,arXiv:2407.18720}.
The following result is {\cite{collinaut}*{Theorem 1.1}} and another proof is provided in {\cite{elliott2023description}*{Theorem 4.14}}.

\begin{proposition}\label{autV}
The set \(B_{2,1}\) is a group and is isomorphic to \(\actweird{V}{\operatorname{Aut}}\).
This isomorphism is given by mapping \(f\in B_{2,1}\) to the automorphism \(v\mapsto \comp{f^{-1}}{\circ}{v}{\circ}{f}\). 
We will often write \(v^f\) to denote \(\comp{f^{-1}}{\circ}{v}{\circ}{f}\).
\end{proposition}

We now have a full description of the automorphism group of \(V\) using
transducers. In order to compute the automorphic growth of \(V\), we will also
require knowledge of the dynamics of specific elements so that we can analyse
how much this changes when we apply our bisynchronising maps. For this we need
the notion of a revealing tree pair.

\begin{defn}
    A \emph{finite rooted binary tree} is a non-empty finite prefix closed subset \(A\subseteq \{0,1\}^\ast\) such that for all \(w\in A\) we either have \(\{w,w0,w1\}\subseteq A\) (we call \(\{w,w0,w1\}\) a \emph{caret}) or \(\{w0,w1\}\cap A=\varnothing\).
    A \emph{leaf} of \(A\) is a maximal element of \(A\) with respect to the prefix partial order.
    A \emph{tree pair} is a triple \((A,B,\sigma)\) where \(A\) and \(B\) are finite rooted binary trees and \(\sigma\) is a bijection from the leaves of \(A\) to the leaves of \(B\).
    A subset of \(\{0,1\}^\ast\) is called a \emph{complete prefix code} if it is the set of leaves of a finite rooted binary tree. 
    Each tree pair \((A,B,\sigma)\) induces an element of \(v\) which rigidly maps \(a\) to \(\act{a}{\sigma}\) for all leaves \(a\) of \(A\).
\end{defn}
From the compactness of the Cantor space, it is routine to show that every element of \(V\) is defined by some tree pair. 
An example tree pair can be found in \cref{fig:placeholder}. 
To define our revealing tree pairs, we must sort the leaves into different types.

\begin{defn}
Suppose that \(T=(A,B,\sigma)\) is a tree pair. We say that a word is a
\emph{leaf} of \(T\) if it is a leaf of \(A\) or \(B\). In order to define
\(\sigma^n\), we need to take care as the domain and range may not coincide. In such a case, we compose and invert the bijection \(\sigma\) as if it is a
partial function. We say that a leaf \(l\) of \(T\)
is 
\begin{enumerate}
   \item \emph{neutral} if it is a leaf of both \(A\) and \(B\);
   \item a \emph{range of repulsion} if it is not neutral and there is \(n\in \N\setminus \{0\}\) such that \(l\) is a strict prefix of \(\act{l}{\sigma^{-n}}\), in this case \(\act{l}{\sigma^{-n}}\) is called a \emph{repeller};
   \item  a \emph{domain of attraction} if it is not neutral and there is \(n\in \N\setminus\{0\}\) such that \(l\) is a strict prefix of \(\act{l}{\sigma^{n}}\), in this case \(\act{l}{\sigma^{n}}\) is called an \emph{attractor};
   \item a \emph{source} if it is a leaf of \(A\) which is not a repeller, neutral, or a domain of attraction;
   \item a \emph{sink} if it is a leaf of \(B\) which is not an attractor, neutral, or a range of repulsion.
\end{enumerate}
\end{defn}
Note that in (2), \((l)\sigma^{-n}\) being well-defined implies that \((l)\sigma^{-n}\) is a leaf of \(A\) and has a leaf of \(B\) as a strict prefix.
In particular, all the repellers are leaves of the left tree and not the right tree.
Similarly, all the attractors are leaves of the right tree and not the left tree.
Thus these five classes partition the leaves of any tree pair.
In \cref{fig:placeholder} we have \(0\) is neutral, \(10\) is a range of repulsion, \(100\) is a repeller, \(11\) is a domain of attraction, \(111\) is an attractor, \(1010,1011\) are sources, and \(1100,1101\) are sinks.
Revealing pairs are defined in terms of these repellers and attractors.
In particular we need one for each \emph{component}.

\begin{defn}
   Note that if \((A,B,\sigma)\) is a tree pair, then \(C \coloneqq A\cap B\) is
   also a finite rooted binary tree. Each leaf of \(C\) is either a leaf of
   \(A\) or a leaf of \(B\). For each leaf \(l\) of \(C\) which is not a leaf of
   \(A\), we define \emph{the component of \(A\) under \(l\)} to be the set of
   elements of \(A\) with \(l\) as a prefix. 
   We call such a set a 
   \emph{component of \(A\setminus B\)}. The \emph{components of \(B\setminus
   A\)} are defined similarly.
\end{defn}
If a component of \(A\setminus B\) contains a repeller leaf, then \(\sigma\) will eventually map this leaf to the unique leaf of \(B\) with the repeller as a prefix.
As \(\sigma\) is injective, it follows that every component of \(A\setminus B\) contains at most one repeller leaf. Similarly each component of \(B\setminus A\), contains at most one attractor leaf.
Hence the following definition.
\begin{defn}
We say that a tree pair \((A,B,\sigma)\) is \emph{revealing} if each component of \(A\setminus B\) contains a repeller leaf and each component of \(B\setminus A\) contains an attractor leaf.
\end{defn}

In particular, the tree pair in \cref{fig:placeholder} is revealing.
In \cite{brin2004higher}, it is shown that every element of \(V\) can be defined using such a tree pair.
\begin{proposition}
   For all \(v \in V\), there is a revealing pair representing \(v\).
\end{proposition}
As such, we will be working exclusively with tree pairs of this type in our analysis of the automorphic growth of \(V\).

\begin{figure}

\usetikzlibrary{decorations.markings}

\begin{tikzpicture}[
middlearrow/.style={
        postaction={decorate, decoration={
            markings,
            mark=at position 0.5 with {\arrow{#1}}
        }}
}
]

\begin{scope}[xshift=0cm]
\draw[teal] (-1.5,3) -- (-1,2);
\draw[teal] (-1.5,3) -- (-2,2) node[left] {$a$};

\draw[blue] (-1,2) -- (-1.8,1);
\draw[magenta] (-1,2) -- (-0.2,1);
\draw[magenta] (-0.2,1) -- (-0.5,0) node[right] {$f$};
\draw[blue] (-0.2,1) -- (0.3,0) node[right] {$g$};
\draw[blue] (-1.8,1) -- (-1.2,0);
\draw[blue] (-1.8,1) -- (-2.4,0) node[left] {$b$};

\draw[blue] (-1.2,0) -- (-1.8,-1);
\draw[blue] (-1.2,0) -- (-0.5,-1) node[right] {$e$};


\draw[blue] (-1.8,-1) -- (-2.2,-2) node[left] {$c$};
\draw[blue] (-1.8,-1) -- (-1.2,-2) node[right] {$d$};

\end{scope}

\begin{scope}[xshift=7cm]
\draw[teal] (-0.5,3) -- (-1,2) node[right]{};
\draw[teal] (-0.5,3) -- (0,2) node[right] {$f$};

\draw[blue] (-1,2) -- (-1.8,1) node[left]{$b$};
\draw[magenta] (-1,2) -- (-0.5,1);
\draw[magenta] (-0.5,1) -- (-1.2,0) node[right] {$a$};

\draw[blue] (-0.5,1) -- (0.1,0) node[left] {};
\draw[blue] (0.1,0) -- (-0.4,-1) node[left] {};
\draw[blue] (0.1,0) -- (0.6,-1) node[right] {$d$};
\draw[blue]  (-0.4,-1)-- (-0.9,-2)node[left] {$g$};
\draw[blue]  (-0.4,-1)-- (0.1,-2)node[left] {};
\draw[blue]   (0.1,-2) --(-0.4,-3) node[left] {$c$};
\draw[blue]   (0.1,-2) --(0.6,-3) node[right] {$e$};





\end{scope}

\end{tikzpicture}
\caption{An element of \(V\) represented by a revealing tree pair. The common tree can be seen in green. The attractor and repeller can be seen in purple and the sources and sinks can be seen in blue.}\label{mainveltfig}
\end{figure}
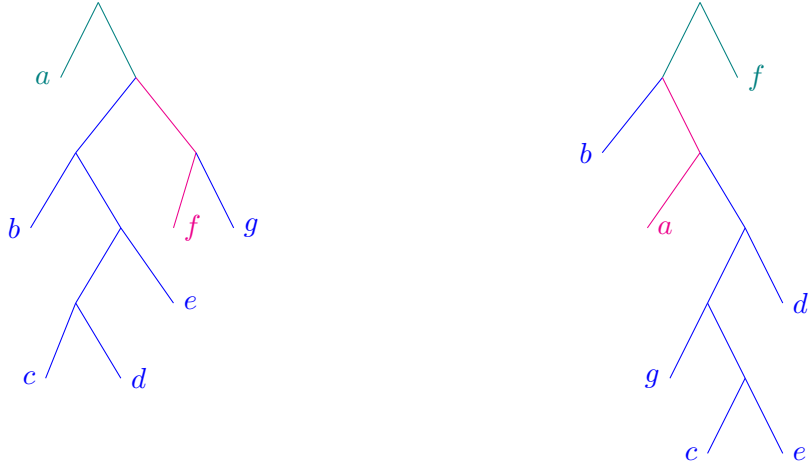
\begin{figure}

\usetikzlibrary{decorations.markings, arrows.meta}

\begin{tikzpicture}[
middlearrow/.style={
        postaction={decorate, decoration={
            markings,
            mark=at position 0.5 with {\arrow{#1}}
        }}
}
]

\begin{scope}[yshift=3cm]

\draw[magenta,middlearrow={>}] (1.2,2) to[out=90, in=90] (-1.2,2);
\draw[magenta,middlearrow={>}] (-1.2,2) to[out=270, in=270] (1.2,2);

\draw[blue,middlearrow={>}] (1.2,2) -- (2,1);
\node at (2.2,0.8) {};
\node at (1.6,1.8) {$1$};
\node at (1.3,2.3) {$0$};

\draw[blue,middlearrow={>}] (-1.2,2) -- (-1.8,1);
\draw[blue,middlearrow={>}] (-1.8,1) -- (-1.2,0);
\draw[blue,middlearrow={>}] (-1.8,1) -- (-2.4,0) node[left] {};
\node at (-0.95,1.8) {$1$};
\node at (-1.5,1.8) {$0$};

\draw[blue,middlearrow={>}] (-1.2,0) -- (-1.8,-1);
\draw[blue,middlearrow={>}] (-1.2,0) -- (-0.5,-1) node[right] {};
\node at (-1.5,0.8) {$1$};
\node at (-2.1,0.8) {$0$};

\node at (-0.9,-0.2) {$1$};
\node at (-1.5,-0.2) {$0$};

\draw[blue,middlearrow={>}] (-1.8,-1) -- (-2.2,-2) node[left] {};
\draw[blue,middlearrow={>}] (-1.8,-1) -- (-1.2,-2) node[right] {};
\node at (-1.5,-1.2) {$1$};
\node at (-2.1,-1.2) {$0$};

\end{scope}

\begin{scope}[yshift=-5cm, yscale=-1]

\draw[magenta,middlearrow={>}] (1.2,2) to[out=90, in=90] (-1.2,2);
\draw[magenta,middlearrow={>}] (-1.2,2) to[out=270, in=270] (1.2,2);

\draw[blue,middlearrow={>}] (-1.8,1) -- (-1.2,2);
\node at (-1.8,0.8) {};
\node at (-1.5,1.9) {$0$};
\node at (-1.35,2.25) {$1$};

\draw[blue,middlearrow={>}] (2,1) -- (1.2,2);
\node at (0.9,1.8) {$0$};
\node at (1.6,1.8) {$1$};

\draw[blue,middlearrow={>}] (1.2,0) -- (2,1);
\draw[blue,middlearrow={>}] (2.8,0) -- (2,1);
\node at (2.8,-0.2) {};
\node at (1.6,0.8) {$0$};
\node at (2.4,0.8) {$1$};

\draw[blue,middlearrow={>}] (0.8,-1) -- (1.2,0);
\draw[blue,middlearrow={>}] (1.8,-1) -- (1.2,0);
\node at (0.8,-1.2) {};
\node at (0.9,-0.2) {$0$};
\node at (1.5,-0.2) {$1$};

\draw[blue, middlearrow={>}] (2.2,-2) -- (1.8,-1);
\node at (2.2,-2.2) {};
\draw[blue, middlearrow={>}] (1.4,-2) -- (1.8,-1);
\node at (1.4,-2.2) {};
\node at (1.5,-1.2) {$0$};
\node at (2.1,-1.2) {$1$};

\end{scope}

\draw[-, thick, red] 
  (-2.2,1) to[out=240, in=120] (1.4,-3); 

\draw[-, thick, red] 
  (-2.4,3) to[out=240, in=120] (-1.8,-6);
  
\draw[-, thick, red] 
  (-1.2,1) to[out=300, in=45] (2.8,-5);

\draw[-, thick, red] 
  (-0.5,2) to[out=300, in=60] (2.2,-3);

\draw[-, thick, red] 
  (2,4) to[out=300, in=120] (0.8,-4);

\end{tikzpicture}
\caption{By connecting the leaves of \cref{mainveltfig} with the same labels, connecting the roots of the two trees, and performing reductions as described in \cite{zbMATH06309498}, we obtain  an abstract closed strand diagram.
One of the main results of \cite{zbMATH06309498}, is that two elements of \(V\) are conjugate if and only if they have the same abstract closed strand diagram (together with a cohomology class which we omit for simplicity).
In particular if two elements of \(v\) have the same closed stand diagram (and cohomology), then they are in the same automorphic orbit. Moreover, the automorphisms of \(V\) have an action on the set of all such diagrams via its action on the set of conjugacy classes.
The letters \(0\) and \(1\) track the `rotation system' required in \cite{zbMATH06309498} by providing a circular order for the edges at each vertex. 
The red parts of the edges separate the vertices with in-degree 2 from the vertices with in-degree 1.
 The diagram is drawn to resemble the strand diagram of an element of \(V\). 
 In fact, if we view each purple circuit as a single node, the above diagram exactly represents the strand diagram of an element of \(V\).
 In \cref{separatedstranddiagrams} we see that by removing the red parts of the diagram we have a `tree pair' for this element.
}\label{closedstranddiagram}
\end{figure}
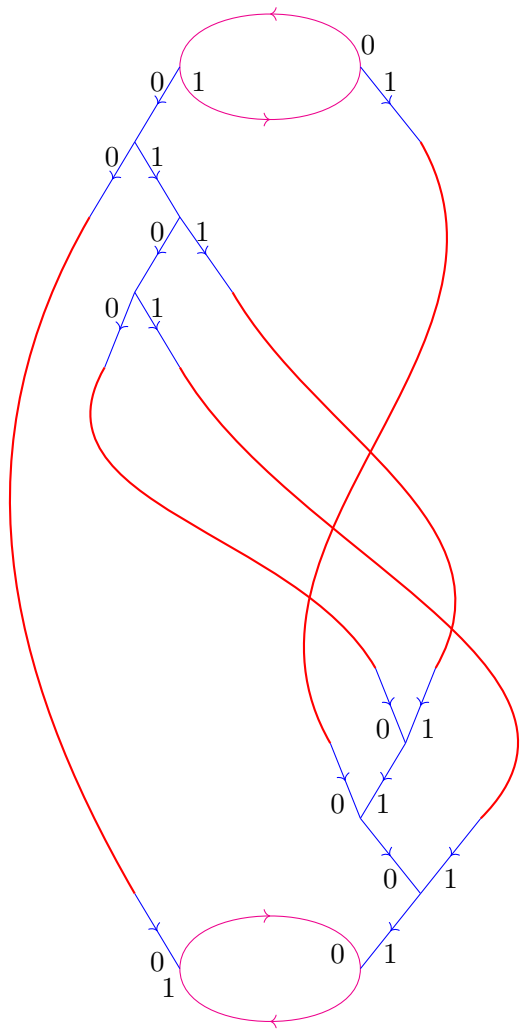

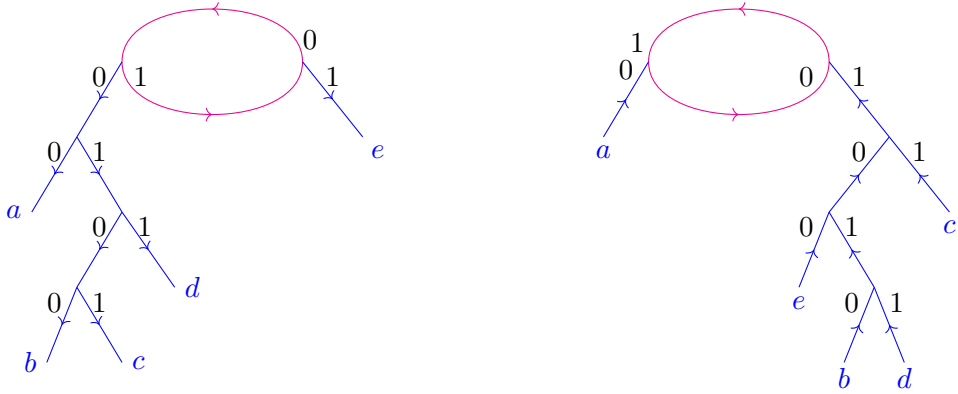
\begin{figure}

\usetikzlibrary{decorations.markings}

\begin{tikzpicture}[
middlearrow/.style={
        postaction={decorate, decoration={
            markings,
            mark=at position 0.5 with {\arrow{#1}}
        }}
}
]

\begin{scope}[xshift=0cm]

\draw[magenta,middlearrow={>}] (1.2,2) to[out=90, in=90] (-1.2,2);
\draw[magenta,middlearrow={>}] (-1.2,2) to[out=270, in=270] (1.2,2);

\draw[blue,middlearrow={>}] (1.2,2) -- (2,1);
\node[blue] at (2.2,0.8) {$e$};
\node at (1.6,1.8) {$1$};
\node at (1.3,2.3) {$0$};

\draw[blue,middlearrow={>}] (-1.2,2) -- (-1.8,1);
\draw[blue,middlearrow={>}] (-1.8,1) -- (-1.2,0);
\draw[blue,middlearrow={>}] (-1.8,1) -- (-2.4,0) node[left] {$a$};
\node at (-0.95,1.8) {$1$};
\node at (-1.5,1.8) {$0$};

\draw[blue,middlearrow={>}] (-1.2,0) -- (-1.8,-1);
\draw[blue,middlearrow={>}] (-1.2,0) -- (-0.5,-1) node[right] {$d$};
\node at (-1.5,0.8) {$1$};
\node at (-2.1,0.8) {$0$};

\node at (-0.9,-0.2) {$1$};
\node at (-1.5,-0.2) {$0$};

\draw[blue,middlearrow={>}] (-1.8,-1) -- (-2.2,-2) node[left] {$b$};
\draw[blue,middlearrow={>}] (-1.8,-1) -- (-1.2,-2) node[right] {$c$};
\node at (-1.5,-1.2) {$1$};
\node at (-2.1,-1.2) {$0$};

\end{scope}

\begin{scope}[xshift=7cm]

\draw[magenta,middlearrow={>}] (1.2,2) to[out=90, in=90] (-1.2,2);
\draw[magenta,middlearrow={>}] (-1.2,2) to[out=270, in=270] (1.2,2);

\draw[blue,middlearrow={>}] (-1.8,1) -- (-1.2,2);
\node[blue] at (-1.8,0.8) {$a$};
\node at (-1.5,1.9) {$0$};
\node at (-1.35,2.25) {$1$};

\draw[blue,middlearrow={>}] (2,1) -- (1.2,2);
\node at (0.9,1.8) {$0$};
\node at (1.6,1.8) {$1$};

\draw[blue,middlearrow={>}] (1.2,0) -- (2,1);
\draw[blue,middlearrow={>}] (2.8,0) -- (2,1);
\node[blue] at (2.8,-0.2) {$c$};
\node at (1.6,0.8) {$0$};
\node at (2.4,0.8) {$1$};

\draw[blue,middlearrow={>}] (0.8,-1) -- (1.2,0);
\draw[blue,middlearrow={>}] (1.8,-1) -- (1.2,0);
\node[blue] at (0.8,-1.2) {$e$};
\node at (0.9,-0.2) {$0$};
\node at (1.5,-0.2) {$1$};

\draw[blue,middlearrow={>}] (2.2,-2) -- (1.8,-1);
\node[blue] at (2.2,-2.2) {$d$};
\draw[blue,middlearrow={>}] (1.4,-2) -- (1.8,-1);
\node[blue] at (1.4,-2.2) {$b$};
\node at (1.5,-1.2) {$0$};
\node at (2.1,-1.2) {$1$};

\end{scope}

\end{tikzpicture}
\caption{The `prefix replacement map' inside the closed strand diagram of
\cref{closedstranddiagram}. It is this data that we use to show automorphic
growth in \(V\) is exponential. In particular this is the `decoration' map we
get from \(v\) in \cref{decmaps}. The word `decoration' comes with the idea
that we are decorating the circuits with strands sticking out of them and
potentially attaching trees to each decoration in order to define prefix
replacement. Note that the \(01\) circuits in the diagram correspond to the
periodic points of the action on the element from \cref{mainveltfig}; this is why periodic points come
up in \cref{decorationpointreps}.
}\label{separatedstranddiagrams}
\end{figure}

\subsection{Automorphic growth of V}
This section includes many definitions which the reader may find unintuitive.
In particular, given an element of \(V\) we associate to it some data which we later show is automorphism invariant.
In \cref{closedstranddiagram} and \cref{separatedstranddiagrams}, we apply these constructions to the element of \(V\) described in \cref{mainveltfig}.
The idea behind and motivation for these definitions comes from diagrams of this type.
The exponential automorphic growth argument essentially amounts to formally describing the data in these diagrams in such a way that we can act on them with the transducers from \cref{autV}.

Each point in the diagram \cref{closedstranddiagram} where the purple circuit has a blue decoration coming off of it is thought of as a \emph{decoration point}.
Formally, these decoration points correspond to a certain equivalence class of finite words.
We call these finite words \emph{decoration point representatives} and are only important up to the equivalence we will introduce in \cref{equivdecpt}.
\begin{defn}\label{decorationpointreps}
For \(v\in V\) we refer to elements of \(\mathfrak{C}\) which belong to a finite orbit of the group \(\langle v\rangle\leq \actweird{\mathfrak{C}}{\operatorname{Homeo}}\) as \emph{periodic points} of \(v\).
We will often refer to orbits of \(\langle v \rangle\) as orbits of \(v\).
A word \(p\in \{0,1\}^\ast\) is called a \emph{source decoration point representative} of \(v\in V\) if there is a sequence of finite words
\[p_0\mapsto p_1\mapsto \ldots \mapsto p_n\]
each mapped rigidly to the next by \(v\) and such that \(p=p_n\) is a proper prefix of \(p_0\).
If instead \(p=p_0\) is a proper prefix of \(p_n\), then \(p\) is called a \emph{sink decoration point representative} of \(v\in V\).
\end{defn}
Each of these representatives naturally corresponds to a periodic point in our action on the Cantor space.
These periodic points will be helpful in understanding the dynamics of how the automorphisms of \(V\) affect the data which we attach to each element.
\begin{defn}
    Suppose that \(p\) is a source decoration point representative of \(v\in V\), \(w\in \{0,1\}^\ast\setminus \{\varepsilon\}\) and \(n>0\)  are such that \(v^n\) rigidly maps \(pw\) to \(p\).
    It follows that \(pwww\ldots\in \mathfrak{C}\) is fixed by \(v^n\).
We refer to \(pww\ldots\) as \emph{the periodic point corresponding to \(p\)}.
Periodic points corresponding to source decoration point representatives are called \emph{source periodic points}.
We define \emph{sink periodic points} similarly.
\end{defn}
Note that if \(p\) is a source periodic point of \(v\) with period \(n\), then whenever \(p'\) is close enough to \(p\), the sequence \((p')v^{-n}, (p')v^{-2n},\ldots\) converges to \(p\) and the sequence \((p')v^n, (p')v^{2n},\ldots\) does not converge to \(p\).
In particular, a periodic point cannot be both a source and a sink.
Intuitively, the `decoration points' are the points on the circuits (see \cref{separatedstranddiagrams}) where the blue parts are attached, and the decorations are the first blue edges coming out of the circuits. Sources have edges coming out, and sinks have edges going in.

\begin{defn}\label{equivdecpt}
Let \(\operatorname{SourceDecPt}_v\) be the set of equivalence classes of source decoration point representatives of \(v\) with respect to the relation:
\begin{itemize}
    \item \(p\sim p'\) if there is \(n\in \Z\) such that \(v^n\) rigidly maps \(p\) to \(p'\).
\end{itemize}
We refer to an element of  \(\operatorname{SourceDecPt}_v\) as a \emph{source decoration point}.
We define 
\[\operatorname{SourceDec}_v:=\makeset{(S,a)\in \operatorname{SourceDecPt}_v\times \{0,1\}}{for all \(s\in S\), \(sa\) is not a prefix of the periodic\\ point corresponding to \(s\)}\]
and we call the elements of \(\operatorname{SourceDec}_v\) the \emph{source decorations} of \(v\).
We will usually write \(Sa\) instead of \((S,a)\).
We define the set \(\operatorname{SinkDecPt}_v\) of \emph{sink decoration points} and the set \(\operatorname{SinkDec}_v\) of sink decorations similarly.
\end{defn}

We now have our decoration points. 
We use these to define the decoration maps using the trees coming out of these points as shown in \cref{separatedstranddiagrams}.
This is the primary data we use to distinguish different automorphic orbits in \cref{notautoV}.

\begin{defn}\label{decmaps}
    If \(X\) and \(Y\)
    are
    sets, then we write \(X\mathfrak{C}\) and \(Y\mathfrak{C}\) to mean, respectively 
    \(X\times \mathfrak{C}\) and \(Y\times \mathfrak{C}\)
    with the product topology,
    where \(X\) and \(Y\) are discrete.
    We define \(\actweird{X,Y}{V}\) to be the set of homeomorphisms from \(X\mathfrak{C}\) to \(Y\mathfrak{C}\) which are defined by prefix replacement. 
    That is to say if \(f\in \actweird{X,Y}{V}\) then for all \(x\in X\) and \(z\in \mathfrak{C}\), there is a prefix \(p\) of \(z\), \(y\in Y\) and a finite word \(p'\in \{0,1\}^\ast\) such that for all \(z'\in \mathfrak{C}\) we have \(\act{xpz'}{f}=yp'z'\). In this case, we say that \(f\) rigidly maps \(xp\) to \(yp'\).
Given \(v\in V\), we define its \emph{decoration map} 
\[d_v\in V
(\operatorname{SourceDec}_v,\operatorname{SinkDec}_v)\]
to be the map obtained by mapping
\(S_1l_1w_1x\) to \(S_2l_2w_2x\) (where \(S_1l_1\in \operatorname{SourceDec}_v\), \(S_2l_2\in \operatorname{SinkDec}_v\), \(w_1,w_2\in \{0,1\}^\ast\), \(x\in \mathfrak{C}\)) if for all \(s_1\in S_1\) and \(s_2\in S_2\) there is a positive power of \(v\) which rigidly maps \(s_1l_1w_1\) to \(s_2l_2w_2\).
We verify in \cref{well-defined decoration map} that \(d_v\) is well-defined (including that
every point in the domain of \(d_v\) is indeed mapped somewhere).
\end{defn}
We provide an example of these definitions using a concrete element of \(V\).
\begin{ex}\label{vexample}
    Let \(v\in V\) be the element which rigidly maps \(0\to 0\), \(100\to 10\), \(1010\to 1101\), \(1011\to 1100\), and \(11\to 111\). This element is shown in \cref{fig:placeholder}.

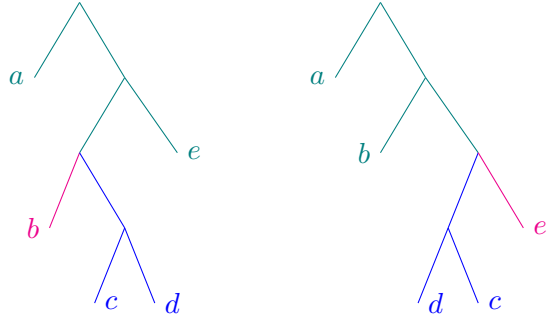
\begin{figure}

\usetikzlibrary{decorations.markings}

\begin{tikzpicture}[
middlearrow/.style={
        postaction={decorate, decoration={
            markings,
            mark=at position 0.5 with {\arrow{#1}}
        }}
}
]

\begin{scope}[xshift=0cm]

\draw[teal] (-1.8,1) -- (-1.2,0);
\draw[teal] (-1.8,1) -- (-2.4,0) node[left] {$a$};

\draw[teal] (-1.2,0) -- (-1.8,-1);
\draw[teal](-1.2,0) -- (-0.5,-1) node[right] {$e$};


\draw[magenta] (-1.8,-1) -- (-2.2,-2) node[left] {$b$};
\draw[blue](-1.8,-1) -- (-1.2,-2) node[right] {};
\draw[blue] (-1.2,-2) -- (-1.6,-3) node[right] {$c$};
\draw[blue](-1.2,-2) -- (-0.8,-3) node[right] {$d$};


\end{scope}

\begin{scope}[xshift=4cm]

\draw[teal] (-1.8,1) -- (-1.2,0);
\draw[teal] (-1.8,1) -- (-2.4,0) node[left] {$a$};

\draw[teal] (-1.2,0) -- (-1.8,-1) node[left] {$b$};
\draw[teal] (-1.2,0) -- (-0.5,-1) node[right] {};

\draw[blue] (-0.5,-1) -- (-0.9,-2) node[left] {};
\draw[magenta] (-0.5,-1) -- (0.1,-2) node[right] {$e$};
\draw[blue ](-0.9,-2) -- (-1.3,-3) node[right] {$d$};
\draw[blue] (-0.9,-2) -- (-0.5,-3) node[right] {$c$};

\end{scope}

\end{tikzpicture}\centering
       \caption{The element of \(V\) from \cref{vexample}. This is a revealing pair for this element. The common tree (including neutral leaves) can be seen in green. The attractor and repeller can be seen in purple and the sources and sinks can be seen in blue.} 
        \label{fig:placeholder}
\end{figure}
    This element has uncountably many periodic points as all points in
    \(0\mathfrak{C}\) are fixed. Outside of this cone there are two remaining
    periodic points, namely \(10000\cdots\) and \(1111\cdots\), which are both
    fixed. The point \(10000\cdots\) is a source and the point \(1111\cdots\) is
    a sink, since every long enough prefix of \(1000\cdots\) is mapped rigidly
    to a shorter word (namely \(10^n\) is mapped to \(10^{n - 1}\) for all \(n
    \geq 2\)) and every long enough prefix of \(1111\cdots\) is mapped rigidly
    to a longer word (namely \(1^n\) is mapped to \(1^{n + 1}\) for all \(n \geq
    2\)). Thus, the source decoration point representatives are
    \(10,100,1000,\ldots\) and the sink decoration point representatives
    are \(11, 111, 111, \ldots\). 
    Note that the decoration points \(10\) and \(11\) are where the purple and blue diverge in \cref{fig:placeholder}. Moreover, the repeating part of the fixed points can be seen as the purple in \cref{fig:placeholder}. Another example of this can be seen in \cref{mainveltfig} earlier.
    
    The source decoration point representatives of \(v\) are all rigidly mapped to each other by powers of \(v\) (\(\act{10^n}{v^k} = 10^{k + n}\)), and so we have a unique source decoration point
    \(\{100^n : n \in \N\}\) with a unique source decoration \(\{100^n : n \in \N\}1\).
    We think of the source decoration point as \(1000\cdots 000\) and the decoration as \(1000\cdots 0001\).
    Similarly there is a unique sink decoration which one can think of as \(111\cdots 1110\).
    If we repeatedly apply the rigid action of \(v\) to a word of the form \(1000\cdots 00010w\), we eventually perform the replacement \(1010w\to 1101w\) (after having used \(v\) to remove all but one of the \(0\)s before \(10w\)). 
    If we continue to apply \(v\) we then obtain the words \(11\cdots 1101w\).
    Thus \(d_v\) rigidly maps \((1000\cdots 0001)0\) to \((111\cdots 1110)1\). Similarly, \(d_v\) rigidly maps \((1000\cdots 0001)1\) to \((111\cdots 1110)0\). Note that \(d_v\) comes from the sources and sinks of the revealing pair in \cref{fig:placeholder}.
\end{ex}

Roughly speaking, the following lemma says that any long enough word starting with a source decoration eventually moves to a word starting with a sink decoration when mapped by the
corresponding element of \(V\). 
We will need this for proving that \cref{decmaps} is well-defined in \cref{well-defined decoration map}.
\begin{lem}\label{sourcetosinkdec}
    Suppose that \(v\in V\) and \(Sl\) is a source decoration of \(v\). Then
    there exists \(n\in \N\) such that for all \(w\in \{0,1\}^n\) and \(s\in
    S\), there exists \(m\in \Z\), a sink decoration \(S'l'\), \(s'\in S'\) and
    \(w'\in \{0,1\}^\ast\) such that
    \[(slw)v^m=s'l'w'.\]
\end{lem}
\begin{proof}
    Let \(T=(A,B,\sigma)\) be a revealing pair for \(v\). Let \(s\in S\) be longer than all words in \(A\) (this also covers the shorter words in \(S\) as they are rigidly mapped to longer words by a power of \(v\)).
    Let \(a\) be a leaf of \(A\) with \(a<s\).
    We have one of the following
    \begin{itemize}
        \item \(a\) is a repeller,
        \item \(a\) is a domain of attraction,
        \item \(a\) is a source,
        \item \(a\) is neutral.
    \end{itemize}

Suppose that \(a\) is a domain of attraction. Then there exists \(b\in
\{0,1\}^\ast\) and \(n>0\) such that \(\act{a}{\sigma^n} =ab\). The point
\(abb\ldots\) is the only periodic point with \(a\) as a prefix (since applying a positive power of \(v\) to a word beginning with \(a\) can only yield an incomparable word or insert copies of \(b\) after \(a\)) and this
periodic point is  a sink periodic point. This is a contradiction as the
periodic point with \(s\) as a prefix is a source periodic point. Thus \(a\) is
neutral, a source or a repeller.

If \(a\) is neutral, then a word of the form \(\act{s}{v^{n}}\) for \(n\in \Z\) can only be comparable with \(s\) if \(\act{s}{v^{n}}=s\) (since \(a\) is the only leaf of \(A\) or \(B\) comparable with \(a\)). This contradicts the fact that \(s\) is a source decoration point representative, so either \(a\) is a source or \(a\) is a repeller. 

If \(a\) is a source, then the sequence \(a,\act{a}{\sigma},\act{a}{\sigma^2},\ldots\) must eventually reach a sink after potentially travelling though some neutral leaves. Each sink has a domain of attraction as a prefix, thus in this case there are no periodic points with \(a\) as a prefix.
This contradicts the fact that \(a < s\). 

Thus \(a\) is a repeller. Let \(n\in \N\) be such that \(\act{a}{\sigma^n}\) is a
strict prefix of \(a\). Suppose that \(a',b\in \{0,1\}^\ast\) are such that
\(a'=\act{a}{\sigma^n}\) and \(a'b=a\). In particular, \(a'bbb\ldots=abbb\ldots\) is a
source periodic point of \(v\) with orbit size \(n\). This is the only
periodic point of \(v\) with \(a\) as a prefix, and so \(s\) is a prefix of
\(a'bbb\ldots\).
Suppose that \(s=a'b^kb'\) where \(k\in \N\) and \(b'\) is a strict prefix of \(b\).
By the definition of a source decoration, \(sl\) is not a prefix of \(a'bbb\ldots\). Thus \(b'l\) is not a prefix of \(b\), and so
\(\act{sl}{v^{nk}}=a'b'l\) has \(a'\) as a prefix but does not have \(a\) as a
prefix. It follows that there is \(r\in \N\) such that for all \(w\in \{0,1\}^{r}\),
\(\act{sl}{v^{nk}}w=\act{slw}{v^{nk}}\) has a source as a prefix (since these words have a
range of repulsion \(a'\) as a prefix but not the repeller \(a\) as a prefix).
Thus there exists \(m\in \N\) such that \(\act{slw}{v^m}\) has a sink as a prefix. In
particular, \(\act{slw}{v^m}\) has a prefix which is a domain of attraction but
\(\act{slw}{v^m}\) is not a prefix of the sink periodic point corresponding to this
domain of attraction. The result follows.
\end{proof}
We can now finally establish that our last definition works as intended.
\begin{proposition}\label{well-defined decoration map}
    The decoration maps from \cref{decmaps} are well-defined. Moreover, if \(v\in V\), then \(d_{v^{-1}}=d_v^{-1}\).
\end{proposition}
\begin{proof}
For the first claim we need to show that given \(S_1l_1\in
\operatorname{SourceDec}_v\) and for sufficiently long \(w_1\in \{0,1\}^\ast\)
we have:
\begin{enumerate}
\item the required element \(S_2l_2\in \operatorname{SinkDec}_v\) and word \(w_2\) always exist,
 \item \(S_2\) and \(w_2\) are independent of the choice of \(s_1\in S_1\),
    \item if a power of \(v\) rigidly maps \(S_1l_1w_1\) to \(S_2l_2w_2\) and \(S_1l_1w_1'\) to \(S_2l_2w_2'\), and \(S_1l_1w_1x=S_1l_1w_1'x'\), then \(S_2l_2w_2x=S_2l_2w_2'x'\),
    \item the resulting prefix exchange map \(d_v\) is a bijection.
\end{enumerate}

Condition (1): Let \(s_1\in S_1\).
Note that if a word is not a sink decoration point representative but has a sink decoration point representative as a prefix, then it is rigidly mapped by \(v\) to another word which is not a sink decoration point representative but with a sink decoration point representative as a prefix.
 Thus, from \cref{sourcetosinkdec} we can see that all but finitely many positive powers of \(v\)  rigidly map \(s_1l_1w_1\) to a word which is not a sink decoration representative but has a sink decoration point representative as a prefix. 
 Suppose that \(k\in \N\) and \(\act{s_1l_1w_1}{v^k}\) has the sink decoration point representative \(s_2\) as a prefix (and no longer sink decoration point representative prefix).
 It follows that if \(l_2\in \{0,1\}\) and \(w_2\in \{0,1\}^\ast\) are such that \(\act{s_1l_1w_1}{v^k}=s_2l_2w_2\), then these \(l_2\) and \(w_2\) are as required where \(S_2\) is the sink decoration point containing \(s_2\). 

Condition (2): If \(s_1,s_1'\in S_1\), then there is a power of \(v\) which rigidly maps \(s_1\) to \(s_1'\), and so there are \(k,k'\in \N\) such that \(\act{s_1}{v^k}=\act{s_1'}{v^{k'}}\). 
When we take a large enough power of \(v\) we reach words which have a sink decoration point representative as a prefix and all of these are rigidly mapped to each other by powers of \(v\) by construction. 
It follows that \(S_2\) does not depend on the choice of \(s_1\). 
Moreover, as the prefix replacements only affect the sink decoration point representative prefix (once we have applied \(v\) enough times for it to appear), the suffix \(w_2\) is the same for distinct powers of \(v\).
 
Condition (3): This follows from the observation that if a power of \(v\) rigidly maps \(s_1l_1w_1\) to \(s_2l_2w_2\) then for all \(a\in \{0,1\}^\ast\), the same power rigidly maps  \(s_1l_1w_1a\) to \(s_2l_2w_2a\).

Condition (4): As we have shown conditions (1), (2), and (3), we now know that
\( d_{v} \colon \operatorname{SourceDec}_v\to \operatorname{SinkDec}_v\) is a
well-defined continuous map. As \(v\) was arbitrary,  \(
d_{v^{-1}} \colon \operatorname{SourceDec}_{v^{-1}}\to
\operatorname{SinkDec}_{v^{-1}}\) is also a well-defined continuous map. By
definition we have
\(\operatorname{SinkDec}_{v^{-1}}=\operatorname{SourceDec}_{v}\) and
\(\operatorname{SourceDec}_{v^{-1}}=\operatorname{SinkDec}_{v}\). As repeated
rigid application of \(v\) is the inverse of repeated rigid application of
\(v^{-1}\), it follows that \(d_{v^{-1}}\) is a two-sided inverse of \(d_{v}\).
In particular, both maps are bijections.
We have shown that \(d_{v^{-1}}=d_v^{-1}\); that is, we have shown (4).
\end{proof}

We have constructed our data which we will use to establish large automorphic growth.
It remains to explain what happens to this data when we apply automorphisms.
The following lemma is our main tool for understanding the action of \(B_{2,1}\) on \(V\).
\begin{lem}
    \label{trans-action}
    Suppose that \(v\in V\) and \(c\in B_{2,1}\).
    Let \(T = (Q,\pi,\lambda,q_0)\) be a minimal transducer for \(c\).
    All but finitely many \(w\in \{0,1\}^\ast\) satisfy:
    \begin{itemize}
        \item \(w\) is mapped rigidly by \(v\);
        \item \(v^c\) rigidly maps
    \(\actweird{q_0,w}{\lambda}\) to \(\actweird{q_0,\act{w}{v}}{\lambda}\).
    \end{itemize}  
\end{lem}
\begin{proof}
The first point is immediate as each element of \(V\) is defined by a finite tree pair.
Recall from  \cite[Section 2.3]{GNS} that for initial transducers \(T_1=(Q_1,\pi_1,\lambda_1,q_{0,1})\) and \(T_2=(Q_2,\pi_2,\lambda_2,q_{0,2})\), their product is the transducer
\[\comp{T_1}{T_2}=(Q_1\times Q_2, \pi_{1,2},\lambda_{1,2}, (q_{0,1},q_{0,2}))\]
where 
\[\actweird{(q_1,q_2),w}{\pi_{1,2}}=(\actweird{q_1,w}{\pi_1},\actweird{q_2,\actweird{q_1,w}{\lambda_1}}{\pi_2})\]
\[\actweird{(q_1,q_2),w}{\lambda_{1,2}}=\actweird{q_2,\actweird{q_1,w}{\lambda_1}}{\lambda_2}.\]
Moreover, if \(T_1\) and \(T_2\) are non-degenerate then the continuous map defined by \(\comp{T_1}{T_2}\) is the composite of the continuous maps defined by \(T_1\) and \(T_2\). Additionally, recall from \cite[Theorem 7.16]{collinaut} that \(V\) consists of the bisynchronising homeomorphisms whose minimal transducer has a state which defines the identity map.

    Note that, as \(v^c=\conjugate{v}{c}\), we have \(\comp{v}{c}=\comp{c}{v^c}\). Let \(U\) be the minimal transducer for \(v\) and \(W\) be the minimal transducer for \(v^c\). As \(v,v^c\in V\), both these transducers have an identity state which is reached upon reading any long enough word.
    Note that the transducers
    \(\comp{U}{T}\) and \(\comp{T}{W}\) both define the homeomorphism \(\comp{v}{c}=\comp{c}{v^c}\) of \(\mathfrak{C}\). 
    Let \(k\in \N\) be such that \(v\) maps all words of length \(k\) rigidly and such that reading any word of length \(k\) from \(q_0\) writes a word mapped rigidly by \(v^c\).

    Let \(w\in \{0,1\}^\ast\) have length at least \(k\) (all but finitely many \(w\in \{0,1\}^\ast\) have this property). 
    As \(v\) rigidly maps \(w\), reading \(w\) from the initial state of \(U\) brings us to the identity state. Thus reading \(w\) from the initial state of \(\comp{U}{T}\) results in a state which is the identity state of \(U\) paired with a state of \(T\). Let \(q_{w,1}\) be this state of \(T\). 
    Similarly, 
    reading \(w\) from the initial state of \(\comp{T}{W}\) results in a state of \(T\) paired with the identity state of \(W\). Let \(q_{w,2}\) be this state of \(T\).
    As \(\comp{v}{c}=\comp{c}{v^c}\), it follows that for all \(x\in \mathfrak{C}\), we have 
    \[p_1\actweird{q_{w,1},x}{\lambda}=\act{wx}{\comp{v}{c}}=\act{wx}{\comp{c}{v^c}}=p_2\actweird{q_{w,2},x}{\lambda},\]
    for some words \(p_1\) and \(p_2\) independent of \(x\), where the first equality comes from the transducer \(\comp{U}{T}\) for \(\comp{v}{c}\) and the third comes from the transducer \(\comp{T}{W}\) for \(\comp{c}{v^c}\).
    As \(T\) has complete response, it follows that for different choices of \(x\), the words \(\actweird{q_{w,1},x}{\lambda}\) and \(\actweird{q_{w,2},x}{\lambda}\) can begin with either a \(0\) or a \(1\). Thus we must have \(p_1=p_2\).
    We now examine these prefixes \(p_1\) and \(p_2\). 
    \[\actweird{q_{0},\act{w}{v}}{\lambda} = p_1=p_2=\act{\actweird{q_{0},w}{\lambda}}{v^c}.\]
    This equality says that \(v^c\) rigidly maps \(\actweird{q_0,w}{\lambda}\) to \(\actweird{q_0,\act{w}{v}}{\lambda}\), as required.
\end{proof}

We can now use our action on \(V\) to describe the induced action on decoration maps.
\begin{lem}\label{action_on_decoration_maps}
Suppose that \(v\in V\) and \(c\in B_{2,1}\) with transducer \(T = (Q,\pi,\lambda,q_0)\).
Suppose further that \(S_x\) and \(S_y\) are source and sink decoration points of \(v\) respectively, with \(x\) and \(y\) periodic points of \(v\) corresponding to \(S_x\) and \(S_y\) respectively.
Then the following hold:
\begin{enumerate}
    \item \(\act{x}{c}\) is a source periodic point of \(v^c\);
    \item there exists \(q_x\in Q\) such that all but finitely many elements of \(S_x\) synchronise \(T\) to the state \(q_x\).
    Moreover, there is a source decoration point \(S_{x,c}\) corresponding to the periodic point \(\act{x}{c}\) of \(v^c\) such that
    \[\makeset{\actweird{q_0,s_x}{\lambda}}{\(s_x\in S_x\)}\triangle S_{x,c}\]
    is finite (note that \(q_x\) depends not only on \(x\) but also \(S_x\));
    \item \(\act{y}{c}\) is a sink periodic point of \(v^c\);
    \item there exists \(q_y\in Q\) such that all but finitely many elements of \(S_y\) synchronise \(T\) to the state \(q_y\).
    Moreover, there is a sink decoration point \(S_{y,c}\) corresponding to the periodic point \(\act{y}{c}\) of \(v^c\) such that
    \[\makeset{\actweird{q_0,s_y}{\lambda}}{\(s_y\in S_y\)}\triangle S_{y,c}\]
    is finite;
    \item if \(l_1,l_2\in \{0,1\}\), 
    \(w_1,w_2\in \{0,1\}^\ast\) are sufficiently long, and \(d_v\) rigidly maps \(S_xl_1w_1\)
    to \(S_y l_2w_2\), then \(d_{v^c}\) rigidly maps
    \(S_{x,c}\actweird{q_x,l_1w_1}{\lambda}\) to \(S_{y,c}\actweird{q_y,l_2w_2}{\lambda}\).
\end{enumerate}

\end{lem}
\begin{proof}
\((1)\): We have \(\act{\comp{x}{\cdot}{\langle v\rangle}}{c} = \comp{\act{x}{c}}{\cdot}{\langle v^c\rangle}\).
In particular, \(\comp{x}{\cdot}{\langle v\rangle}\) is finite if and only if
\(\comp{\act{x}{c}}{\cdot}{\langle v^c\rangle}\) is finite. 
As \(x\) is a source, it has
arbitrarily small neighbourhoods which are mapped to strict supersets of
themselves by positive powers of \(v\). This is not true of other periodic
points, thus \(\act{x}{c}\) must be a source of \(v^c\).

\((2)\): Let \(k\) be the synchronising length of \(T\). By the definition of a
source, there is a largest \(z\in \mathbb{Z}\) such that \(z<0\) and \(\act{s_x}{v^{z}}\) has
\(s_x\) as a strict prefix. So let \(w \neq \varepsilon\) be such that
\(\act{s_x}{v^{z}}=s_xw\). Note that all but finitely many elements of \(S_x\) have
\(w^k\) as a suffix. Since any word with \(w^k\) as a suffix synchronises \(T\)
to some state \(q_x\), it follows that \(q_x\) with the required property exists
and is the state reached by reading the word \(w^k\) from any state.

We next need to show that \(\makeset{\actweird{q_0,s_x}{\lambda}}{\(s_x\in S_x\)}\) is `almost' a source decoration point of \(\act{x}{c}\) with respect to \(v^c\). By \cref{trans-action}, all but finitely many of the words \(s_x\in S_x\) satisfy:
\begin{itemize}
    \item \((v^z)^c\) rigidly maps \(\actweird{q_0,s_x}{\lambda}\) to \(\actweird{q_0,\act{s_x}{v^z}}{\lambda}\).
\end{itemize}
Thus all but finitely many of the words \(s_x\in S_x\) satisfy:
\begin{itemize}
    \item \((v^c)^z\) rigidly maps \(\actweird{q_0,s_x}{\lambda}\) to \(\actweird{q_0,s_xw}{\lambda}=\actweird{q_0,s_x}{\lambda}\actweird{q_{x},w}{\lambda}\).
\end{itemize}
It follows that for all but finitely many \(s_x\in S_x\), we have \(\actweird{q_0,s_x}{\lambda}\actweird{q_x,w}{\lambda}\actweird{q_x,w}{\lambda}\actweird{q_x,w}{\lambda}\ldots= \act{x}{c}\) is a fixed point of \((v^c)^z\) and \(\actweird{q_0,s_xw}{\lambda}=\actweird{q_0,s_x}{\lambda}\actweird{q_x,w}{\lambda}\) is a source decoration point representative of \(v^c\). 
The result follows as \(\{\actweird{q_0,s_x}{\lambda}, \actweird{q_0,s_x}{\lambda}\actweird{q_x,w}{\lambda},\actweird{q_0,s_x}{\lambda}\actweird{q_x,w}{\lambda}\actweird{q_x,w}{\lambda},\ldots\}\) is a cofinite subset of a source decoration point.

\((3)\) and \((4)\): Symmetric to \((1)\) and \((2)\).

\((5)\): Suppose that \(l_1,l_2\in \{0,1\}\) and \(w_1,w_2\in \{0,1\}^\ast\) are such that \(d_v\) rigidly maps \(S_xl_1w_1\) to \(S_yl_2 w_2\). From the definition of \(d_v\), for all \(s_x\in S_x\) and \(s_y\in S_y\), there is a positive power of \(v\) which rigidly maps \(s_xl_1w_1\) to \(s_yl_2w_2\).  From parts \((2)\) and \((4)\), let \(s_x\in S_x\), \(s_y\in S_y\) be such that
\(\actweird{q_0, s_x}{\lambda}\in S_{x, c}\) and \(\actweird{q_0, s_y}{\lambda}\in S_{y,c}\).
Let \(n>0\) be such that \(v^n\) rigidly maps \(s_xl_1w_1\) to \(s_yl_2w_2\). 
By \cref{trans-action}, for long enough \(w_1\) and \(w_2\), we have that \((v^c)^n\) rigidly maps \(\actweird{q_0, s_xl_1w_1}{\lambda}\) to \(\actweird{q_0, s_yl_2w_2}{\lambda}\).
It follows that \((v^c)^n\) rigidly maps \(\actweird{q_0, s_x}{\lambda}\actweird{q_x,l_1w_1}{\lambda}\) to \(\actweird{q_0, s_y}{\lambda}\actweird{q_y,l_2w_2}{\lambda}\), as required.
\end{proof}

We now understand the action of \(\actweird{V}{\Aut}\) on the data from the diagram in \cref{separatedstranddiagrams}.
To show exponential automorphic growth, we build a large family of elements and show that they are not in the same orbit. We must also track how long these elements are as words and show they have different data in the sense of \cref{separatedstranddiagrams}.
First recall that \(V\) is 2-generated, see \cite{CollinMartyn}.
\begin{lemma}\label{constructed_elements}
Let \(\{a,b\}\) be a generating set for \(V\). There exists a function \(v\mapsto v'\) from \(V\) to \(V\), \(m\in \N\), and \(x_1,x_2,x_3, y_5,y_7,y_{11}\in \mathfrak{C}\) such that:
\begin{enumerate}
\item all the points \(x_1,x_2,x_3, y_5,y_7,y_{11}\) end in an infinite sequence of \(1\)s;
    \item for all \(v\in V\), \(|v'|_{\{a,b\}}\leq m|v|_{\{a,b\}}+m\);
    \item for all \(v\in V\), \(v'\) has exactly \(29\) periodic points. This includes exactly \(6\) source periodic points in exactly \(3\) orbits 
 and exactly \(23\) sink periodic points in exactly \(3\) orbits.
    The points \(x_1, x_2, x_3\) are representatives for the \(3\) source periodic point orbits.
    The points \(y_5,y_7, y_{11}\) are representatives for the \(3\) sink periodic point orbits;
    \item for all \(v\in V\), \(|\comp{x_i}{\langle v'\rangle}|=i\) for all \(i\in \{1,2,3\}\) and \(|\comp{y_i}{\langle v'\rangle}|=i\) for all \(i\in \{5,7,11\}\);
    
    \item for all \(v_1,v_2\in V\), each periodic point orbit of \(v_1'\) has exactly one corresponding decoration point and this is also a decoration point of \(v_2'\). We denote these by \(S_{x_1}, S_{x_2}, S_{x_3}, S_{y_5}, S_{y_7}, S_{y_{11}}\);
    \item for all \(v\in V\), and for all \(z\in \mathfrak{C}\) we have
    \[\act{S_{x_1}0z}{d_{v'}}=S_{y_5}0\act{z}{a},\quad \act{S_{x_2}0z}{d_{v'}}=S_{y_7}0\act{z}{b},\quad \act{S_{x_3}0z}{d_{v'}}=S_{y_{11}}0\act{z}{v}.\]
     \item for all \(v_1,v_2\in V\), and \(z\in \{x_1,x_2,x_3,y_5,y_7,y_{11}\}\), we have \((\act{z}{v_1'^i})_{i\in \N}=(\act{z}{v_2'^i})_{i\in \N}\).
\end{enumerate}

\end{lemma}
\begin{proof}
    Let \(P\) be a fixed complete prefix code of size \(1+2+3+5+7+11=29\). Let \(P=\{p_{i,j}:i\in\{1,2,3,5,7,11\}\), \(0
    \leq j<i \}\).
    For each \(v\in V\), let \(\phi_v\) be a bijection between complete prefix codes defining \(v\). Additionally, for all \(v \in V\), 
    let \(v'\) be the element of \(V\) defined by:
    \begin{itemize}
        \item \(p_{1,0}0w\mapsto p_{5,0}0 \cdot \act{w}{\phi_a}\) for \(w\in \operatorname{dom}(\phi_a)\),
        \item  \(p_{2,1}0w\mapsto p_{7,0}0 \cdot \act{w}{\phi_b
        }\) for \(w\in \operatorname{dom}(\phi_b)\),
        \item  \(p_{3,2}0w\mapsto p_{11,0}0 \cdot \act{w}{\phi_v}\) for \(w\in \operatorname{dom}(\phi_v)\),
        \item \(p_{1,0}1\mapsto p_{1,0}\), \(p_{2,1}1\mapsto p_{2,0}\), and \(p_{3,2}1\mapsto p_{3,0}\),
        \item for  \(i\in \{1,2,3,5,7,11\}\) and \(0\leq j<i-1\) we rigidly map \(p_{i,j}\to p_{i,j+1}\),
        \item \(p_{5,4}\mapsto p_{5,0}1\),\(p_{7,6}\mapsto p_{7,0}1\), and \(p_{11,10}\mapsto p_{11,0}1\).
    \end{itemize}
    For each \(i\in \{1,2,3\}\), let \(x_i=p_{i,0}11111\ldots\), and for each \(i\in \{5,7,11\}\), let \(y_i=p_{i,0}11111\ldots\).
 For each \(v\in V\), let \(v^\#\) be the element of \(V\) which acts on \(p_{11,0}0\mathfrak{C}\) as \(v\) acts on \(\mathfrak{C}\) and fixes all points not in \(p_{11,0}0\mathfrak{C}\). 
    Let \(m\in \N\) be greater than \(|a^\#|_{\{a,b\}},|b^\#|_{\{a,b\}}\) and \(|\operatorname{1_V}'|_{\{a,b\}}\).
    We show that these choices of \(v\mapsto v'\), \(x_1,x_2,x_3,y_5,y_7,y_{11}\), and \(n\) satisfy all the required conditions. \((1)\) is immediate from the definition
    of each \(x_i\) and \(y_i\).
    
    \((2)\): All \(v\in V\) can be expressed as a product of elements of \(\{a, a^{-1}, b,
    b^{-1}\}\).
    Each expression of \(v\) as such a word defines an expression for \(v^\#\)
    as a product in \(a^\#\), \(b^\#\) and their inverses.
    Thus for all \(v\in V\), we have \[|v^\#|_{\{a,b\}}\leq |v|_{\{a,b\}}\cdot \max(|a^\#|_{\{a,b\}},|b^\#|_{\{a,b\}}).\]
    Moreover, for all \(v\in V\), we have \(v'=1_V'v^\#\). In particular,
    \[|v'|_{\{a,b\}}= |1_V'v^\#|_{\{a,b\}}\leq |1_V'|_{\{a,b\}}+|v|_{\{a,b\}}\cdot \max(|a^\#|_{\{a,b\}},|b^\#|_{\{a,b\}})\leq m|v|_{\{a,b\}}+m.\]

\((3)\) and \((4)\): Suppose that \(v\in V\) and \(x\in \mathfrak{C}\). As \(P\) is a complete prefix code, there is \(x_p\in P\) with \(x_p\leq x\). If \(x_p0\leq x\), then the construction of \(v'\) shows that there is \(x_p'\in P\) such that we can obtain points with prefixes \(x_p'10,x_p'110,x_p'1110,\ldots\) by applying positive powers of \(v'\).
Thus, if \(x_p0\leq x\), then \(x\) is not a periodic point of \(v'\).
Now suppose that \(x_p1\leq x\) and that \(x_p=p_{i,j}\) for some \(i\in \{1,2,3\}\). Directly applying \(v'\) now shows that \((v')^i\) rigidly maps \(x_p1\) to \(x_p\).
Thus \(x_p1111\ldots\) is the only periodic point with prefix \(x_p\) and this is a source periodic point with period \(i\).
Finally, suppose that \(x_p1\leq x\) and that \(x_p=p_{i,j}\) for some \(i\in \{5,7,11\}\). Directly applying \(v'\) now shows that \((v')^i\) rigidly maps \(x_p1\) to \(x_p11\).
Thus \(x_p1111\ldots\) is the only periodic point with prefix \(x_p\) and this is a sink periodic point with period \(i\).
We have now found all periodic points and the periodic points were in bijective correspondence with the prefixes in \(P\), where the first subscript gives the period and the associated point is a source precisely when the first subscript is at most \(3\). Moreover the elements \(x_1,x_2,x_3,y_5,y_7,y_{11}\) we defined earlier are representatives for the orbits. 
In particular, there are \(1+2+3=6\) source periodic points and \(5+7+11=23\) sink periodic points. 

\((5)\): Suppose that \(v\in V\) and consider the element \(v'\). 
If \(p\) is a source decoration point representative for \(x_i\in \{x_1,x_2,x_3\}\), then we must have that \((p)(v')^{-i}=p1\).
In particular, since we can add/remove one letter at a time, every source decoration point representative for \(x_i\) is mapped rigidly to \(p_{i,0}\) by some power of \(v'\).
It follows that each \(x\in \{x_1,x_2,x_3\}\) has a unique source decoration point.
Similarly, each \(y\in \{y_5,y_7,y_{11}\}\) has a unique sink decoration point.
Each of these decoration points has a single decoration, so the result follows.

\((6)\): From \((5)\), each source decoration point representative can be thought of as \(p_{i,j}11\ldots 11\) for \(i\in \{1,2,3\}\) and 
each sink decoration point representative can be thought of as \(p_{i,j}11\ldots 11\) for \(i\in \{5,7,11\}\).
In particular, \(S_{x_1}=[p_{1,0}],S_{x_2}=[p_{2,0}],S_{x_3}=[p_{3,0}],S_{y_5}=[p_{5,0}],S_{y_7}=[p_{7,0}]\), and \(S_{y_{11}}=[p_{11,0}]\).
If \(v\in V\), then by the definition of \(v'\), repeatedly applying \(v'\) eventually rigidly maps 
\[p_{1,0}11\ldots 110 z \to p_{5,0}11\ldots 110 (z)a\]
\[p_{2,0}11\ldots 110 z \to p_{7,0}11\ldots 110 (z)b\]
\[p_{3,0}11\ldots 110 z \to p_{11,0}11\ldots 110 (z)v\]
as required.
    
\((7)\): This can be seen by directly applying \(v'\) for any \(v\in V\) to these points.
\end{proof}
We can now show that these elements are indeed automorphically distinct.
\begin{proposition}\label{notautoV}
    For all \(v,w\in V\), if \(v'\) is in the same automorphic orbit as \(w'\) then \(v=w\) (using the map \(v\mapsto v'\) from \cref{constructed_elements}).
\end{proposition}
\begin{proof}
    By \cref{autV} if \(v'\) and \(w'\) are in the same automorphic orbit, then
    there is \(c\in B_{2,1}\) with \({(v')}^c={w'}\). 
    By
    \cref{action_on_decoration_maps}(1), \(\act{x_1}{c}\) is a source periodic point of
    \(w'\).
    By \cref{constructed_elements}(3), \(x_1\) is fixed by \(v'\) and so \((x_1)c\) is fixed by \({(v')}^c={w'}\).
    By \cref{constructed_elements}(3) and (4), \(x_1\) is the only fixed point
    of \(w'\), and so \(\act{x_1}{c}=x_1\). Similarly,
    \(\act{x_2}{c}\in \comp{x_2}{\langle w' \rangle}\), \(\act{x_3}{c}\in \comp{x_3}{\langle w' \rangle}\),
    \(\act{y_5}{c}\in \comp{y_5}{\langle w' \rangle}\), \(\act{y_7}{c}\in \comp{y_7}{\langle w' \rangle}\) and
    \(\act{y_{11}}{c}\in \comp{y_{11}}{\langle w' \rangle}\). 
    By Sunzi's remainder theorem,
    we can choose \(k\in \N\) such that \(\comp{c}{\cdot}{(w')^k}\) fixes each of
    \(x_1,x_2,x_3,y_5,y_7,y_{11}\). Note that \((v')^{\comp{c}{\cdot}{(w')^k}}=w'\). We redefine
    \(c\) to be \(\comp{c}{\cdot}{(w')^k}\).

    We now have \({(v')}^c=w'\) and \(c\) fixes each element of the set \(\{x_1,x_2,x_3,y_5,y_7,y_{11}\}\).
    Let \(T=(Q,\pi,\lambda,q_0)\) be the minimal transducer for \(c\) and let \(q_1\) be the unique state of \(T\) reached by reading arbitrarily long strings of \(1\)s.
    By \cref{action_on_decoration_maps}(2) there is a source decoration point \(S_{x_1,c}\) corresponding to \(\act{x_1}{c}=x_1\) such that  \[\makeset{\actweird{q_0,s}{\lambda}}{\(s\in S_{x_1}\)}\triangle S_{x_1,c}\]
    is finite. As \(\actweird{q_0,x_1}{\lambda}=x_1\), it follows from \cref{constructed_elements}(5) that \(S_{x_1,c}=S_{x_1}\). 
    Similarly, using the notation of \cref{action_on_decoration_maps}, we have \(S_{x_2,c}=S_{x_2}, S_{x_3,c}=S_{x_3}, S_{y_5,c}=S_{y_5},S_{y_7,c}=S_{y_7},\) and \(S_{y_{11},c}=S_{y_{11}}\).
    By \cref{action_on_decoration_maps}(5), for all \(i,j\in \{1,2,3,5,7,11\}\) and long enough \(w_1,w_2\in \{0,1\}^\ast\) such that \(d_{v'}\) rigidly maps \(S_{x_i}0w_1\) to \(S_{y_j}0w_2\), we have \(d_{w'}\) rigidly maps  \(S_{x_i}\lambda(q_1,0w_1)\) to \(S_{y_j}\lambda(q_1,0w_2)\). 
    By \cref{constructed_elements}(6), if \(z\in \mathfrak{C}\), then \(\act{S_{x_1} 0z}{d_{v'}}=S_{y_5}0\act{z}{a}\) and \(\act{S_{x_2} 0z}{d_{v'}}=S_{y_7}0\act{z}{b}\). 
    Together, these imply that for all \(z\in \mathfrak{C}\), we have
   \[(S_{x_1}\lambda(q_1,0z))d_{w'}=S_{y_5}\lambda(q_1,0(z)a) \text{ and } (S_{x_2}\lambda(q_1,0z))d_{w'}=S_{y_7}\lambda(q_1,0(z)b).\]
   Let \(p\coloneqq \lambda(q_1,0)\) and let \(c'\) be the rational map of the state \(\pi(q_1,0)\) of \(T\).  In particular, the above equalities can now be expressed as
   \[(S_{x_1}p\cdot (z)c')d_{w'}=S_{y_5}p\cdot ((z)a)c' \text{ and } (S_{x_2}p\cdot (z)c')d_{w'}=S_{y_7}p\cdot ((z)b)c'\]
   for all \(z\in \mathfrak{C}\).
   As the only decoration of the decoration point \(S_{x_1}\) is \(S_{x_1}0\), and \(S_{x_1}p\) is a prefix of an element of the domain of \(d_{w'}\) (which is precisely the cones corresponding to decorations), we must have \(p=0p'\) for some \(p'\in \{0,1\}^\ast\).
    As such, we can apply \cref{constructed_elements}(6) to \(w'\) in order to rewrite the above equalities as
    \[S_{y_5}0(p'\cdot (z)c')a=S_{y_5}0p'\cdot ((z)a)c' \text{ and } S_{y_7}0(p'\cdot (z)c')b=S_{y_7}0p'\cdot ((z)b)c'\]
     for all \(z\in \mathfrak{C}\). Let \(t\colon \mathfrak{C}\to \mathfrak{C}\) be the continuous map which prepends the prefix \(p'\) to every infinite word.
     The displayed equalities imply that \(c'ta=ac't\) and \(c'tb=bc't\). These imply that the continuous map \(c't\) commutes with both \(a\) and \(b\). As \(\langle a,b\rangle =V\), it follows that \(c't\) is the identity map. In particular, \(t\) is surjective, implying that \(p'=\varepsilon\) and \(t\) is the identity map.
     Thus \(c'\) is also the identity map. 
     As \(\lambda(q_1,0)=p=0p'=0\), it follows that 
     \[\lambda(q_1,0z)=0z\] 
     for all \(z\in \mathfrak{C}\).
     Recall that for all sufficiently long \(w_1,w_2\in \{0,1\}^\ast\) such that \(d_{v'}\) rigidly maps \(S_{x_3}0w_1\) to \(S_{y_{11}}0w_2\), we have \(d_{w'}\) rigidly maps \(S_{x_3}\actweird{q_1,0w_1}{\lambda}\) to \(S_{y_{11}}\actweird{q_1,0w_2}{\lambda}\).
    Thus for all sufficiently long \(w_1,w_2\in \{0,1\}^\ast\) such that \(d_{v'}\) rigidly maps \(S_{x_3}0w_1\) to \(S_{y_{11}}0w_2\), we have  \(d_{w'}\) rigidly maps \(S_{x_3}0w_1\) to \(S_{y_{11}}0w_2\).
    Hence \(d_{v'}|_{S_{x_3}0\mathfrak{C}}=d_{w'}|_{S_{x_3}0\mathfrak{C}}\). It follows from the third equality in \cref{constructed_elements}(6), that \(v=w\) as required.
\end{proof}
We can now bring everything together and show our main result.
\begin{theorem}
    \label{thm:V}
    The automorphic growth of Thompson's group \(V\) is exponential.
\end{theorem}
\begin{proof}
    Recall that the standard growth of \(V\) is exponential (note that \(V\) contains free groups for example). 
    Let \(\actweird{n}{\beta}\) denote the number of elements in \(V\) of length at most \(n\).
     Let \(\actweird{n}{\alpha}\) denote the number of automorphic orbits of elements of \(V\) of length at most \(n\).
    From \cref{constructed_elements}(2) and \cref{notautoV}, it follows that there is a constant \(m\in \N\setminus \{0\}\) such that
    \[\actweird{mn + m}{\alpha}\geq \actweird{n}{\beta}\]
    for all \(n \in \N\). 
    Thus \(\beta \preccurlyeq \alpha\). As the growth of \(V\) is exponential, the result follows.
\end{proof}

As the conjugacy is bounded below by the automorphic growth we obtain the following immediately.

\begin{cor}
    The conjugacy growth of Thompson's group \(V\) is exponential.
\end{cor}

Unlike Thompson's group \(V\), computing the automorphic growth of Thompson's group \(T\)
is fairly straightforward, in part, due to the finite outer automorphism group.
\begin{theorem}\label{Texpconj}
  Thompson's group \(T\) has exponential conjugacy growth.
\end{theorem}
\begin{proof}
Recall that \(T\leq V\) is the subgroup preserving the cyclic order on the Cantor space induced the lexicographic order on infinite words. 
   Consider the elements \(y_0,y_1\in V\) defined by
   \begin{enumerate}
       \item \(y_0\): \(0\to 00\), \(10\mapsto 01\), \(11\mapsto 1\),
       \item \(y_1\): \(0\to 01\), \(10\mapsto 1\), \(11\mapsto 00\).
   \end{enumerate}
Note that \(y_0,y_1\in T\). We consider word lengths in \(T\) with respect to a fixed finite generating set for \(T\) which contains \(y_0\) and \(y_1\).

   For a word \(w\in \{0,1\}^\ast\), let \(y_w\) be the corresponding product in the elements \(y_{0}\) and \(y_1\) in \ifthenelse{\boolean{rightactions}}{reverse }{}order. 
   For example \(y_{011}:=\comp{y_1}{y_1}{y_0}\). 
   Note that \(|y_w|\leq |w|\) and \(y_w\) rigidly maps \(0\) to \(0w\). 
   For each \(w\in \{0,1\}^*\setminus \{\varepsilon\}\), define \(x_w:=0www\ldots\).
   For all \(w\in \{0,1\}^*\setminus \{\varepsilon\}\), any long enough prefix \(p\) of \(x_w\) is rigidly mapped by \(y_w\) to \(p\tilde{w}\) where \(\tilde{w}\) is a cyclic permutation of \(w\). As \(y_0\) and \(y_1\) do not increase the length of any word beginning with a \(1\), it follows that \(w\) is the unique longest word (up to cyclic permutation)  satisfying the following condition:
   \begin{itemize}
       \item there is an infinite word \(x\) such that \(y_w\) rigidly maps any long enough prefix \(p\leq x\) to the word \(p\tilde{w}\), where \(\tilde{w}\) is a cyclic permutation of \(w\).
   \end{itemize}
   Now let \(t\in T\) and  \(w\in \{0,1\}^*\setminus \{\varepsilon\}\) be arbitrary.
   If \(x\) is as above and \(p\) is any long enough prefix of \(\act{x}{t}\), then 
\[(p)t^{-1}y_wt=(((p)t^{-1})y_w)t=(((p)t^{-1})\tilde{w})t=p\tilde{w}\]
for some cyclic permutation \(\tilde{w}\) of \(w\).
Thus \(w\) is the unique maximum length word (up to cyclic permutation) such that there is an infinite word \(x\) such that \(t^{-1}y_wt\) rigidly maps any long enough prefix \(p\leq x\) to the word \(p\tilde{w}\), where \(\tilde{w}\) is a cyclic permutation of \(w\).

In particular,  the elements \(y_w\) and \(y_{w'}\) are never conjugate when \(w\) and \(w'\) do not differ by a cyclic permutation.
 It follows that the growth of the set \(\makeset{y_w}{\(w\in \{0,1\}^\ast\)}\) up to conjugacy is at least \(\frac{2^n}{n}\), which is exponential.
\end{proof}

We can now leverage the finite outer automorphism group of \(T\) to conclude that the
automorphic growth must also be exponential.

\begin{cor}
    \label{Tautgrowth}
    The automorphic growth of \(T\) is exponential. 
\end{cor}
\begin{proof}
    By \cite{autft}*{Theorem 1}, \(\Out(T)\) is finite, and thus
    \cref{finiteOut} tells us that the conjugacy growth and automorphic growth
    rates coincide. The result this now follows from \cref{Texpconj}.
\end{proof}

\section{Further questions}
Since all of the groups we have considered have either exponential or polynomial
automorphic growth rates, it is natural to ask whether intermediate growth can occur,
as it does in both standard and conjugacy growth. The Grigorchuk group provides an
example of both \cites{Fink, Grigorchuk}, and thus is a reasonable candidate to consider.
\begin{question}
    Is there a finitely generated group with intermediate automorphic growth?
\end{question}

Gromov's Theorem \cite{bass, Gromov} tells us that all groups with a polynomially
bounded standard growth rate have a polynomial growth rate, forbidding growth rates
such as \(n \log n\). For conjugacy growth, this is not the case; the Heisenberg
group has a growth rate of \(n^2 \log n\). In the case of automorphic growth, the
Heisenberg group is no longer such an example; it has quadratic growth
(\cref{thm:Heis}). It thus remains open whether automorphic growth does indeed
exhibit this behaviour. 
\begin{question}
    Are there any groups with polynomially bounded, but not polynomial automorphic
    growth rates? 
\end{question}

From \cref{allpossibleVA}, we know that there are virtually abelian groups with polynomial automorphic growth rates of all possible degrees.
From \cref{VAmain1}, we know exactly when each virtually \(\Z^2\) group has which degree automorphic growth rate.
The description of when this happens is based on the inner automorphism group of the group in question. 
As such, there may be a more complex classification in terms of inner automorphisms for higher rank virtually abelian groups.
\begin{question}
    Can the automorphic growth be fully classified for all virtually abelian groups in
    a similar manner to \cref{VAmain1}?
\end{question}

We have shown that many virtually abelian groups have linear automorphic growth.
However, when we move to class 2 nilpotent groups, we see that the Heisenberg
group has quadratic automorphic growth. We do not know if this `collapsing' to
linear growth that was exhibited in certain virtually abelian examples can occur
in other settings. Hull and Osin's example of a group with bounded conjugacy
(and hence automorphic) growth \cite{HullOsin} is relevant here, and probably
with a direct product with \(\Z\) could produce a counterexample, but it
certainly would not be finitely presented.

\begin{question}
    Are there finitely presented non-virtually abelian groups with linear automorphic growth?
\end{question}

If a finitely generated group has exponential automorphic growth, it
automatically has exponential conjugacy growth as well. Conversely, all the
examples in this document with exponential conjugacy growth
also had exponential automorphic growth. 
\begin{question}
    Is there a finitely generated group with exponential conjugacy growth but polynomial automorphic growth?
\end{question}

There are, in fact, various conjectures suggesting that under certain
conditions, exponential conjugacy and standard growths are equivalent
\cite{GubaSapir}. Whilst not true in general \cite{HullOsin}, a finitely
presented example has yet to be found. Whilst we doubt that the class of
finitely presented groups is strong enough for this equivalence to occur, it
remains nonetheless, a class of groups that forbids certain pariah behaviour
as was used in \cite{HullOsin}, and thus any counterexample here would be
interesting.
\begin{question}
Is there a finitely presented group with exponential standard growth, but polynomial
automorphic growth?
\end{question}
We have shown that Thompson's groups \(T\) and \(V\) have exponential
automorphic growth, using the known descriptions of \(\actweird{T}{\Aut}\) and
\(\actweird{V}{\Aut}\). The groups \(T\) and \(V\) were first introduced
alongside the group \(F\), which has been shown to have exponential conjugacy
growth \cite{GubaSapir}. Moreover, the \(\actweird{F}{\Aut}\) has also received
attention in the literature \cite{autft}.
\begin{question}
    What is the automorphic growth of Thompson's group \(F\)? 
\end{question}

\section*{Acknowledgements}
All authors were partially supported by the Heilbronn Institute for Mathematical
Research. The third named author was also partially supported by the EPSRC Fellowship grant
EP/V032003/1 ‘Algorithmic, topological and geometric aspects of infinite groups,
monoids and inverse semigroups’. The authors would like to thank Corentin Bodart for some helpful comments.

\bibliography{references}

@book {Meier,
    AUTHOR = {Meier, J.},
     TITLE = {Groups, graphs and trees},
    SERIES = {London Mathematical Society Student Texts},
    VOLUME = {73},
      NOTE = {An introduction to the geometry of infinite groups},
 PUBLISHER = {Cambridge University Press, Cambridge},
      YEAR = {2008},
     PAGES = {xii+231},
      ISBN = {978-0-521-71977-3},
   MRCLASS = {20F65 (05-01 05C25 20-01)},
  MRNUMBER = {2498449},
MRREVIEWER = {Dmytro M. Savchuk},
       DOI = {10.1017/CBO9781139167505},
       URL = {https://doi.org/10.1017/CBO9781139167505},
}

@article {CollinMartyn,
    AUTHOR = {Bleak, Collin and Quick, Martyn},
     TITLE = {The infinite simple group {$V$} of {R}ichard {J}. {T}hompson:
              presentations by permutations},
   JOURNAL = {Groups Geom. Dyn.},
  FJOURNAL = {Groups, Geometry, and Dynamics},
    VOLUME = {11},
      YEAR = {2017},
    NUMBER = {4},
     PAGES = {1401--1436},
      ISSN = {1661-7207,1661-7215},
   MRCLASS = {20F05 (20E32 20F65)},
  MRNUMBER = {3737287},
MRREVIEWER = {Silvana\ Rinauro},
       DOI = {10.4171/GGD/433},
       URL = {https://doi-org.uea.idm.oclc.org/10.4171/GGD/433},
}

@article {MR4982576,
    AUTHOR = {Levine, Alex},
     TITLE = {Quadratic {D}iophantine equations, the {H}eisenberg group and
              formal languages},
   JOURNAL = {Israel J. Math.},
  FJOURNAL = {Israel Journal of Mathematics},
    VOLUME = {269},
      YEAR = {2025},
    NUMBER = {1},
     PAGES = {337--383},
      ISSN = {0021-2172,1565-8511},
   MRCLASS = {11U05 (11D09 20F10 20F18 20F65 68Q45)},
  MRNUMBER = {4982576},
       DOI = {10.1007/s11856-025-2746-x},
       URL = {https://doi-org.uea.idm.oclc.org/10.1007/s11856-025-2746-x},
}

@article{autfreehard,
 author = {Vogtmann, Karen},
 title = {Automorphisms of free groups and outer space},
 fjournal = {Geometriae Dedicata},
 journal = {Geom. Dedicata},
 issn = {0046-5755},
 volume = {94},
 pages = {1--31},
 year = {2002},
 language = {English},
 doi = {10.1023/A:1020973910646},
 zbMATH = {1864092},
 Zbl = {1017.20035}
}

@article{finitepresautfree,
    AUTHOR = {McCool, J.},
     TITLE = {A presentation for the automorphism group of a free group of
              finite rank},
   JOURNAL = {J. London Math. Soc. (2)},
  FJOURNAL = {Journal of the London Mathematical Society. Second Series},
    VOLUME = {8},
      YEAR = {1974},
     PAGES = {259--266},
      ISSN = {0024-6107,1469-7750},
   MRCLASS = {20E05},
  MRNUMBER = {340421},
MRREVIEWER = {B.\ H.\ Neumann},
       DOI = {10.1112/jlms/s2-8.2.259},
       URL = {https://doi-org.uea.idm.oclc.org/10.1112/jlms/s2-8.2.259},
}

@article{zbMATH06309498,
 author = {Belk, James and Matucci, Francesco},
 title = {Conjugacy and dynamics in {Thompson}'s groups.},
 fjournal = {Geometriae Dedicata},
 journal = {Geom. Dedicata},
 issn = {0046-5755},
 volume = {169},
 pages = {239--261},
 year = {2014},
 language = {English},
 doi = {10.1007/s10711-013-9853-2},
 zbMATH = {6309498},
 Zbl = {1321.20038}
}

@article{elliott2023description,
 author = {Elliott, L.},
 title = {A description of {{\(\operatorname{Aut} (d {V_n})\)}} and {{\(\operatorname{Out} (d {V_n})\)}} using transducers},
 fjournal = {Groups, Geometry, and Dynamics},
 journal = {Groups Geom. Dyn.},
 issn = {1661-7207},
 volume = {17},
 number = {1},
 pages = {293--312},
 year = {2023},
 language = {English},
 doi = {10.4171/GGD/697},
 zbMATH = {7672086},
 Zbl = {1522.20117}
}

@book {HoffmanKunze,
    AUTHOR = {Hoffman, K. and Kunze, R.},
     TITLE = {Linear algebra},
   EDITION = {Second},
 PUBLISHER = {Prentice-Hall, Inc., Englewood Cliffs, NJ},
      YEAR = {1971},
     PAGES = {viii+407},
   MRCLASS = {15.00},
  MRNUMBER = {276251},
}

@book {IsaacsFinGpTheory,
    AUTHOR = {Isaacs, I. Martin},
     TITLE = {Finite group theory},
    SERIES = {Graduate Studies in Mathematics},
    VOLUME = {92},
 PUBLISHER = {American Mathematical Society, Providence, RI},
      YEAR = {2008},
     PAGES = {xii+350},
      ISBN = {978-0-8218-4344-4},
   MRCLASS = {20Dxx (20-01)},
  MRNUMBER = {2426855},
MRREVIEWER = {Ronald\ Solomon},
       DOI = {10.1090/gsm/092},
       URL = {https://doi.org/10.1090/gsm/092},
}

@book {CGT,
    AUTHOR = {Lyndon, R. and Schupp, P.},
     TITLE = {Combinatorial group theory},
    SERIES = {Classics in Mathematics},
      NOTE = {Reprint of the 1977 edition},
 PUBLISHER = {Springer-Verlag, Berlin},
      YEAR = {2001},
     PAGES = {xiv+339},
      ISBN = {3-540-41158-5},
   MRCLASS = {20Fxx (20Exx 57M07)},
  MRNUMBER = {1812024},
       DOI = {10.1007/978-3-642-61896-3},
       URL = {https://doi.org/10.1007/978-3-642-61896-3},
}

@article {whitehead_algm,
    AUTHOR = {Whitehead, J. H. C.},
     TITLE = {On equivalent sets of elements in a free group},
   JOURNAL = {Ann. of Math. (2)},
  FJOURNAL = {Annals of Mathematics. Second Series},
    VOLUME = {37},
      YEAR = {1936},
    NUMBER = {4},
     PAGES = {782--800},
      ISSN = {0003-486X},
   MRCLASS = {DML},
  MRNUMBER = {1503309},
       DOI = {10.2307/1968618},
       URL = {https://doi.org/10.2307/1968618},
}

@article {Bleak2013centralizers,
    AUTHOR = {Bleak, Collin and Bowman, Hannah and Gordon Lynch, Alison and
              Graham, Garrett and Hughes, Jacob and Matucci, Francesco and
              Sapir, Eugenia},
     TITLE = {Centralizers in the {R}. {T}hompson group {$V_n$}},
   JOURNAL = {Groups Geom. Dyn.},
  FJOURNAL = {Groups, Geometry, and Dynamics},
    VOLUME = {7},
      YEAR = {2013},
    NUMBER = {4},
     PAGES = {821--865},
      ISSN = {1661-7207,1661-7215},
   MRCLASS = {20F65 (20E07 37C85)},
  MRNUMBER = {3134027},
MRREVIEWER = {Andrzej\ W.\ Bi\'s},
       DOI = {10.4171/GGD/207},
       URL = {https://doi-org.uea.idm.oclc.org/10.4171/GGD/207},
}

@article {Brin2004higher,
    AUTHOR = {Brin, Matthew G.},
     TITLE = {Higher dimensional {T}hompson groups},
   JOURNAL = {Geom. Dedicata},
  FJOURNAL = {Geometriae Dedicata},
    VOLUME = {108},
      YEAR = {2004},
     PAGES = {163--192},
      ISSN = {0046-5755,1572-9168},
   MRCLASS = {20B27 (20E32 37B99 57M07)},
  MRNUMBER = {2112673},
MRREVIEWER = {Wolfgang\ Knapp},
       DOI = {10.1007/s10711-004-8122-9},
       URL = {https://doi-org.uea.idm.oclc.org/10.1007/s10711-004-8122-9},
}

@article {GubaSapir,
    AUTHOR = {Guba, Victor and Sapir, Mark},
     TITLE = {On the conjugacy growth functions of groups},
   JOURNAL = {Illinois J. Math.},
  FJOURNAL = {Illinois Journal of Mathematics},
    VOLUME = {54},
      YEAR = {2010},
    NUMBER = {1},
     PAGES = {301--313},
      ISSN = {0019-2082,1945-6581},
   MRCLASS = {20F65 (20E45 20F69)},
  MRNUMBER = {2776997},
MRREVIEWER = {Michel\ Coornaert},
       URL = {http://projecteuclid.org.uea.idm.oclc.org/euclid.ijm/1299679750},
}

@article {GekhtmanYang,
    AUTHOR = {Gekhtman, Ilya and Yang, Wen-yuan},
     TITLE = {Counting conjugacy classes in groups with contracting
              elements},
   JOURNAL = {J. Topol.},
  FJOURNAL = {Journal of Topology},
    VOLUME = {15},
      YEAR = {2022},
    NUMBER = {2},
     PAGES = {620--665},
      ISSN = {1753-8416,1753-8424},
   MRCLASS = {20F65 (20E45 20F67)},
  MRNUMBER = {4441600},
MRREVIEWER = {Alexander\ Fel\cprime shtyn},
       DOI = {10.1112/topo.12221},
       URL = {https://doi-org.uea.idm.oclc.org/10.1112/topo.12221},
}

@article {Gromov,
    AUTHOR = {Gromov, Mikhael},
     TITLE = {Groups of polynomial growth and expanding maps},
   JOURNAL = {Inst. Hautes \'Etudes Sci. Publ. Math.},
  FJOURNAL = {Institut des Hautes \'Etudes Scientifiques. Publications
              Math\'ematiques},
    NUMBER = {53},
      YEAR = {1981},
     PAGES = {53--73},
      ISSN = {0073-8301,1618-1913},
   MRCLASS = {53C20 (22E40 58F15)},
  MRNUMBER = {623534},
MRREVIEWER = {J.\ A.\ Wolf},
       URL = {http://www.numdam.org/item?id=PMIHES_1981__53__53_0},
}

@article {Babenko,
    AUTHOR = {Babenko, I. K.},
     TITLE = {Closed geodesics, asymptotic volume and the characteristics of
              growth of groups},
   JOURNAL = {Izv. Akad. Nauk SSSR Ser. Mat.},
  FJOURNAL = {Izvestiya Akademii Nauk SSSR. Seriya Matematicheskaya},
    VOLUME = {52},
      YEAR = {1988},
    NUMBER = {4},
     PAGES = {675--711, 895},
      ISSN = {0373-2436},
   MRCLASS = {58F17 (22E40 53C22 58E10)},
  MRNUMBER = {966980},
MRREVIEWER = {Boris\ N.\ Apanasov},
       DOI = {10.1070/IM1989v033n01ABEH000811},
       URL = {https://doi-org.uea.idm.oclc.org/10.1070/IM1989v033n01ABEH000811},
}

@article {HullOsin,
    AUTHOR = {Hull, Michael and Osin, Denis},
     TITLE = {Conjugacy growth of finitely generated groups},
   JOURNAL = {Adv. Math.},
  FJOURNAL = {Advances in Mathematics},
    VOLUME = {235},
      YEAR = {2013},
     PAGES = {361--389},
      ISSN = {0001-8708,1090-2082},
   MRCLASS = {20F69 (20E45 20F65 20F67)},
  MRNUMBER = {3010062},
MRREVIEWER = {Alexander\ Fel\cprime shtyn},
       DOI = {10.1016/j.aim.2012.12.007},
       URL = {https://doi-org.uea.idm.oclc.org/10.1016/j.aim.2012.12.007},
}

@article {Blachere,
    AUTHOR = {Blach\`{e}re, S\'{e}bastien},
     TITLE = {Word distance on the discrete {H}eisenberg group},
   JOURNAL = {Colloq. Math.},
  FJOURNAL = {Colloquium Mathematicum},
    VOLUME = {95},
      YEAR = {2003},
    NUMBER = {1},
     PAGES = {21--36},
      ISSN = {0010-1354,1730-6302},
   MRCLASS = {20F65},
  MRNUMBER = {1967551},
MRREVIEWER = {Piotr\ Haj\l asz},
       DOI = {10.4064/cm95-1-2},
       URL = {https://doi-org.uea.idm.oclc.org/10.4064/cm95-1-2},
}

@incollection{GNS,
 author = {Grigorchuk, R. I. and Nekrashevich, V. V. and Sushchanskii, V. I.},
 title = {Automata, dynamical systems, and groups},
 booktitle = {Dynamical systems, automata, and infinite groups. Transl. from the Russian},
 pages = {128--203},
 year = {2000},
 publisher = {Moscow: MAIK Nauka/Interperiodica Publishing},
 language = {English},
 zbMATH = {1729301},
 Zbl = {1155.37311}
}

@article {Grigorchuk,
    AUTHOR = {Grigorchuk, R. I.},
     TITLE = {Degrees of growth of finitely generated groups and the theory
              of invariant means},
   JOURNAL = {Izv. Akad. Nauk SSSR Ser. Mat.},
  FJOURNAL = {Izvestiya Akademii Nauk SSSR. Seriya Matematicheskaya},
    VOLUME = {48},
      YEAR = {1984},
    NUMBER = {5},
     PAGES = {939--985},
      ISSN = {0373-2436},
   MRCLASS = {20F05 (43A07)},
  MRNUMBER = {764305},
 MRREVIEWER = {P.\ Gerl},
}

@book{collinaut,
 author = {Bleak, C. and Cameron, P. and Maissel, Y. and Navas, A. and Olukoya, F.},
 title = {The further chameleon groups of {Richard} {Thompson} and {Graham} {Higman}: {Automorphisms} via dynamics for the {Higman}-{Thompson} groups~{{\(G_{n,r}\)}}},
 fseries = {Memoirs of the American Mathematical Society},
 series = {Mem. Am. Math. Soc.},
 issn = {0065-9266},
 volume = {1510},
 isbn = {978-1-4704-7145-3; 978-1-4704-7954-1},
 year = {2024},
 publisher = {Providence, RI: American Mathematical Society (AMS)},
 language = {English},
 doi = {10.1090/memo/1510},
 zbMATH = {7927090},
 Zbl = {1552.20003}
}

@article {Osin,
    AUTHOR = {Osin, D.},
     TITLE = {Small cancellations over relatively hyperbolic groups and
              embedding theorems},
   JOURNAL = {Ann. of Math. (2)},
  FJOURNAL = {Annals of Mathematics. Second Series},
    VOLUME = {172},
      YEAR = {2010},
    NUMBER = {1},
     PAGES = {1--39},
      ISSN = {0003-486X},
   MRCLASS = {20F67 (20F06)},
  MRNUMBER = {2680416},
MRREVIEWER = {Richard Weidmann},
       DOI = {10.4007/annals.2010.172.1},
       URL = {https://doi.org/10.4007/annals.2010.172.1},
}

@article {Fink,
    AUTHOR = {Fink, Elisabeth},
     TITLE = {Conjugacy growth and width of certain branch groups},
   JOURNAL = {Internat. J. Algebra Comput.},
  FJOURNAL = {International Journal of Algebra and Computation},
    VOLUME = {24},
      YEAR = {2014},
    NUMBER = {8},
     PAGES = {1213--1231},
      ISSN = {0218-1967,1793-6500},
   MRCLASS = {20F65 (20F69)},
  MRNUMBER = {3296364},
MRREVIEWER = {Mustafa\ G\"okhan\ Benli},
       DOI = {10.1142/S0218196714500544},
       URL = {https://doi-org.uea.idm.oclc.org/10.1142/S0218196714500544},
}

@article {HuaReiner,
    AUTHOR = {Hua, L. K. and Reiner, I.},
     TITLE = {Automorphisms of the unimodular group},
   JOURNAL = {Trans. Amer. Math. Soc.},
  FJOURNAL = {Transactions of the American Mathematical Society},
    VOLUME = {71},
      YEAR = {1951},
     PAGES = {331--348},
      ISSN = {0002-9947,1088-6850},
   MRCLASS = {10.0X},
  MRNUMBER = {43847},
MRREVIEWER = {E.\ Grosswald},
       DOI = {10.2307/1990696},
       URL = {https://doi-org.uea.idm.oclc.org/10.2307/1990696},
}

@article {Eick,
    AUTHOR = {Eick, B.},
     TITLE = {The automorphism group of a finitely generated virtually
              abelian group},
   JOURNAL = {Groups Complex. Cryptol.},
  FJOURNAL = {Groups. Complexity. Cryptology},
    VOLUME = {8},
      YEAR = {2016},
    NUMBER = {1},
     PAGES = {35--45},
      ISSN = {1867-1144,1869-6104},
   MRCLASS = {20F28 (20H15)},
  MRNUMBER = {3498299},
MRREVIEWER = {Peeter\ Puusemp},
       DOI = {10.1515/gcc-2016-0007},
       URL = {https://doi.org/10.1515/gcc-2016-0007},
}

@article {ConjLinear,
    AUTHOR = {Breuillard, Emmanuel and Cornulier, Yves and Lubotzky,
              Alexander and Meiri, Chen},
     TITLE = {On conjugacy growth of linear groups},
   JOURNAL = {Math. Proc. Cambridge Philos. Soc.},
  FJOURNAL = {Mathematical Proceedings of the Cambridge Philosophical
              Society},
    VOLUME = {154},
      YEAR = {2013},
    NUMBER = {2},
     PAGES = {261--277},
      ISSN = {0305-0041,1469-8064},
   MRCLASS = {20H20 (20E45)},
  MRNUMBER = {3021813},
MRREVIEWER = {Timothy\ C.\ Burness},
       DOI = {10.1017/S030500411200059X},
       URL = {https://doi-org.uea.idm.oclc.org/10.1017/S030500411200059X},
}

@article {BreuillarddeCornulier,
    AUTHOR = {Breuillard, Emmanuel and de Cornulier, Yves},
     TITLE = {On conjugacy growth for solvable groups},
   JOURNAL = {Illinois J. Math.},
  FJOURNAL = {Illinois Journal of Mathematics},
    VOLUME = {54},
      YEAR = {2010},
    NUMBER = {1},
     PAGES = {389--395},
      ISSN = {0019-2082,1945-6581},
   MRCLASS = {20F16 (20E45 20F69 20G25)},
  MRNUMBER = {2777001},
MRREVIEWER = {Fausto\ De Mari},
       URL = {http://projecteuclid.org.uea.idm.oclc.org/euclid.ijm/1299679754},
}

@article {Greenfeld,
    AUTHOR = {Greenfeld, B.},
     TITLE = {Conjugacy growth of free nilpotent groups},
   JOURNAL = {Internat. J. Algebra Comput.},
  FJOURNAL = {International Journal of Algebra and Computation},
    VOLUME = {34},
      YEAR = {2024},
    NUMBER = {6},
     PAGES = {881--886},
      ISSN = {0218-1967,1793-6500},
   MRCLASS = {99-06},
  MRNUMBER = {4801961},
       DOI = {10.1142/S0218196724500334},
       URL = {https://doi.org/10.1142/S0218196724500334},
}

@article {HeisenbergConjugacy,
    AUTHOR = {Evetts, A.},
     TITLE = {Conjugacy growth in the higher {H}eisenberg groups},
   JOURNAL = {Glasg. Math. J.},
  FJOURNAL = {Glasgow Mathematical Journal},
    VOLUME = {65},
      YEAR = {2023},
    NUMBER = {S1},
     PAGES = {S148--S169},
      ISSN = {0017-0895,1469-509X},
   MRCLASS = {20F65 (20E45 20F18)},
  MRNUMBER = {4594272},
MRREVIEWER = {John\ M.\ Mackay},
       DOI = {10.1017/S0017089522000428},
       URL = {https://doi.org/10.1017/S0017089522000428},
}

@article {CiobanuEvettsHo,
    AUTHOR = {Ciobanu, Laura and Evetts, Alex and Ho, Meng-Che},
     TITLE = {The conjugacy growth of the soluble {B}aumslag-{S}olitar
              groups},
   JOURNAL = {New York J. Math.},
  FJOURNAL = {New York Journal of Mathematics},
    VOLUME = {26},
      YEAR = {2020},
     PAGES = {473--495},
      ISSN = {1076-9803},
   MRCLASS = {20F65 (05E16 20E45 20F69)},
  MRNUMBER = {4102998},
}

@book {Mann,
    AUTHOR = {Mann, A.},
     TITLE = {How groups grow},
    SERIES = {London Mathematical Society Lecture Note Series},
    VOLUME = {395},
 PUBLISHER = {Cambridge University Press, Cambridge},
      YEAR = {2012},
     PAGES = {x+199},
      ISBN = {978-1-107-65750-2},
   MRCLASS = {20F69 (20F05)},
  MRNUMBER = {2894945},
MRREVIEWER = {Wolfgang\ Woess},
}

@article {Lee1,
    AUTHOR = {Lee, Donghi},
     TITLE = {Counting words of minimum length in an automorphic orbit},
   JOURNAL = {J. Algebra},
  FJOURNAL = {Journal of Algebra},
    VOLUME = {301},
      YEAR = {2006},
    NUMBER = {1},
     PAGES = {35--58},
      ISSN = {0021-8693,1090-266X},
   MRCLASS = {20E05 (20F28)},
  MRNUMBER = {2230319},
MRREVIEWER = {Evgeny\ I.\ Timoshenko},
       DOI = {10.1016/j.jalgebra.2006.04.012},
       URL = {https://doi-org.uea.idm.oclc.org/10.1016/j.jalgebra.2006.04.012},
}

@article {Lee2,
    AUTHOR = {Lee, Donghi},
     TITLE = {A tighter bound for the number of words of minimum length in
              an automorphic orbit},
   JOURNAL = {J. Algebra},
  FJOURNAL = {Journal of Algebra},
    VOLUME = {305},
      YEAR = {2006},
    NUMBER = {2},
     PAGES = {1093--1101},
      ISSN = {0021-8693,1090-266X},
   MRCLASS = {20E05 (20F10)},
  MRNUMBER = {2266870},
MRREVIEWER = {Armando\ Martino},
       DOI = {10.1016/j.jalgebra.2006.03.038},
       URL = {https://doi-org.uea.idm.oclc.org/10.1016/j.jalgebra.2006.03.038},
}

@article {Coulon22,
    AUTHOR = {Coulon, R\'emi},
     TITLE = {Examples of groups whose automorphisms have exotic growth},
   JOURNAL = {Algebr. Geom. Topol.},
  FJOURNAL = {Algebraic \& Geometric Topology},
    VOLUME = {22},
      YEAR = {2022},
    NUMBER = {4},
     PAGES = {1497--1510},
      ISSN = {1472-2747,1472-2739},
   MRCLASS = {20E36 (20F06 20F28 20F65 20F67 20F69)},
  MRNUMBER = {4495664},
MRREVIEWER = {Alexander\ Fel\cprime shtyn},
       DOI = {10.2140/agt.2022.22.1497},
       URL = {https://doi-org.uea.idm.oclc.org/10.2140/agt.2022.22.1497},
}

@article {Levitt09,
    AUTHOR = {Levitt, Gilbert},
     TITLE = {Counting growth types of automorphisms of free groups},
   JOURNAL = {Geom. Funct. Anal.},
  FJOURNAL = {Geometric and Functional Analysis},
    VOLUME = {19},
      YEAR = {2009},
    NUMBER = {4},
     PAGES = {1119--1146},
      ISSN = {1016-443X,1420-8970},
   MRCLASS = {20E36 (20E05 20E08 20F65)},
  MRNUMBER = {2570318},
MRREVIEWER = {Arnaud\ Hilion},
       DOI = {10.1007/s00039-009-0016-4},
       URL = {https://doi-org.uea.idm.oclc.org/10.1007/s00039-009-0016-4},
}

@article{Fioravanti26,
 author = {Fioravanti, Elia},
 title = {Growth of automorphisms of virtually special groups},
 year = {2026},
 journal = {arxiv e-prints},
 number = {arxiv:2501.12321},
}

@article{arxiv:2004.00516,
 author = {Olukoya, F.},
 title = {The core growth of strongly synchronizing transducers},
 year = {2020},
 journal = {arxiv e-prints},
 url = {https://arxiv.org/abs/2004.00516},
 number = {arxiv:2004.00516},
}

@article{zbMATH07720878,
    AUTHOR = {Olukoya, Feyishayo},
     TITLE = {Automorphisms of the generalized {T}hompson's group
              {$T_{n,r}$}},
   JOURNAL = {Trans. London Math. Soc.},
  FJOURNAL = {Transactions of the London Mathematical Society},
    VOLUME = {9},
      YEAR = {2022},
    NUMBER = {1},
     PAGES = {86--135},
      ISSN = {2052-4986},
   MRCLASS = {20E36 (20E07 20E45)},
  MRNUMBER = {4535659},
MRREVIEWER = {Conchita\ Mart\'inez-P\'erez},
       DOI = {10.1112/tlm3.12044},
       URL = {https://doi-org.uea.idm.oclc.org/10.1112/tlm3.12044},
}

@misc{bishop2026periodgrowthcocontextfreegroups,
      title={Period growth and co-context-free groups}, 
      author={Alex Bishop and Corentin Bodart and Letizia Issini and Davide Perego},
      year={2026},
      eprint={2601.13058},
      archivePrefix={arXiv},
      primaryClass={math.GR},
      url={https://arxiv.org/abs/2601.13058}, 
}

@article{arXiv:2407.18720,
 author = {Olukoya, F.},
 title = {Automorphisms of the two-sided shift and the {Higman}--{Thompson} groups {III}: extensions},
 year = {2024},
 journal = {arXiv e-prints},
 number = {arXiv:2407.18720},
}

@article {salazar2010thompson,
    AUTHOR = {Salazar-D\'iaz, Olga Patricia},
     TITLE = {Thompson's group {$V$} from a dynamical viewpoint},
   JOURNAL = {Internat. J. Algebra Comput.},
  FJOURNAL = {International Journal of Algebra and Computation},
    VOLUME = {20},
      YEAR = {2010},
    NUMBER = {1},
     PAGES = {39--70},
      ISSN = {0218-1967,1793-6500},
   MRCLASS = {37B99 (20B07 20B27 37E99)},
  MRNUMBER = {2655915},
MRREVIEWER = {Dmytro\ M.\ Savchuk},
       DOI = {10.1142/S0218196710005534},
       URL = {https://doi-org.uea.idm.oclc.org/10.1142/S0218196710005534},
}

@misc{Visafullgroup,
 author = {Katzlinger, Leonhard},
 title = {Topological {Full} {Groups}},
 year = {2019},
 howpublished = {Preprint, {arXiv}:1907.07424 [math.{GR}] (2019)},
 url = {https://arxiv.org/abs/1907.07424},
 arXiv = {arXiv:1907.07424}
}

@article {bass,
    AUTHOR = {Bass, H.},
     TITLE = {The degree of polynomial growth of finitely generated
              nilpotent groups},
   JOURNAL = {Proc. London Math. Soc. (3)},
  FJOURNAL = {Proceedings of the London Mathematical Society. Third Series},
    VOLUME = {25},
      YEAR = {1972},
     PAGES = {603--614},
      ISSN = {0024-6115,1460-244X},
   MRCLASS = {20F05},
  MRNUMBER = {379672},
MRREVIEWER = {I.\ B. S. Passi},
       DOI = {10.1112/plms/s3-25.4.603},
       URL = {https://doi-org.uea.idm.oclc.org/10.1112/plms/s3-25.4.603},
}

@article{higman1974finitely,
  title={Finitely Presented Infinite Simple Groups, Notes on Pure Math. 8, {IAS}},
  author={Higman, G},
  journal={Austral. Nat. Univ., Canberra},
  year={1974}
}

@article{skipper2021almost,
  title={Almost-automorphisms of trees, cloning systems and finiteness properties},
  author={Skipper, Rachel and Zaremsky, Matthew CB},
  journal={Journal of Topology and Analysis},
  volume={13},
  number={01},
  pages={101--146},
  year={2021},
  publisher={World Scientific}
}

@article{SalazarDaz2010ThompsonsGV, title={Thompson's Group V from a Dynamical Viewpoint}, author={Olga Patricia Salazar-D{\'i}az}, journal={Int. J. Algebra Comput.}, year={2010}, volume={20}, pages={39-70}, url={https://api.semanticscholar.org/CorpusID:3024762} }

@article{barker2016power,
  title={The power conjugacy problem in Higman--Thompson groups},
  author={Barker, Nathan and Duncan, Andrew J and Robertson, David M},
  journal={International Journal of Algebra and Computation},
  volume={26},
  number={02},
  pages={309--374},
  year={2016},
  publisher={World Scientific}
}

@article{ConnonFloydParry,
 author = {Cannon, J. W. and Floyd, W. J. and Parry, W. R.},
 title = {Introductory notes on {Richard} {Thompson}'s groups},
 fjournal = {L'Enseignement Math{\'e}matique. 2e S{\'e}rie},
 journal = {Enseign. Math. (2)},
 issn = {0013-8584},
 volume = {42},
 number = {3-4},
 pages = {215--256},
 year = {1996},
 language = {English},
 zbMATH = {1013724},
 Zbl = {0880.20027}
}

@article{autft,
 author = {Brin, Matthew G.},
 title = {The {Chameleon} groups of {Richard} {J}. {Thompson}: {Automorphisms} and dynamics},
 fjournal = {Publications Math{\'e}matiques},
 journal = {Publ. Math., Inst. Hautes {\'E}tud. Sci.},
 issn = {0073-8301},
 volume = {84},
 pages = {5--33},
 year = {1996},
 language = {English},
 doi = {10.1007/BF02698834},
 url = {https://eudml.org/doc/104117},
 zbMATH = {1057022},
 Zbl = {0891.57037}
}

@article{belk2014conjugacy,
  title={Conjugacy and dynamics in Thompson’s groups},
  author={Belk, James and Matucci, Francesco},
  journal={Geometriae Dedicata},
  volume={169},
  number={1},
  pages={239--261},
  year={2014},
  publisher={Springer}
}

@article{rubinshort,
 author = {Belk, James and Elliott, Luna and Matucci, Francesco},
 title = {A short proof of {Rubin}'s theorem},
 fjournal = {Israel Journal of Mathematics},
 journal = {Isr. J. Math.},
 issn = {0021-2172},
 volume = {267},
 number = {1},
 pages = {157--169},
 year = {2025},
 language = {English},
 doi = {10.1007/s11856-024-2700-3},
 zbMATH = {8069401},
 Zbl = {1572.22023}
}

@article{rubin,
 author = {Rubin, Matatyahu},
 title = {On the reconstruction of topological spaces from their groups of homeomorphisms},
 fjournal = {Transactions of the American Mathematical Society},
 journal = {Trans. Am. Math. Soc.},
 issn = {0002-9947},
 volume = {312},
 number = {2},
 pages = {487--538},
 year = {1989},
 language = {English},
 doi = {10.2307/2001000},
 zbMATH = {4109619},
 Zbl = {0677.54029}
}
\bibliographystyle{alpha}
\end{document}